\documentclass[11pt]{article}

\usepackage[hidelinks]{hyperref}
\hypersetup{
  colorlinks   = true, 
  urlcolor     = blue, 
  linkcolor    = blue, 
  citecolor   = red 
}
\usepackage{amsmath,amsthm,amssymb}
\usepackage{graphicx}
\usepackage{bbm}
\newcommand{\sy}{\boldsymbol{\Psi}}
\newcommand{\py}{\boldsymbol{\Phi}}

\usepackage{xcolor}
\usepackage[margin=1in]{geometry} 
\usepackage{enumerate,mathtools,mathrsfs,bm,graphicx}
\usepackage[numbers, super]{natbib}
\usepackage[hidelinks]{hyperref}
\newcommand{\N}{\mathbb{N}}									

\newcommand{\R}{\mathbb{R}}

\newcommand{\vertiii}[1]{{\left\vert\kern-0.25ex\left\vert\kern-0.25ex\left\vert #1 
    \right\vert\kern-0.25ex\right\vert\kern-0.25ex\right\vert}}
    
\newcommand{\inner}[2]{\langle #1, #2 \rangle}
\newcommand{\floor}[1]{\left \lfloor{#1}\right \rfloor}		
\DeclarePairedDelimiter\abs{\lvert}{\rvert}					
\DeclarePairedDelimiter\norm{\lVert}{\rVert}				

\newtheorem{theorem}{Theorem}[section]
\newtheorem{corollary}{Corollary}[theorem]

\newtheorem{lemma}[theorem]{Lemma}
\newtheorem{proposition}[theorem]{Proposition}
\newtheorem*{remark}{Remark}
\newtheorem{definition}[theorem]{Definition}

\newtheorem{assumption}[theorem]{Assumption}

\begin{document}
	\title{Navier-Stokes Equations with Navier Boundary Conditions and Stochastic Lie Transport: Well-Posedness and Inviscid Limit}
	\author{Daniel Goodair}
	\date{\today} 
	\maketitle
\setcitestyle{numbers}	
\thispagestyle{empty}
\begin{abstract}
    We prove the existence and uniqueness of global, probabilistically strong, analytically strong solutions of the 2D Stochastic Navier-Stokes Equation under Navier boundary conditions. The choice of noise includes a large class of additive, multiplicative and transport models. We emphasise that with a transport type noise, the Navier boundary conditions enable direct energy estimates which appear to be prohibited for the usual no-slip condition. The importance of the Stochastic Advection by Lie Transport (SALT) structure, in comparison to a purely transport Stratonovich noise, is also highlighted in these estimates. In the particular cases of SALT noise, the free boundary condition and a domain of non-negative curvature, the inviscid limit exists and is a global, probabilistically weak, analytically weak solution of the corresponding Stochastic Euler Equation. 
\end{abstract}
	
\tableofcontents
\textcolor{white}{Hello}
\thispagestyle{empty}
\newpage

\setcounter{page}{1}
\addcontentsline{toc}{section}{Introduction}

\section*{Introduction}

The theoretical analysis of Stochastic Navier-Stokes Equations dates back to the work of Bensoussan and Temam [\cite{bensoussan1973equations}] in 1973, where the problem of existence of solutions is addressed in the presence of a random forcing term. The well-posedness question for additive and multiplicative noise has since seen significant developments, for example through the works [\cite{bensoussan1995stochastic}, \cite{breit2016stochastic}, \cite{brzezniak1992stochastic} \cite{brzezniak1999strong}, \cite{brzezniak2013existence}, \cite{capinski1991stochastic}, \cite{flandoli1995martingale}, \cite{glatt2009strong}, \cite{langa2003existence}, \cite{liu2015stochastic}, \cite{menaldi2002stochastic}] and references therein. The choice of noise to best encapsulate physical properties of this fluid equation is in constant study, recently yielding a strong argument for transport type stochastic perturbations (where the stochastic integral depends on the gradient of the solution). The paper of Brze\'{z}niak, Capinski and Flandoli [\cite{brzezniak1992stochastic}] in 1992 was one of the first to bring attention to the significance of fluid dynamics equations with transport noise, whilst such ideas have only recently been cemented through the specific stochastic transport schemes of [\cite{holm2015variational}] and [\cite{memin2014fluid}]. In these papers Holm and M\'{e}min establish a new class of stochastic equations driven by transport type noise which serve as fluid dynamics models by adding uncertainty in the transport of fluid parcels to reflect the unresolved scales. Indeed the physical significance of such equations in modelling, numerical analysis and data assimilation continues to be well documented, see [\cite{alonso2020modelling}, \cite{cotter2020data}, \cite{cotter2018modelling}, \cite{cotter2019numerically}, \cite{crisan2021theoretical}, \cite{dufee2022stochastic}, \cite{flandoli20212d}, \cite{flandoli2022additive}, \cite{holm2020stochastic}, \cite{holm2019stochastic}, \cite{lang2022pathwise}, \cite{van2021bayesian}, \cite{street2021semi}]. With this motivation our main object of study is the Navier-Stokes Equation under Stochastic Advection by Lie Transport (SALT) introduced in [\cite{holm2015variational}], given by \begin{equation} \label{number2equationSALT}
    u_t - u_0 + \int_0^t\mathcal{L}_{u_s}u_s\,ds - \nu\int_0^t \Delta u_s\, ds + \int_0^t B(u_s) \circ d\mathcal{W}_s + \nabla \rho_t= 0
\end{equation}
where $u$ represents the fluid velocity, $\rho$ the pressure\footnote{The pressure term is a semimartingale, and an explicit form for the SALT Euler Equation is given in [\cite{street2021semi}] Subsection 3.3}, $\mathcal{W}$ is a Cylindrical Brownian Motion, $\mathcal{L}$ represents the nonlinear term and $B$ is a first order differential operator (the SALT Operator) properly introduced in Subsection \ref{sub salt}. Intrinsic to this stochastic methodology is that $B$ is defined relative to a collection of functions $(\xi_i)$ which physically represent spatial correlations. These $(\xi_i)$ can be determined at coarse-grain resolutions from finely resolved numerical simulations, and mathematically are derived as eigenvectors of a velocity-velocity correlation matrix (see [\cite{cotter2020data}, \cite{cotter2018modelling}, \cite{cotter2019numerically}]). We are also interested in the cases of a more general noise, which could be of It\^{o} or Stratonovich type. That is, we consider the well-posedness of equations
\begin{equation} \label{number2equation}
    u_t - u_0 + \int_0^t\mathcal{L}_{u_s}u_s\,ds - \nu\int_0^t \Delta u_s\, ds + \int_0^t \mathcal{G}(u_s) d\mathcal{W}_s + \nabla \rho_t = 0
\end{equation}
and
\begin{equation} \label{number3equation}
    u_t - u_0 + \int_0^t\mathcal{L}_{u_s}u_s\,ds - \nu\int_0^t \Delta u_s\, ds + \int_0^t \mathcal{G}(u_s) \circ d\mathcal{W}_s + \nabla \rho_t = 0.
\end{equation} 
for a function $\mathcal{G}$ representing a variety of additive, multiplicative and transport type operators. We pose the equations in a two dimensional smooth bounded domain $\mathscr{O}$ and impose the divergence-free constraint on $u$. It remains to enforce boundary conditions, the choice of which breeds another interesting problem. The classical option is the so-called `no-slip' condition, given by $u = 0$ on the boundary $\partial \mathscr{O}$, which is well motivated in [\cite{day1990no}, \cite{richardson1973no}, \cite{ruckenstein1983no}]. Other considerations are far from redundant though, and perhaps the most appreciated of them are the Navier boundary conditions given by \begin{equation} \label{navier boundary conditions}
    u \cdot \mathbf{n} = 0, \qquad 2(Du)\mathbf{n} \cdot \mathbf{\iota} + \alpha u\cdot \mathbf{\iota} = 0
\end{equation} 
where $\mathbf{n}$ is the unit outwards normal vector, $\mathbf{\iota}$ the unit tangent vector, $Du$ is the rate of strain tensor $(Du)^{k,l}:=\frac{1}{2}\left(\partial_ku^l + \partial_lu^k\right)$ and $\alpha \in C^2(\partial \mathscr{O};\R)$ represents a friction coefficient which determines the extent to which the fluid slips on the boundary relative to the tangential stress. These conditions were first proposed by Navier in [\cite{navier1822memoire}, \cite{navier1827lois}], and have been derived in [\cite{maxwell1879vii}] from the kinetic theory of gases and in [\cite{masmoudi2003boltzmann}] as a hydrodynamic limit. Furthermore these conditions have proven viable for modelling rough boundaries as seen in [\cite{basson2008wall}, \cite{gerard2010relevance}, \cite{pare1992existence}]. The theoretical attraction of these conditions lies largely in their relation to vorticity, as if $w$ is the vorticity of the fluid and $\kappa \in C^2(\partial \mathscr{O};\R)$ denotes the curvature of the boundary then the conditions (\ref{navier boundary conditions}) are equivalent to \begin{equation} \label{its a new rep}
    u \cdot \mathbf{n} = 0, \qquad w = (2\kappa - \alpha)u \cdot \iota.
\end{equation}
This connection stands at the forefront of study into the inviscid limit. The effect of viscosity in the presence of a boundary is an extensively studied and physically meaningful phenomenon, which is well summarised in [\cite{maekawa2016inviscid}] and has seen developments across [\cite{arakeri2000ludwig}, \cite{cebeci1977momentum}, \cite{clauser1956turbulent}, \cite{gie2012boundary}, \cite{gorbushin2018asymptotic}, \cite{iftimie2011viscous}, \cite{keller1978numerical}, \cite{lee2013effect}, \cite{malik1990numerical}, \cite{metivier2004small}, \cite{moore1963boundary}, \cite{morris1984boundary}, \cite{schetz2011boundary}, \cite{shaing2008collisional}, \cite{stevenson2002incipient}, \cite{temam2002boundary}, \cite{weinan2000boundary}] to list only a few contributions in the theory, observation and numerics of this analysis. A key observation is the production of vorticity through large gradients of velocity at the boundary ([\cite{lyman1990vorticity}, \cite{morton1984generation}, \cite{serpelloni2013vertical}, \cite{tani1962production}, \cite{wallace2010measurement}]), suggesting that one requires a  control on the vorticity near the boundary to deduce that the zero viscosity limit exists and is a solution of the Euler Equation (where the impermeability boundary condition $u \cdot \mathbf{n} = 0$ is the only reasonable option). In the aforementioned classical case of the no-slip boundary condition, it is completely unclear if such a control is possible outside of very specific cases regarding regularity of initial data or structure of the domain ([\cite{lopes2008vanishing1}, \cite{lopes2008vanishing}, \cite{masmoudi1998euler}, \cite{sammartino998zero}, \cite{sammartino1998zero}]). The problem is in general wide open, with the leading contribution that of Kato in [\cite{kato1984remarks}] who mathematically verifies the physical intuition that it is necessary and sufficient to control gradients of the velocity in a boundary strip of width approaching zero with viscosity. We note the extensions of this result to the stochastic framework in [\cite{goodair2023zero}, \cite{wang2023kato}]. A positive result seems much more reasonable for the Navier boundary conditions given the representation (\ref{its a new rep}), most clearly in the specific case of $\alpha = 2\kappa$ (the `free boundary' condition) shown first in [\cite{lions1969quelques}] though extended to general $\alpha$ in [\cite{clopeau1998vanishing}, \cite{filho2005inviscid}, \cite{kelliher2006navier}, \cite{masmoudi2012uniform}] for example. It is also worth observing that if we take $\alpha$ to infinity, the constraint $u \cdot \iota = 0$ dominates the second identity in (\ref{navier boundary conditions}) and as such gives some approximation to the no-slip condition; the corresponding convergence of solutions is shown in [\cite{kelliher2006navier}], which could be of interest given the contrasting states of the vanishing viscosity problem.\\

We will consider the inviscid limit of (\ref{number2equationSALT}) and its approximation to the corresponding SALT Euler Equation
\begin{equation}
\label{number3equationSALT}
    u_t - u_0 + \int_0^t\mathcal{L}_{u_s}u_s\,ds  + \int_0^t B(u_s) \circ d\mathcal{W}_s + \nabla \rho_t= 0.
\end{equation}
The main contributions of this paper are the following:
\begin{enumerate}
    \item Proving the existence and uniqueness of global, probabilistically strong, analytically strong solutions of (\ref{number2equation}), (\ref{number3equation}) under Navier boundary conditions (\ref{navier boundary conditions}) where $\alpha \geq \kappa$: see Theorem \ref{theorem 2D strong}.
    \item Proving the existence of a global, probabilistically weak, analytically weak solution of (\ref{number3equationSALT}) as a zero viscosity limit of solutions of (\ref{number2equationSALT}) under Navier boundary conditions (\ref{navier boundary conditions}) where $\alpha = 2\kappa$ and $\kappa \geq 0$: see Theorem \ref{theorem for limit of navier stokes is euler}.
\end{enumerate}
The first contribution enters a significant gap in the literature regarding the existence of analytically strong solutions of fluid PDEs perturbed by a transport type noise in a bounded domain. Prior to the author's works of [\cite{goodair2022stochastic}, \cite{goodair2022existence1}] the only successful existence result of which we are aware was given in [\cite{brzezniak2021well}] where the authors assume that the gradient dependency is small relative to the viscosity (which is necessary in the It\^{o} case). Such an assumption requires little to no extensions of the proof in the case of a general non-differential multiplicative noise, as the additional derivative can be directly controlled with the viscosity in energy estimates (see [\cite{glatt2009strong}] for a zeroth order multiplicative noise in the Navier-Stokes Equations with no-slip boundary condition). For a Stratonovich transport noise such an assumption may not be necessary, as conversion to It\^{o} form yields a corrector term which may provide cancellation of the top order derivative arising from the noise in energy estimates. This type of estimate has been well understood since the works [\cite{gyongy1989approximation}, \cite{gyongy1992stochastic}, \cite{gyongy2003splitting}] and was demonstrated for SALT noise in the presence of a boundary in [\cite{goodair2022navier}]. Whilst this control has facilitated the existence of analytically strong solutions of Stratonovich transport equations on the torus or whole space in many settings ([\cite{alonso2020well}, \cite{crisan2019solution}, \cite{crisan2021theoretical}, \cite{crisan2023well}, \cite{lang2023well}, \cite{lang2023analytical}, \cite{tang2023stochastic}]), and also for analytically weak solutions on a bounded domain in [\cite{goodair2023zero}], the situation is very different for analytically strong solutions with a boundary. The difficulty lies in the energy norm for the solutions: for strong solutions this is produced from a $W^{1,2}$ inner product, a space in which the Leray Projector is not an orthogonal projection and indeed does not commute with the derivatives entering into consideration from this inner product. The Leray Projector fails to preserve the zero trace property, and in general destroys our boundary estimates in the tangent direction. The author's successful result in [\cite{goodair2022navier}] was for the vorticity form of the equation and the free boundary condition of (\ref{its a new rep}) with $\alpha = 2\kappa$, as the Leray Projector is absent in this formulation and the zero vorticity condition allows the same style of integration by parts argument in the Laplacian term to deduce the required control. We note that this was a local existence result in 3D. Regardless it remained unclear how to obtain solutions of the velocity form even for these specific boundary conditions, whilst we note that the existence result was far from optimal as an initial vorticity in $W^{1,2}_0$ was required. The key property which we exploit in the present paper is that the space of divergence-free functions satisfying the Navier boundary conditions (\ref{navier boundary conditions}) are dense in the range of the Leray Projector in $W^{1,2}$, a property which is untrue for the no-slip condition. To avoid excessive technical detail here, we defer a more thorough inspection of this difference and its role in our proof to the conclusion, as well as the necessity of the condition $\alpha \geq \kappa$. We note that the assumption is physically expected for small viscosity: indeed, the Navier boundary conditions as originally derived are for a function $a:\partial \mathscr{O} \rightarrow \R$ such that $\alpha = \frac{a}{\nu}$ (see [\cite{kelliher2006navier}] (1.1)). Therefore as viscosity becomes small then physically $\alpha$ should be large, where the assumption that $\alpha$ is non-negative is classical to the study of the problem (e.g. [\cite{clopeau1998vanishing}, \cite{filho2005inviscid}]) and its physical interpretation as a coefficient of friction.
\\

As for the second contribution, our primary comment is that the existence of any solution of the Euler Equation with a transport noise on a bounded domain was completely open. The closest result that we have seen comes from [\cite{brzezniak2001stochastic}], where the authors proved the existence of a global, probabilistically weak, analytically weak solution of the Euler Equation with a somewhat general multiplicative noise, though one which does not allow for gradient dependency. The solution is constructed as an inviscid limit of linearised and regularised Navier-Stokes equations satisfying the free boundary condition. For linear multiplicative noise the vanishing viscosity limit of the full Navier-Stokes system was established in [\cite{bessaih1999martingale}], again with the free boundary condition. The zero viscosity limit for the general Navier boundary conditions has thus far only been determined with additive noise, courtesy of [\cite{cipriano2015inviscid}]. The core property of any approach to this problem lies in showing uniform in time  $W^{1,2}$ estimates on the Navier-Stokes Equations which are independent of the viscosity. The classical deterministic method relies on a maximum principle (see the fundamental works of [\cite{clopeau1998vanishing}, \cite{filho2005inviscid}]), which extends to the case of additive noise as the difference of the solution of the equation with the noise satisfies a deterministic parabolic PDE pathwise. Lacking the ability to do this for multiplicative noise, one instead looks to show an energy estimate of the vorticity directly and in doing so requires the free boundary condition for the Laplacian term. Our method is the same and so we impose that $\alpha = 2\kappa$, and our original requirement that $\alpha \geq \kappa$ implies that one must take $\kappa \geq 0$.\\

We now detail the structure of the paper along with the key ideas at each step:
\begin{itemize}
    \item Section \ref{section preliminaries} is devoted to the setup of the problem in terms of notation, functional framework, assumptions on the general noise and a proper introduction of the Stochastic Advection by Lie Transport case. In Subsection \ref{sub def} we properly define our notions of solution and give the exact statements of the main results.
    
    \item Section \ref{section solutions sns} contains the proof of the well-posedness results for the Stochastic Navier-Stokes Equation (Theorem \ref{theorem 2D strong}). We begin in Subsection \ref{sub weak sols} by proving the existence and uniqueness of analytically weak solutions, which needs only a minor adaptation of the no-slip case proved in [\cite{goodair2023zero}]. The result now boils down to showing that this weak solution has additional regularity, which we show in Subsection \ref{sub strong sols} via uniform boundedness at the level of a Galerkin Approximation up until first hitting times of the processes. The approach was used in [\cite{goodair2022existence1}] and we employ some of the abstract arguments of this paper to reduce the technical details required here. Particular attention should be given to how the solutions are shown to be global; in [\cite{goodair2022existence1}] this method of first hitting times is used to initially deduce the existence of a local strong solution, which after some heavy arguments gives rise to a maximal strong solution where the maximal time can be characterised by the blow-up. The local strong solution is obtained through an abstract result in [\cite{glatt2009strong}] (Lemma 5.1), a paper in which global solutions of the 2D Navier-Stokes Equation with multiplicative (non-differential) noise are then established using this characterisation of the maximal time. The abstract lemma asserts that, under assumptions of a Cauchy property of a sequence of processes  up until their first hitting times and some weak equicontinuity at the \textit{initial} time, then a limiting process and positive stopping time exist (which are then argued to be a local strong solution, as a limit of the Galerkin Approximation). No characterisation of this stopping time is given though, hence completely separate arguments are required to pass to a maximal solution and beyond.\\
    
    To greatly simplify this procedure we extend the lemma of [\cite{glatt2009strong}], demonstrating that if instead one imposes that the processes satisfy a weak equicontinuity assumption at \textit{all} times then the limiting stopping time can be taken as a first hitting time of the limiting process for an arbitrarily large hitting parameter. Application of this result \textit{immediately} yields that solutions exist up until blow-up, pertinently in our 2D case as blow-up \textit{in the energy norm of the weak solution}; we emphasise again the advantage that this has, as the previous approaches discussed would only produce a characterisation of the maximal time in the energy norm of the strong solution. One would have to commit to further analysis to show the equivalence of these blow-ups, which is particularly non-trivial for the Navier boundary conditions; indeed this equivalence demands a priori estimates of solutions in the energy norm of the strong solution, which would be produced by considering a $W^{2,2}$ and $L^2$ duality pairing equivalent to the $W^{1,2}$ inner product, though (unlike the no-slip case) this is not provided by the $L^2$ inner product of the gradient and is unclear in general. Our result demands only such an energy estimate at the level of the Galerkin Approximation, which is clearer as they satisfy the identity in $W^{1,2}$ so this inner product could be used directly. Given the potential significance of this result for future work, we leave the statement as an abstract lemma in the appendix and only apply it in Section \ref{section solutions sns}. In Subsection \ref{sub lie transport} we demonstrate that the SALT noise satisfies the assumptions of the general noise used in this result.

    \item Section \ref{section zero viscy limit} contains the proof of the existence result for the SALT Euler Equation as an inviscid limit (Theorem \ref{theorem for limit of navier stokes is euler}). In Subsection \ref{sub estimates indep viscos} we establish the discussed estimates on the SALT Navier-Stokes Equations uniformly in viscosity. In Subsection \ref{sub tighty} we show that the laws of the sequence of solutions with viscosities approaching zero is tight in the space of probability measures over a uniform in time $L^2$ space, using a very similar tightness criterion to one used in [\cite{goodair2023zero}] which combines the classical Aldous result with a weak criterion from Jakubowski. Passage to the SALT Euler Equation on a new probability space in Subsection \ref{passage to stochy euler} is then straightforwards. 

    \item Section \ref{section conclusion} is a conclusion of the paper in which we discuss with greater technical detail how this work fits into the present and potentially future literature, along with other considerations that arose throughout the paper. Notably we provide precise detail as to how the Navier boundary conditions solve the issues faced by the no sip condition for transport noise, the necessity of $\alpha \geq \kappa$ in Theorem \ref{theorem 2D strong}, as well as types of noise which fit our assumptions and for which Theorem \ref{theorem for limit of navier stokes is euler} remains valid. Here we highlight how the SALT structure is imperative to the solution theory for transport noise. In addition, we ponder the uniqueness of solutions of the SALT Euler Equation and an approximation of the Stochatic Navier Stokes Equations equipped with the no-slip condition by those with Navier boundary conditions. 

    \item Section \ref{section appendix} is our appendix which contains abstract results used in the paper. In Subsection \ref{sub cauchy} we state and prove our extension of [\cite{glatt2009strong}] Lemma 5.1 used in Theorem \ref{theorem 2D strong} (which is ironically again Lemma \ref{amazing cauchy lemma}). Subsection \ref{Appendix III} details the framework of [\cite{goodair2022existence1}] used in Theorem \ref{theorem 2D strong}, whilst Subsection \ref{sub useful results} contains other required results from the literature, including a statement and proof of the tightness criterion used in Theorem \ref{theorem for limit of navier stokes is euler}.

\end{itemize}

\section{Preliminaries} \label{section preliminaries}

\subsection{Elementary Notation} \label{sub elementary notation}

In the following $\mathscr{O} \subset \R^2$ will be a smooth bounded domain equipped with Euclidean norm and Lebesgue measure $\lambda$. We consider Banach Spaces as measure spaces equipped with their corresponding Borel $\sigma$-algebra. Let $(\mathcal{X},\mu)$ denote a general topological measure space, $(\mathcal{Y},\norm{\cdot}_{\mathcal{Y}})$ and $(\mathcal{Z},\norm{\cdot}_{\mathcal{Z}})$ be separable Banach Spaces, and $(\mathcal{U},\inner{\cdot}{\cdot}_{\mathcal{U}})$, $(\mathcal{H},\inner{\cdot}{\cdot}_{\mathcal{H}})$ be general separable Hilbert spaces. We introduce the following spaces of functions. 
\begin{itemize}
    \item $L^p(\mathcal{X};\mathcal{Y})$ is the  class of measurable $p-$integrable functions from $\mathcal{X}$ into $\mathcal{Y}$, $1 \leq p < \infty$, which is a Banach space with norm $$\norm{\phi}_{L^p(\mathcal{X};\mathcal{Y})}^p := \int_{\mathcal{X}}\norm{\phi(x)}^p_{\mathcal{Y}}\mu(dx).$$ In particular $L^2(\mathcal{X}; \mathcal{Y})$ is a Hilbert Space when $\mathcal{Y}$ itself is Hilbert, with the standard inner product $$\inner{\phi}{\psi}_{L^2(\mathcal{X}; \mathcal{Y})} = \int_{\mathcal{X}}\inner{\phi(x)}{\psi(x)}_\mathcal{Y} \mu(dx).$$ In the case $\mathcal{X} = \mathscr{O}$ and $\mathcal{Y} = \R^2$ note that $$\norm{\phi}_{L^2(\mathscr{O};\R^2)}^2 = \sum_{l=1}^2\norm{\phi^l}^2_{L^2(\mathscr{O};\R)}, \qquad \phi = \left(\phi^1, \dots, \phi^2\right), \quad \phi^l: \mathscr{O} \rightarrow \R.$$ We denote $\norm{\cdot}_{L^p(\mathscr{O};\R^2)}$ by $\norm{\cdot}_{L^p}$ and $\norm{\cdot}_{L^2(\mathscr{O};\R^2)}$ by $\norm{\cdot}$.
    
\item $L^{\infty}(\mathcal{X};\mathcal{Y})$ is the class of measurable functions from $\mathcal{X}$ into $\mathcal{Y}$ which are essentially bounded. $L^{\infty}(\mathcal{X};\mathcal{Y})$ is a Banach Space when equipped with the norm $$ \norm{\phi}_{L^{\infty}(\mathcal{X};\mathcal{Y})} := \inf\{C \geq 0: \norm{\phi(x)}_Y \leq C \textnormal{ for $\mu$-$a.e.$ }  x \in \mathcal{X}\}.$$

      \item $C(\mathcal{X};\mathcal{Y})$ is the space of continuous functions from $\mathcal{X}$ into $\mathcal{Y}$.

      \item $C_w(\mathcal{X};\mathcal{Y})$ is the space of `weakly continuous' functions from $\mathcal{X}$ into $\mathcal{Y}$, by which we mean continuous with respect to the given topology on $\mathcal{X}$ and the weak topology on $\mathcal{Y}$. 
      
    \item $C^m(\mathscr{O};\R)$ is the space of $m \in \N$ times continuously differentiable functions from $\mathscr{O}$ to $\R$, that is $\phi \in C^m(\mathscr{O};\R)$ if and only if for every $2$ dimensional multi index $\alpha = \alpha_1, \alpha_2$ with $\abs{\alpha}\leq m$, $D^\alpha \phi \in C(\mathscr{O};\R)$ where $D^\alpha$ is the corresponding classical derivative operator $\partial_{x_1}^{\alpha_1} \partial_{x_2}^{\alpha_2}$.
    
    \item $C^\infty(\mathscr{O};\R)$ is the intersection over all $m \in \N$ of the spaces $C^m(\mathscr{O};\R)$.
    
    \item $C^m_0(\mathscr{O};\R)$ for $m \in \N$ or $m = \infty$ is the subspace of $C^m(\mathscr{O};\R)$ of functions which have compact support.
    
    \item $C^m(\mathscr{O};\R^2), C^m_0(\mathscr{O};\R^2)$ for $m \in \N$ or $m = \infty$ is the space of functions from $\mathscr{O}$ to $\R^2$ whose component mappings each belong to $C^m(\mathscr{O};\R), C^m_0(\mathscr{O};\R)$.
    
        \item $W^{m,p}(\mathscr{O}; \R)$ for $1 \leq p < \infty$ is the sub-class of $L^p(\mathscr{O}, \R)$ which has all weak derivatives up to order $m \in \N$ also of class $L^p(\mathscr{O}, \R)$. This is a Banach space with norm $$\norm{\phi}^p_{W^{m,p}(\mathscr{O}, \R)} := \sum_{\abs{\alpha} \leq m}\norm{D^\alpha \phi}_{L^p(\mathscr{O}; \R)}^p,$$ where $D^\alpha$ is the corresponding weak derivative operator. In the case $p=2$ the space $W^{m,2}(\mathscr{O}, \R)$ is Hilbert with inner product $$\inner{\phi}{\psi}_{W^{m,2}(\mathscr{O}; \R)} := \sum_{\abs{\alpha} \leq m} \inner{D^\alpha \phi}{D^\alpha \psi}_{L^2(\mathscr{O}; \R)}.$$
    
    \item $W^{m,\infty}(\mathscr{O};\R)$ for $m \in \N$ is the sub-class of $L^\infty(\mathscr{O}, \R)$ which has all weak derivatives up to order $m \in \N$ also of class $L^\infty(\mathscr{O}, \R)$. This is a Banach space with norm $$\norm{\phi}_{W^{m,\infty}(\mathscr{O}, \R)} := \sup_{\abs{\alpha} \leq m}\norm{D^{\alpha}\phi}_{L^{\infty}(\mathscr{O};\R^2)}.$$
    
   \item $W^{s,p}(\mathscr{O}; \R)$ for $0<s<1$ and $1 \leq p < \infty$ is the sub-class of functions $\phi \in L^p(\mathscr{O}, \R)$ such that $$\int_{\mathscr{O} \times \mathscr{O}} \frac{\abs{\phi(x)-\phi(y)}^p}{\abs{x-y}^{sp+2}}d\lambda(x,y) < \infty.$$ This is a Banach space with respect to the norm $$\norm{\phi}_{W^{s,p}(\mathscr{O}; \R)}^p= \norm{\phi}_{L^p(U;\R)}^p + \int_{\mathscr{O} \times \mathscr{O}} \frac{\abs{\phi(x)-\phi(y)}^p}{\abs{x-y}^{sp+2}}d\lambda(x,y).$$ For $p=2$ this is a Hilbert space with inner product $$\inner{\phi}{\psi}_{W^{s,2}(\mathscr{O}; \R)} = \inner{\phi}{\psi}_{L^2(\mathscr{O};\R)}+ \int_{\mathscr{O} \times \mathscr{O}} \frac{\big(\phi(x)-\phi(y)\big)\big(\psi(x)-\psi(y)\big)}{\abs{x-y}^{2s+2}}d\lambda(x,y).$$
    
    \item $W^{s,p}(\mathscr{O}; \R)$ for $1 \leq s < \infty,$ $s \not\in \N$ and $1 \leq p < \infty$ is, using the notation $\floor{s}$ to mean the integer part of $s$, the sub-class of $W^{\floor{s},p}(\mathscr{O}; \R)$ such that the distributional derivatives $D^\alpha\phi$ belong to $W^{s-\floor{s},p}(\mathscr{O}; \R)$ for every multi-index $\alpha$ such that $\abs{\alpha}=\floor{s}.$ This is a Banach space with norm $$\norm{\phi}_{W^{s,p}(\mathscr{O}; \R)}^p= \norm{\phi}_{W^{\floor{s},p}(\mathscr{O};\R)}^p + \sum_{\abs{\alpha}=\floor{s}} \int_{\mathscr{O} \times \mathscr{O}} \frac{\abs{D^\alpha\phi(x)-D^\alpha\phi(y)}^p}{\abs{x-y}^{(s-\floor{s})p+2}}dxdy$$ and a Hilbert space in the case $p=2$, with inner product $$\inner{\phi}{\psi}_{W^{s,2}(\mathscr{O}; \R)} = \inner{\phi}{\psi}_{W^{\floor{s},2}(\mathscr{O};\R)}+ \int_{\mathscr{O} \times \mathscr{O}} \frac{\big(D^\alpha\phi(x)-D^\alpha\phi(y)\big)\big(D^\alpha\psi(x)-D^\alpha\psi(y)\big)}{\abs{x-y}^{2(s-\floor{s})+2}}d\lambda(x,y).$$

 \item $W^{s,p}(\mathscr{O}; \R^2)$ for $s > 0$ and $1 \leq p < \infty$ is the Banach space of functions $\phi:\mathscr{O} \rightarrow \R^2$ whose components $(\phi^l)$ are each elements of the space $W^{s,p}(\mathscr{O}; \R).$ The associated norm is $$
    \norm{\phi}_{W^{s,p}}^p= \sum_{l=1}^2\norm{\phi^l}^p_{W^{s,p}(\mathscr{O}; \R)}
    $$ and similarly when $p=2$ this space is Hilbert with inner product $$\inner{\phi}{\psi}_{W^{s,2}} = \sum_{l=1}^2\inner{\phi^l}{\psi^l}_{W^{s,2}(\mathscr{O}, \R)}.$$
    
          \item $W^{m,\infty}(\mathscr{O}; \R^2)$ is the sub-class of $L^\infty(\mathscr{O}, \R^2)$ which has all weak derivatives up to order $m \in \N$ also of class $L^\infty(\mathscr{O}, \R^2)$. This is a Banach space with norm $$\norm{\phi}_{W^{m,\infty}} := \sup_{l \leq N}\norm{\phi^l}_{W^{m,\infty}(\mathscr{O}; \R)}.$$

    \item $W^{m,p}_0(\mathscr{O};\R), W^{m,p}_0(\mathscr{O};\R^2)$ for $m \in \N$ and $1 \leq p \leq \infty$ is the closure of $C^\infty_0(\mathscr{O};\R),C^\infty_0(\mathscr{O};\R^2)$ in $W^{m,p}(\mathscr{O};\R), W^{m,p}(\mathscr{O};\R^2)$.

    \item $\mathscr{L}(\mathcal{Y};\mathcal{Z})$ is the space of bounded linear operators from $\mathcal{Y}$ to $\mathcal{Z}$. This is a Banach Space when equipped with the norm $$\norm{F}_{\mathscr{L}(\mathcal{Y};\mathcal{Z})} = \sup_{\norm{y}_{\mathcal{Y}}=1}\norm{Fy}_{\mathcal{Z}}$$ and is simply the dual space $\mathcal{Y}^*$ when $\mathcal{Z}=\R$, with operator norm $\norm{\cdot}_{\mathcal{Y}^*}.$
    
     \item $\mathscr{L}^2(\mathcal{U};\mathcal{H})$ is the space of Hilbert-Schmidt operators from $\mathcal{U}$ to $\mathcal{H}$, defined as the elements $F \in \mathscr{L}(\mathcal{U};\mathcal{H})$ such that for some basis $(e_i)$ of $\mathcal{U}$, $$\sum_{i=1}^\infty \norm{Fe_i}_{\mathcal{H}}^2 < \infty.$$ This is a Hilbert space with inner product $$\inner{F}{G}_{\mathscr{L}^2(\mathcal{U};\mathcal{H})} = \sum_{i=1}^\infty \inner{Fe_i}{Ge_i}_{\mathcal{H}}$$ which is independent of the choice of basis.

     \item For any $T>0$, $\mathscr{S}_T$ is the subspace of $C\left([0,T];[0,T]\right)$ of strictly increasing functions.

     \item For any $T>0$, $\mathcal{D}\left([0,T];\mathcal{Y}\right)$ is the space of c\'{a}dl\'{a}g functions from $[0,T]$ into $\mathcal{Y}$. It is a complete separable metric space when equipped with the metric $$d(\phi,\psi) := \inf_{\eta \in \mathscr{S}_T}\left[\sup_{t \in [0,T]}\left\vert \eta(t)- t\right\vert \vee \sup_{t \in [0,T]}\left\Vert \phi(t)-\psi(\eta(t)) \right\Vert_{\mathcal{Y}} \right]$$ which induces the so called Skorohod Topology (see [\cite{billingsley2013convergence}] pp124 for details).

\end{itemize}

We next give the probabilistic set up. Let $(\Omega,\mathcal{F},(\mathcal{F}_t), \mathbb{P})$ be a fixed filtered probability space satisfying the usual conditions of completeness and right continuity. We take $\mathcal{W}$ to be a cylindrical Brownian motion over some Hilbert Space $\mathfrak{U}$ with orthonormal basis $(e_i)$. Recall (e.g. [\cite{lototsky2017stochastic}], Definition 3.2.36) that $\mathcal{W}$ admits the representation $\mathcal{W}_t = \sum_{i=1}^\infty e_iW^i_t$ as a limit in $L^2(\Omega;\mathfrak{U}')$ whereby the $(W^i)$ are a collection of i.i.d. standard real valued Brownian Motions and $\mathfrak{U}'$ is an enlargement of the Hilbert Space $\mathfrak{U}$ such that the embedding $J: \mathfrak{U} \rightarrow \mathfrak{U}'$ is Hilbert-Schmidt and $\mathcal{W}$ is a $JJ^*-$cylindrical Brownian Motion over $\mathfrak{U}'$. Given a process $F:[0,T] \times \Omega \rightarrow \mathscr{L}^2(\mathfrak{U};\mathscr{H})$ progressively measurable and such that $F \in L^2\left(\Omega \times [0,T];\mathscr{L}^2(\mathfrak{U};\mathscr{H})\right)$, for any $0 \leq t \leq T$ we define the stochastic integral $$\int_0^tF_sd\mathcal{W}_s:=\sum_{i=1}^\infty \int_0^tF_s(e_i)dW^i_s,$$ where the infinite sum is taken in $L^2(\Omega;\mathscr{H})$. We can extend this notion to processes $F$ which are such that $F(\omega) \in L^2\left( [0,T];\mathscr{L}^2(\mathfrak{U};\mathscr{H})\right)$ for $\mathbb{P}-a.e.$ $\omega$ via the traditional localisation procedure. In this case the stochastic integral is a local martingale in $\mathscr{H}$. \footnote{A complete, direct construction of this integral, a treatment of its properties and the fundamentals of stochastic calculus in infinite dimensions can be found in [\cite{prevot2007concise}] Section 2.} We shall make frequent use of the Burkholder-Davis-Gundy Inequality ([\cite{da2014stochastic}] Theorem 4.36), passage of a bounded linear operator through the stochastic integral ([\cite{prevot2007concise}] Lemma 2.4.1) and the It\^{o} Formula ([\cite{da2014stochastic}] Theorem 4.32, Proposition \ref{rockner prop}). 

\subsection{Functional Framework} \label{functional framework subsection}

We now recap the classical functional framework for the study of the deterministic Navier-Stokes Equation. We formally define the operator $\mathcal{L}$ as well as the divergence-free and Navier boundary conditions. The nonlinear operator $\mathcal{L}$ is defined for sufficiently regular functions $f,g:\mathcal{O} \rightarrow \R^2$ by $\mathcal{L}_fg:= \sum_{j=1}^2f^j\partial_jg.$ Here and throughout the text we make no notational distinction between differential operators acting on a vector valued function or a scalar valued one; that is, we understand $\partial_jg$ by its component mappings $(\partial_lg)^l := \partial_jg^l$. For any $m \geq 1$, the mapping $\mathcal{L}: W^{m+1,2}(\mathscr{O};\R^2) \rightarrow W^{m,2}(\mathscr{O};\R^2)$ defined by $f \mapsto \mathcal{L}_ff$ is continuous (see e.g. [\cite{goodair2022navier}] Lemma 1.2). Some more technical properties of the operator are given at the end of this subsection. For the divergence-free condition we mean a function $f$ such that the property $$\textnormal{div}f := \sum_{j=1}^2 \partial_jf^j = 0$$ holds. We require this property and the boundary condition to hold for our solution $u$ at all times, though there is some ambiguity as to how we understand these conditions for a solution $u$ which need not be defined pointwise everywhere on $\bar{\mathscr{O}}$. We shall understand these conditions in their traditional weak sense, that is for weak derivatives $\partial_j$ so $\sum_{j=1}^2 \partial_jf^j = 0$ holds as an identity in $L^2(\mathscr{O};\R)$ whilst the boundary condition is understood in terms of trace. To be precise we first define the restriction mapping on functions $f \in W^{1,2}(\mathscr{O};\R) \cap C(\bar{\mathscr{O}};\R)$ by the restriction of $f$ to the boundary $\partial \mathscr{O}$, which is then shown to be a bounded operator into $W^{\frac{1}{2},2}(\partial \mathscr{O};\R)$ (see Lemma \ref{bounded trace} and more classical sources of e.g. [\cite{evans2010partial}]). By the density of $C(\bar{\mathscr{O}};\R)$ then the trace operator is well defined on the whole of $W^{1,2}(\mathscr{O};\R)$ as a continuous linear extension of the restriction mapping, and furthermore on $W^{1,2}(\mathscr{O};\R^2)$ by the trace of the components. The rate of strain tensor $D$ appearing in (\ref{navier boundary conditions}) is a mapping $D: W^{1,2}(\mathscr{O};\R^2) \rightarrow L^{2}(\mathscr{O};\R^{2 \times 2})$ defined by $$f \mapsto \begin{bmatrix}
        \partial_1f^1 & \frac{1}{2}\left(\partial_1f^1 + \partial_2f^2\right)\\
        \frac{1}{2}\left(\partial_1f^1 + \partial_2f^2\right) & \partial_2f^2
    \end{bmatrix}
    $$
    or in component form, $(Df)^{k,l}:=\frac{1}{2}\left(\partial_kf^l + \partial_lf^k\right)$. Note that if $f \in W^{2,2}(\mathscr{O};\R^2)$ then $Df \in W^{1,2}(\mathscr{O};\R^{2 \times 2})$ so the trace of its components is well defined and henceforth we understand the boundary condition $$ 2(Df)\mathbf{n} \cdot \iota + \alpha f \cdot \iota = 0$$ on $\partial \mathscr{O}$ to be in this trace sense. The same is true for $f \cdot \mathbf{n} = 0$. We look to impose these conditions by incorporating them into the function spaces where our solution takes value, and will always assume that $\alpha \in C^2(\partial \mathscr{O};\R)$ so as to match the regularity required in [\cite{clopeau1998vanishing}].
\begin{definition}
We define $C^{\infty}_{0,\sigma}(\mathscr{O};\R^2)$ as the subset of $C^{\infty}_0(\mathscr{O};\R^2)$ of functions which are divergence-free. $L^2_\sigma$ is defined as the completion of $C^{\infty}_{0,\sigma}(\mathscr{O};\R^2)$ in $L^2(\mathscr{O};\R^2)$, whilst we introduce $\bar{W}^{1,2}_\sigma$ as the intersection of $W^{1,2}(\mathscr{O};\R^2)$ with $L^2_\sigma$ and $\bar{W}^{2,2}_{\alpha}$ by $$\bar{W}^{2,2}_{\alpha}:= \left\{f \in W^{2,2}(\mathscr{O};\R^2) \cap \bar{W}^{1,2}_{\sigma}: 2(Df)\mathbf{n} \cdot \iota + \alpha f \cdot \iota = 0 \textnormal{ on } \partial \mathscr{O}\right\}.$$
\end{definition}

\begin{remark} \label{new first labelled remark}
    $L^2_{\sigma}$ can be characterised as the subspace of $L^2(\mathscr{O};\R^2)$ of weakly divergence-free functions with normal component weakly zero at the boundary (see [\cite{robinson2016three}] Lemma 2.14). Moreover the complement space of $L^2_{\sigma}$ in $L^2(\mathscr{O};\R^2)$ is characterised as the subspace of $L^2(\mathscr{O};\R^2)$ of functions $f$ such that there exists a $g \in W^{1,2}(\mathscr{O};\R)$ with the property that $f = \nabla g$ (see [\cite{robinson2016three}] Theorem 2.16).
\end{remark}

\begin{remark}  \label{first labelled remark}
    $\bar{W}^{1,2}_{\sigma}$ is precisely the subspace of $W^{1,2}(\mathscr{O};\R^2)$ consisting of divergence-free functions $f$ such that $f \cdot \mathbf{n} = 0$ on $\partial \mathscr{O}$. Moreover as both $D: W^{2,2}(\mathscr{O};\R^2) \rightarrow W^{1,2}(\mathscr{O};\R^{2 \times 2})$ and the trace mapping $W^{1,2}(\mathscr{O};\R) \rightarrow L^2(\partial \mathscr{O} ; \R)$ are continuous, then $\bar{W}^{1,2}_{\sigma}$, $\bar{W}^{2,2}_{\alpha}$ are closed in the $W^{1,2}(\mathscr{O};\R^2)$, $W^{2,2}(\mathscr{O};\R^2)$ norms respectively.
\end{remark}

We note that the Poincar\'{e} Inequality holds for the component mappings of functions in $\bar{W}^{1,2}_{\sigma}$. The inequality (see e.g. [\cite{robinson2016three}] Theorem 1.9) holds for the component mapping $f^j$ of $f \in \bar{W}^{1,2}_{\sigma}$ if $$ \int_{\mathscr{O}}f^j d\lambda = 0,$$ and observe that via an approximation with the set $C^{\infty}_{0,\sigma}(\mathscr{O};\R^2)$ which is dense in $L^2_{\sigma}$, one can conclude that for all $g \in L^2_{\sigma}$ and $\psi \in W^{1,2}(\mathscr{O};\R)$, $$\inner{g}{\nabla \psi} = 0$$ (this is the statement of [\cite{robinson2016three}] Lemma 2.11). Moreover by choosing $\phi$ as the function $\phi(x^1,x^2) := x^j$, then $$\int_{\mathscr{O}}f^j d\lambda = \inner{f}{\nabla \phi} = 0$$ so the Poincar\'{e} Inequality holds, and as such we equip $\bar{W}^{1,2}_{\sigma}$ with the inner product $$\inner{f}{g}_1 := \sum_{j=1}^2 \inner{\partial_j f}{\partial_j g}$$ which is equivalent to the full $W^{1,2}(\mathscr{O};\R^2)$ one. We shall also endow $W^{1,2}_0(\mathscr{O};\R)$ with the corresponding one dimensional inner product, and label it $\inner{\cdot}{\cdot}_{W^{1,2}_0}$. Likewise we look to equip $\bar{W}^{2,2}_{\alpha}$ with a new inner product, and we do so by introducing the Leray Projector $\mathcal{P}$ as the orthogonal projection in $L^2(\mathscr{O};\R)$ onto $L^2_{\sigma}$. It is well known (see e.g. [\cite{temam2001navier}] Remark 1.6.) that for any $m \in \N$, $\mathcal{P}$ is continuous as a mapping $\mathcal{P}: W^{m,2}(\mathscr{O};\R^2) \rightarrow W^{m,2}(\mathscr{O};\R^2)$. Through $\mathcal{P}$ we define the Stokes Operator $A: W^{2,2}(\mathscr{O};\R^2) \rightarrow L^2_{\sigma}$ by $A:= -\mathcal{P}\Delta$. We understand the Laplacian as an operator on vector valued functions through the component mappings, $(\Delta f)^l := \Delta f^l$. From the continuity of $\mathcal{P}$ we have immediately that for $m \in \{0\} \cup \N$, $A: W^{m+2,2}(\mathscr{O};\R^2) \rightarrow W^{m,2}(\mathscr{O};\R^2)$ is continuous.

\begin{lemma}
    The bilinear form $\inner{f}{g}_2:= \inner{Af}{Ag}$ defines an inner product on $\bar{W}^{2,2}_{\alpha}$ equivalent to the standard $W^{2,2}(\mathscr{O};\R^2)$ inner product. 
\end{lemma}

\begin{proof}
    We must show the existence of constants $c_1$, $c_2$ such that for all $f \in \bar{W}^{2,2}_{\alpha}$, $$c_1\norm{f}_{W^{2,2}}^2 \leq \norm{f}_2^2 \leq c_2\norm{f}_{W^{2,2}}^2.$$
The constant $c_2$ can in fact be taken as $1$, as $$\norm{f}_2^2 \leq \norm{\Delta f}^2 \leq \norm{f}_{W^{2,2}}^2.$$ The existence of such a $c_1$ is much more challenging and relies on estimates of the Stokes equation with the Navier boundary conditions, which has been proved in [\cite{tapia2021stokes}] Theorem 5.10. We use, in their notation, that $\mathbf{f} = A\mathbf{u}$ is a solution of the Stokes problem with $\pi = 0$, which gives the result.
\end{proof}

The following lemma will further facilitate our work in these spaces.

\begin{lemma} \label{eigenfunctions for navier}
    There exists a collection of functions $(\bar{a}_k)$, $\bar{a}_k \in W^{3,2}(\mathscr{O};\R^2) \cap \bar{W}^{2,2}_{\alpha}$, such that the $(a_k)$ are eigenfunctions of $A$, are an orthonormal basis in $L^2_{\sigma}$ and an orthogonal basis in $\bar{W}^{1,2}_{\sigma}$. Moreover the eigenvalues $(\bar{\lambda}_k)$ are strictly positive and apporach infinity as $k \rightarrow \infty$.
\end{lemma}

\begin{proof}
    This is the content of [\cite{clopeau1998vanishing}] Lemma 2.2, where $(\bar{a}_k)$ is $(v_k)$ in their notation. The fact that this system consists of eigenfunctions of $A$ follows from (2.10) by taking $\mathcal{P}$ of the top line: that is, $$\bar{\lambda}_k\bar{a}_k = \mathcal{P}\left(\bar{\lambda}_k\bar{a}_k\right) = \mathcal{P}\left(-\Delta \bar{a}_k + \nabla \pi_k  \right) = A \bar{a}_k$$ using that $\mathcal{P}\nabla \pi_k = 0$ from Remark \ref{new first labelled remark}.
\end{proof}

One can see that we are building a framework to parallel that of the classic no-slip case, though a significant difference comes in the presence of a boundary integral for Green's type identities. What we achieve now is the following, recalling $\kappa \in C^2(\partial \mathscr{O};\R)$ to be the curvature of $\partial \mathscr{O}$.

\begin{lemma} \label{greens for navier}
    For $f \in \bar{W}^{2,2}_{\alpha}$, $\phi \in \bar{W}^{1,2}_{\sigma}$, we have that $$\inner{\Delta f}{\phi}  = -\inner{f}{\phi}_1 + \inner{(\kappa - \alpha)f}{\phi}_{L^2(\partial \mathscr{O}; \R^2)}.$$
\end{lemma}

\begin{proof}
    This is demonstrated in [\cite{kelliher2006navier}] equation (5.1).
\end{proof}

Due to this boundary integral we will make great use of Lebesgue spaces and fractional Sobolev spaces defined on the boundary $\partial \mathscr{O}$. The spaces $L^p(\partial\mathscr{O};\R^2)$ can be defined precisely as in Subsection \ref{sub elementary notation} for $\partial \mathscr{O}$ equipped with its surface measure, and in fact the same is true for $W^{s,2}(\partial \mathscr{O};\R)$ with $0 < s < 1$ and hence $W^{s,2}(\partial \mathscr{O};\R^2)$ in the same manner. Indeed this definition is given in [\cite{grisvard2011elliptic}] pp.20, where it is shown to be equivalent to the often used definition locally as the space $W^{s,2}(\R;\R)$ via a coordinate transformation and partition of unity. The stability under a change of variables ensures results of H\"{o}lder's Inequality and (one dimensional) Sobolev Embeddings hold on the boundary for these spaces. The relation to the trace of functions in $\mathscr{O}$ is stated now.
\begin{lemma} \label{bounded trace}
    For $\frac{1}{2} <s < \frac{3}{2}$ the trace operator is bounded and linear from $W^{s,2}(\mathscr{O};\R)$ into $W^{s- \frac{1}{2},2}(\partial \mathscr{O};\R)$. 
\end{lemma}

\begin{proof}
    See [\cite{ding1996proof}] Theorem 1. 
\end{proof}

We shall make use of this result for the characterisation of the fractional Sobolev spaces as interpolation spaces. In fact the renowned book of Adams [\cite{adams2003sobolev}] defines these spaces in this way (7.57, pp.250), and their equivalence is well understood (see for example [\cite{tartar2007introduction}] pp.83). A proof of this equivalence relies on a norm preserving extension operator for the space as defined by interpolation (which follows from the classical integer valued case, see [\cite{adams2003sobolev}] Theorem 5.24 and [\cite{stein1970singular}] chapter 6), the equivalence of the interpolation space on $\R^2$ with that defined by Fourier transformations (see [\cite{adams2003sobolev}] 7.63 pp.252), the further equivalence of this space with our definition on $\R^2$ ([\cite{demengel2012functional}] Proposition 4.17), and finally a norm preserving extension for this fractional Sobolev space ([\cite{di2012hitchhikers}] Theroem 5.4). Therefore there exists a constant $c$ such that for $f \in W^{1,2}(\mathscr{O};\R^2)$ and $0 < s < 1$, we have that \begin{equation} \label{interpolation}
    \norm{f}_{W^{s,2}} \leq c\norm{f}^{1-s}\norm{f}_{W^{1,2}}^s.
\end{equation}
In particular for $s=\frac{1}{2}$, if the result of Lemma \ref{bounded trace} were true in this limiting case (that is, an embedding of $W^{\frac{1}{2},2}(\mathscr{O};\R^2)$ into $L^2(\partial \mathscr{O};\R^2)$) then combined with (\ref{interpolation}) we would obtain \begin{equation}
    \label{inequality from Lions}
    \norm{f}_{L^2(\partial \mathscr{O};\R^2)}^2 \leq c\norm{f}\norm{f}_{W^{1,2}}.
\end{equation}
This inequality is in fact true and is classical in the study of our problem (see for example [\cite{lions1996mathematical}] pp.130, [\cite{kelliher2006navier}] equation (2.5)) though some additional machinery is required to prove it. In short the result can be achieved by showing that the trace operator is a continuous linear operator from an appropriate Besov space which similarly interpolates between $L^2(\mathscr{O};\R^2)$ and $W^{1,2}(\mathscr{O};\R^2)$ (see [\cite{adams2003sobolev}] Theorem 7.43 and Remark 7.45 for the trace embedding, and the Besov Spaces subchapter for the interpolation). Moving on we introduce the finite dimensional projections $(\bar{\mathcal{P}}_n)$, where $\bar{\mathcal{P}}_n$ is the orthogonal projection onto $\bar{V}_n:=\textnormal{span}\{\bar{a}_1, \dots, \bar{a}_n\}$ in $L^2(\mathscr{O};\R^2)$ (note that $\bar{V}_n$ is a Hilbert Space equipped with any $W^{k,2}(\mathscr{O};\R^2)$ for $k = 0, 1, 2, 3$). That is, $\bar{\mathcal{P}}_n$ is given by $$\bar{\mathcal{P}}_n : f \mapsto \sum_{k=1}^n \inner{f}{\bar{a}_k}\bar{a}_k.$$ Note that $\bar{\mathcal{P}}_n$ is also self-adjoint in $\bar{W}^{1,2}_{\sigma}$ as the system $(\bar{a}_k)$ forms an orthogonal basis in this space. It will also be of use to us to consider the vorticity in this context, which we do by introducing the operator $\textnormal{curl}: W^{1,2}(\mathscr{O};\R^2) \rightarrow L^2(\mathscr{O};\R)$ by $$\textnormal{curl}: f \rightarrow \partial_1 f^2 - \partial_2 f^1.$$
A significant property in the study of vorticity is the following.
\begin{lemma} \label{lemma for curl and P}
    For all $f \in W^{1,2}(\mathscr{O};\R^2)$, $\textnormal{curl}(\mathcal{P}f) = \textnormal{curl}f$.
\end{lemma}
\begin{proof}
     If $\phi = \nabla g \in W^{1,2}(\mathscr{O};\R^2)$ then $\textnormal{curl}\phi = \partial_1\partial_2g - \partial_2\partial_1g = 0$, which from Remark \ref{new first labelled remark} establishes that $\textnormal{curl}\left(\mathcal{P}f\right) = \textnormal{curl}\left([\mathcal{P} + \mathcal{P}^{\perp}]f \right) = \textnormal{curl}f$ where $\mathcal{P}^{\perp}$ is the complement projection $I - \mathcal{P}$ on $L^2(\mathscr{O};\R^2)$.
\end{proof}
The curl is intrinsically related to the Navier boundary conditions through $\kappa$, by the relation proven as Lemma 2.1 in [\cite{clopeau1998vanishing}] which we state here.
\begin{lemma}
    For all $f \in W^{2,2}(\mathscr{O};\R^2)$ with $f \cdot \mathbf{n} = 0$ on $\partial \mathscr{O}$, we have that $$2(Df)\mathbf{n} \cdot \iota +2\kappa f \cdot \iota - \textnormal{curl}f = 0$$ on $\partial \mathscr{O}$. 
\end{lemma}

Moreover we can make the condition (\ref{its a new rep}) discussed in the introduction precise.

\begin{corollary} \label{vorticity corollary}
    Suppose that $f \in W^{2,2}(\mathscr{O};\R^2)$ with $f \cdot \mathbf{n} = 0$ on $\partial \mathscr{O}$. Then $f$ satisfies $$2(Df)\mathbf{n} \cdot \mathbf{\iota} + \alpha f\cdot \mathbf{\iota} = 0$$ on $\partial \mathscr{O}$ if and only if it satisfies $$ \textnormal{curl}f = (2\kappa - \alpha)f \cdot \iota$$ on $\partial \mathscr{O}$.
\end{corollary}

 Returning to the nonlinear term, we shall frequently understand the $L^2(\mathscr{O};\R^2)$ inner product as a duality pairing between $L^{\frac{4}{3}}(\mathscr{O};\R^2)$ and $L^4(\mathscr{O};\R^2)$ as justified in the following.

\begin{lemma} \label{the 2D bound lemma}
There exists a constant $C$ such that for every $\phi,f,g, \in W^{1,2}(\mathscr{O};\R^2)$, we have that $$ \left\vert  \inner{\mathcal{L}_{\phi}f}{g} \right\vert \leq C\norm{\phi}^{\frac{1}{2}}\norm{\phi}_1^{\frac{1}{2}}\norm{f}_1\norm{g}^{\frac{1}{2}}\norm{g}_1^{\frac{1}{2}}.$$
\end{lemma}

\begin{proof}
   From two instances of H\"{o}lder's Inequality as well as Theorem \ref{gagliardonirenberginequality} with $p=4, q=2, \alpha = 1/2$ and $m=1$, 
  \begin{align} \nonumber
       \left\vert\inner{\mathcal{L}_{\phi}f}{g}\right\vert \leq \norm{\mathcal{L}_{\phi}f}_{L^{4/3}}\norm{g}_{L^4} &\leq c\sum_{k=1}^2\norm{\phi}_{L^4}\norm{\partial_kf}\norm{g}_{L^4} \leq c\norm{\phi}^{\frac{1}{2}}\norm{\phi}_1^{\frac{1}{2}}\norm{f}_1\norm{g}^{\frac{1}{2}}\norm{g}_1^{\frac{1}{2}}\label{a bound in align}
    \end{align}
\end{proof}

\begin{remark}
    This inequality is critical in the study of 2D Navier-Stokes, its failure in 3D responsible for the lack of global strong solutions and uniqueness of weak solutions.
\end{remark}

This subsection concludes with a symmetry result for the trilinear form defined by the nonlinear term which is classical when zero trace is assumed, but perhaps not in lieu of this assumption.

\begin{lemma} \label{navier boundary nonlinear}
    For every $\phi \in \bar{W}^{1,2}_{\sigma}$, $f,g \in W^{1,2}(\mathscr{O};\R^2)$ we have that \begin{equation}\label{wloglhs}\inner{\mathcal{L}_{\phi}f}{g}= -\inner{f}{\mathcal{L}_{\phi}g}.\end{equation}
    Moreover, \begin{equation} \label{cancellationproperty'} \inner{\mathcal{L}_{\phi}f}{f}= 0.\end{equation}
\end{lemma}

\begin{proof}
    Of course (\ref{cancellationproperty'}) follows from (\ref{wloglhs}) with the choice $g:= f$ and symmetry of the inner product, so we just show (\ref{wloglhs}). As stated it is classical that the result holds in the case that $\phi$ has zero trace (see e.g. [\cite{robinson2016three}] Lemma 3.2) by an approximation in $W^{1,2}(\mathscr{O};\R^2)$ of compactly supported functions, though without that we have to do the integration by parts directly:
    \begin{align*}
    \inner{\mathcal{L}_{\phi} f}{g} &= \sum_{j=1}^2\sum_{l=1}^2\inner{\phi^j\partial_jf^l}{g^l}_{L^2(\mathscr{O};\R)}\\
    &= \sum_{j=1}^2\sum_{l=1}^2\left(\inner{\phi^j\partial_jf^l}{g^l}_{L^2(\mathscr{O};\R)} + \inner{\partial_j\phi^jf^l}{g^l}_{L^2(\mathscr{O};\R)}\right)\\
    &= \sum_{j=1}^2\sum_{l=1}^2\inner{\partial_j(\phi^jf^l)}{g^l}_{L^2(\mathscr{O};\R)}\\
    &= -\sum_{j=1}^2\sum_{l=1}^2\inner{\phi^jf^l}{\partial_jg^l}_{L^2(\mathscr{O};\R)} + \sum_{j=1}^2\sum_{l=1}^2\inner{\phi^jf^l}{g^l\mathbf{n}^j}_{L^2(\partial \mathscr{O};\R)}\\
    &= -\inner{f}{\mathcal{L}_{\phi} g}
\end{align*}
    where we have used that $\sum_{j=1}^2\partial_j \phi^j = 0$ (divergence-free) in $\mathscr{O}$ and $\sum_{j=1}^2\phi^j \mathbf{n}^j = 0$ ($\phi \cdot \mathbf{n}=0$) on $\partial \mathscr{O}$.

\end{proof}

\subsection{Assumptions on the Noise} \label{subsection assumptions}

With the framework established we now properly introduce the stochastic Navier-Stokes equation

\begin{equation} \label{projected Ito}
    u_t = u_0 - \int_0^t\mathcal{P}\mathcal{L}_{u_s}u_s\ ds + \nu\int_0^t \mathcal{P} \Delta u_s\, ds + \frac{1}{2}\int_0^t\sum_{i=1}^\infty \mathcal{P}\mathcal{Q}_i^2u_s ds - \int_0^t \mathcal{P}\mathcal{G}u_s d\mathcal{W}_s 
\end{equation}
where $\mathcal{Q}_i$ is either $\mathcal{P}\mathcal{G}_i$ or $0$, satisfying assumptions to be stated in this subsection. The case $\mathcal{Q}_i = 0$ leaves us with the projected form of (\ref{number2equation}) whilst $\mathcal{Q}_i=\mathcal{P}\mathcal{G}_i$ corresponds to (\ref{number3equation}) via taking the Leray Projection and then converting to It\^{o} Form. This conversion is rigorously justified in [\cite{goodair2022stochastic}] Subsection 2.3. In the case where $\mathcal{P}\mathcal{G}_i^2 = (\mathcal{P}\mathcal{G}_i )^2$ then we can instead take $\mathcal{Q}_i = \mathcal{G}_i$ as the resulting equation (\ref{projected Ito}) is the same; this is the case for SALT noise, discussed in the next subsection. The key definitions and results regarding the existence and uniqueness of solutions are given in Subsection \ref{sub def}. We impose the existence of some $p,q,r \in \R$ and constants $(c_i)$ such that for all $f, g \in L^2_{\sigma} \cap W^{1,2}(\mathscr{O};\R^2)$, $\phi,\psi \in \bar{W}^{2,2}_{\alpha} \cap W^{3,2}(\mathscr{O};\R^2)$, defining $\tilde{K}_2(f,g):= 1 + \norm{f}^p + \norm{g}^q + \norm{f}_{W^{1,2}}^2 + \norm{g}_{W^{1,2}}^2$, $\forall \varepsilon > 0$,
\begin{align}
   \label{assumpty 1}  \norm{\mathcal{G}_if}^2 &\leq c_i\left(1 + \norm{f}_{W^{1,2}}^2\right)\\
   \label{assumpty 4}  \norm{\mathcal{G}_if - \mathcal{G}_ig}^2 &\leq c_i\left[1 + \norm{f}_{W^{1,2}}^p + \norm{g}_{W^{1,2}}^q\right]\norm{f-g}_{W^{1,2}}^2\\
   \label{assumpty 2} \norm{\mathcal{G}_i\phi}_{W^{1,2}}^2 &\leq c_i\left(1 + \norm{\phi}_2^2\right)\\
   \label{assumpty 11} \norm{\mathcal{Q}_i\phi}_{W^{2,2}}^2 &\leq c_i\left(1 + \norm{\phi}_{W^{3,2}}^2\right)\\
    \label{assumpty 7} \inner{\mathcal{Q}_i^2\phi }{\phi} + \norm{\mathcal{G}_i\phi}^2 &\leq c_i\left(1 + \norm{\phi}^2\right) + k_i\norm{\phi}_1^2\\
\label{assumpty 12} \inner{\mathcal{Q}_i^2\phi }{A\phi} + \norm{\mathcal{P}\mathcal{G}_i\phi}_1^2 &\leq c_i\left(1 + \norm{\phi}_{1}^2\right) + k_i\norm{\phi}_2^2\\
    \label{assumpty8} \inner{\mathcal{G}_if}{f}^2 &\leq c_i\left(1 + \norm{f}^4\right)\\
    \label{assumpty13} \inner{\mathcal{G}_i\phi}{A\phi}^2 &\leq C_{\varepsilon}c_i\left(1 + \norm{\phi}_1^6\right) + c_i\varepsilon\norm{\phi}^2_2\\
    \label{assumpty9} \inner{\mathcal{G}_if}{g}^2 &\leq c_i\left[1 + \norm{f}^2 + \norm{g}^p\right]\norm{g}_{W^{1,2}}^2\\
    \label{assumpty10} \inner{\mathcal{G}_if - \mathcal{G}_ig}{\phi}^2 &\leq  c_i\left[1  + \norm{\phi}_2^p\right]\norm{f-g}^2\\
    \label{assumpty 6}
\inner{\mathcal{G}_if - \mathcal{G}_ig}{f-g}^2 &\leq c_i\tilde{K}_2(f,g)\norm{f -g}^4
\end{align}
where $\sum_{i=1}^\infty c_i < \infty$, $\sum_{i=1}^\infty k_i \leq \nu$\footnote{Actually, we only need that $\sum_{i=1}^\infty k_i < 2\nu$. Our choice just avoids heavy notation in the proofs.}, $C_{\varepsilon}$ depends on $\varepsilon$ and all constants may depend on $\alpha$ and $\nu$. We only consider this noise in the case of a fixed $\nu$ so this dependency is not meaningful. For each $i \in \N$, $\mathcal{Q}_i$ must be linear and possess a densely defined adjoint $\mathcal{Q}_i^*$ in $L^2(\mathscr{O};\R^2)$ with domain of definition $W^{1,2}(\mathscr{O};\R^2)$ where
\begin{align}
   \label{assumpty 3} \norm{\mathcal{Q}_i^*f}^2 &\leq c_i\norm{f}_{W^{1,2}}^2\\ 
   \inner{\mathcal{Q}_i(f - g)}{\mathcal{Q}_i^*(f-g)} + \norm{\mathcal{G}_if - \mathcal{G}_ig}^2 &\leq c_i\tilde{K}_2(f,g)\norm{f -g}^2 + k_i\norm{f -g}^2_{W^{1,2}}. \label{assumpty 5}
\end{align}



\subsection{Stochastic Advection by Lie Transport} \label{sub salt}

We also present both the Navier-Stokes and Euler Equations under Stochastic Advection by Lie Transport (SALT), given by

\begin{equation} \label{projected Ito Salt}
    u_t = u_0 - \int_0^t\mathcal{P}\mathcal{L}_{u_s}u_s\ ds + \nu\int_0^t \mathcal{P} \Delta u_s\, ds + \frac{1}{2}\int_0^t\sum_{i=1}^\infty \mathcal{P}B_i^2u_s ds - \int_0^t \mathcal{P}Bu_s d\mathcal{W}_s 
\end{equation}
and
\begin{equation} \label{projected Ito Salt Euler}
    u_t = u_0 - \int_0^t\mathcal{P}\mathcal{L}_{u_s}u_s\ ds  + \frac{1}{2}\int_0^t\sum_{i=1}^\infty \mathcal{P}B_i^2u_s ds - \int_0^t \mathcal{P}Bu_s d\mathcal{W}_s 
\end{equation}
respectively for the operator $$B_i:f \mapsto \mathcal{L}_{\xi_i}f + \mathcal{T}_{\xi_i}f, \qquad  \mathcal{T}_{g}f := \sum_{j=1}^2 f^j\nabla g^j$$ where $\xi_i \in L^2_{\sigma} \cap W^{3,2}_0(\mathscr{O};\R^2) \cap W^{3,\infty}(\mathscr{O};\R^2)$ such that $\sum_{i=1}^\infty \norm{\xi_i}^2_{W^{3,\infty}} < \infty$. The equivalence between (\ref{projected Ito Salt}) and the Stratonovich form (\ref{number2equationSALT}) is rigorously proven in [\cite{goodair2022navier}] Subsection 2.1. Fundamental properties of the operator $B_i$ are proven in [\cite{goodair2022navier}] Subsection 1.4, including the fact that $\mathcal{P}B_i = \mathcal{P}B_i\mathcal{P}$ (Lemma 1.28) hence $\mathcal{P}B_i^2 = (\mathcal{P}B_i)^2$ and we can take $\mathcal{G}_i = B_i = \mathcal{Q}_i$. We note that (\ref{projected Ito Salt}) is of the form (\ref{projected Ito}) for $\mathcal{G}_i = \mathcal{Q}_i = B_i$. The majority of the assumptions in Subsection \ref{subsection assumptions} have been verified in [\cite{goodair2023zero}] Subsection 1.4, though conditions pertaining to the $W^{1,2}$ inner product, that is (\ref{assumpty 12}) and (\ref{assumpty13}), are new and must be shown precisely. This is done in Subsection \ref{sub lie transport}.

\subsection{Notions of Solution and Main Results} \label{sub def}

We fix an arbitrary $T>0$ and give the definition of a weak solution of the equation (\ref{projected Ito}).

\begin{definition} \label{definitionofspatiallyweak}
Let $u_0: \Omega \rightarrow L^2_{\sigma}$ be $\mathcal{F}_0-$measurable and $\alpha \in C^2(\partial \mathscr{O}; \R)$. A process $u$ which is progressively measurable in $\bar{W}^{1,2}_{\sigma}$ and such that for $\mathbb{P}-a.e.$ $\omega$, $u_{\cdot}(\omega) \in C\left([0,T];L^2_{\sigma}\right) \cap L^2\left([0,T];\bar{W}^{1,2}_{\sigma}\right)$, is said to be a weak solution of the equation (\ref{projected Ito}) if the identity
\begin{align} \nonumber
     \inner{u_t}{\phi} = \inner{u_0}{\phi} - \int_0^{t}\inner{\mathcal{L}_{u_s}u_s}{\phi}ds &- \nu\int_0^{t} \inner{u_s}{\phi}_1 ds + \nu\int_0^t \inner{(\kappa - \alpha)u_s}{\phi}_{L^2(\partial \mathscr{O}; \R^2)}ds\\ &+ \frac{1}{2}\int_0^{t}\sum_{i=1}^\infty \inner{\mathcal{Q}_iu_s}{\mathcal{Q}_i^*\phi} ds - \int_0^{t} \inner{\mathcal{G}u_s}{\phi} d\mathcal{W}_s\label{identityindefinitionofspatiallyweak}
\end{align}
holds for every $\phi \in \bar{W}^{1,2}_{\sigma}$, $\mathbb{P}-a.s.$ in $\R$ for all $t\in[0,T]$.
\end{definition}

We briefly note that from Lemma \ref{the 2D bound lemma}, (\ref{assumpty 1}), the continuity of the trace mapping of $W^{1,2}(\mathscr{O};\R^2)$ into $L^2(\partial \mathscr{O};\R^2)$ and boundedness of $\kappa, \alpha$ on $\partial \mathscr{O}$ that these integrals are well defined. It is likely less clear why this is a correct weak formulation of the problem, which we address in the following lemma.

\begin{definition} \label{definitionofstrongsolution}
Let $u_0: \Omega \rightarrow \bar{W}^{1,2}_{\sigma}$ be $\mathcal{F}_0-$measurable and $\alpha \in C^2(\partial \mathscr{O}; \R)$. A process $u$ which is progressively measurable in $\bar{W}^{2,2}_{\alpha}$ and such that for $\mathbb{P}-a.e.$ $\omega$, $u_{\cdot}(\omega) \in C\left([0,T];\bar{W}^{1,2}_{\sigma}\right) \cap L^2\left([0,T];\bar{W}^{2,2}_{\alpha}\right)$, is said to be a strong solution of the equation (\ref{projected Ito}) if the identity
   \begin{equation} \label{a new id}
    u_t = u_0 - \int_0^t\mathcal{P}\mathcal{L}_{u_s}u_s\ ds - \nu\int_0^t A u_s\, ds + \frac{1}{2}\int_0^t\sum_{i=1}^\infty \mathcal{P}\mathcal{Q}_i^2u_s ds - \int_0^t \mathcal{P}\mathcal{G}u_s d\mathcal{W}_s 
\end{equation}
holds $\mathbb{P}-a.s.$ in $L^2_{\sigma}$ for all $t\in[0,T]$.
\end{definition}

\begin{lemma} \label{lemma for strong is weak}
    Let $u_0: \Omega \rightarrow L^2_{\sigma}$ be $\mathcal{F}_0-$measurable, $\alpha \in C^2(\partial \mathscr{O}; \R)$ and $u$ be progressively measurable in $\bar{W}^{2,2}_{\alpha}$ such that for $\mathbb{P}-a.e.$ $\omega$, $u_{\cdot}(\omega) \in L^{\infty}\left([0,T];\bar{W}^{1,2}_{\sigma}\right) \cap L^2\left([0,T];\bar{W}^{2,2}_{\alpha}\right)$. Then the following are equivalent:
    \begin{enumerate}
        \item \label{firstestitem} The identity (\ref{a new id}) holds $\mathbb{P}-a.s.$ in $L^2_{\sigma}$ for all $t\in[0,T]$.
        \item \label{secondestitem} The identity (\ref{identityindefinitionofspatiallyweak}) holds for every $\phi \in \bar{W}^{1,2}_{\sigma}$, $\mathbb{P}-a.s.$ in $\R$ for all $t\in[0,T]$.
    \end{enumerate}
\end{lemma}

\begin{proof}
    We prove each implication of \ref{firstestitem} $\iff$ \ref{secondestitem} in turn:
    \begin{itemize}
        \item[$\implies$:] By taking the inner product on both sides of (\ref{projected Ito}) with any such $\phi$ and using the adjoint of $\mathcal{P}$ and $\mathcal{Q}_i$, we only need to prove that $$\inner{\Delta u_s}{\phi}  = -\inner{u_s}{\phi}_1 + \inner{(\kappa - \alpha)u_s}{\phi}_{L^2(\partial \mathscr{O}; \R^2)}$$ for $\lambda-a.e.$ $s\in[0,t]$, which is done if we verfiy this for every $s$ such that $u_s \in \bar{W}^{2,2}_{\alpha}$. This of course follows from Lemma \ref{greens for navier}. It should be noted that we have taken the inner product through the stochastic integral as $\inner{\cdot}{\phi}$ is a bounded linear operator, a property which we shall use freely, recalling again [\cite{prevot2007concise}] Lemma 2.4.1.
        \item[$\impliedby$:] Via the identical process in reverse, we obtain that \begin{align} \nonumber
     \inner{u_t}{\phi} = \inner{u_0}{\phi} - \inner{\int_0^{t}\mathcal{P}\mathcal{L}_{u_s}u_sds}{\phi} &+ \nu\inner{\int_0^{t}\mathcal{P} \Delta u_sds}{\phi} \\ &+ \frac{1}{2}\inner{\int_0^{t}\sum_{i=1}^\infty \mathcal{P}\mathcal{Q}_i^2u_sds}{\phi}  - \inner{\int_0^{t} \mathcal{P}\mathcal{G}u_sd\mathcal{W}_s}{\phi} \nonumber
\end{align}
for every $\phi \in \bar{W}^{1,2}_{\sigma}$, from which the density of this space in $L^2_{\sigma}$ gives the result. 
    \end{itemize}
\end{proof}

\begin{remark}
    An alternative weak formulation is used in [\cite{clopeau1998vanishing}, \cite{filho2005inviscid}]. Their equivalence is proved in [\cite{kelliher2006navier}] after Definition 5.1.
\end{remark}

The existence of a strong solution would then follow from proving the existence of a weak solution and that this process has the regularity required to be a strong solution; we use this fact in Subsection \ref{sub strong sols}.

\begin{definition}
    A solution of the equation (\ref{projected Ito}) is said to be unique if for any other such solution $v$, $$ \mathbb{P}\left(\left\{\omega \in \Omega: u_t(\omega) = v_t(\omega) \quad \forall t \geq 0\right\}\right) = 1.$$
\end{definition}


\begin{theorem} \label{theorem 2D}
   For any given $\mathcal{F}_0-$measurable $u_0: \Omega \rightarrow L^2_{\sigma}$ and $\alpha \in C^2(\partial \mathscr{O}; \R)$, there exists a unique weak solution $u$ of the equation (\ref{projected Ito}).
\end{theorem}

Theorem \ref{theorem 2D} will be proved in Subsection \ref{sub weak sols}. We now state the first main result of this paper. 

\begin{theorem} \label{theorem 2D strong}
 For any given $\mathcal{F}_0-$measurable $u_0: \Omega \rightarrow \bar{W}^{1,2}_{\sigma}$ and $\alpha \in C^2(\partial \mathscr{O}; \R)$ such that $\alpha \geq \kappa$ on $\partial \mathscr{O}$, there exists a unique strong solution $u$ of the equation (\ref{projected Ito}).
\end{theorem}

Thereom \ref{theorem 2D strong} will be proved in Subsection \ref{sub strong sols}. As for the Euler Equation and inviscid limit, our results are as follows.

\begin{definition} \label{definitionofspacetimeweakmartingale}
Let $u_0: \Omega \rightarrow \bar{W}^{1,2}_{\sigma}$ be $\mathcal{F}_0-$measurable. If there exists a filtered probability space $\left(\tilde{\Omega},\tilde{\mathcal{F}},(\tilde{\mathcal{F}}_t), \tilde{\mathbbm{P}}\right)$, a cylindrical Brownian Motion $\tilde{\mathcal{W}}$ over $\mathfrak{U}$ with respect to $\left(\tilde{\Omega},\tilde{\mathcal{F}},(\tilde{\mathcal{F}}_t), \tilde{\mathbbm{P}}\right)$, an $\tilde{\mathcal{F}_0}-$measurable $\tilde{u}_0: \tilde{\Omega} \rightarrow \bar{W}^{1,2}_{\sigma}$ with the same law as $u_0$, and a progressively measurable process $\tilde{u}$ in $W^{1,2}_{\sigma}$ such that for $\tilde{\mathbb{P}}-a.e.$ $\tilde{\omega}$, $\tilde{u}_{\cdot}(\omega) \in L^{\infty}\left([0,T];\bar{W}^{1,2}_{\sigma}\right)\cap C\left([0,T];L^2_{\sigma}\right)$ and
\begin{align} 
     \inner{\tilde{u}_t}{\phi} = \inner{\tilde{u}_0}{\phi} - \int_0^{t}\inner{\mathcal{L}_{\tilde{u}_s}\tilde{u}_s}{\phi}ds + \frac{1}{2}\int_0^{t}\sum_{i=1}^\infty \inner{B_i\tilde{u}_s}{B_i^*\phi} ds - \int_0^{t} \inner{B\tilde{u}_s}{\phi} d\tilde{\mathcal{W}}_s\label{newid1}
\end{align}
holds for every $\phi \in \bar{W}^{1,2}_{\sigma}$ $\tilde{\mathbb{P}}-a.s.$ in $\R$ for all $t \in [0,T]$, then $\tilde{u}$ is said to be a martingale weak solution of the equation (\ref{projected Ito Salt Euler}).
\end{definition}

\begin{theorem} \label{theorem for limit of navier stokes is euler}
    Let $\kappa \geq 0$ everywhere on $\partial \mathscr{O}$. For any given $\mathcal{F}_0-$measurable $u_0 \in L^2\left(\Omega; \bar{W}^{1,2}_{\sigma}\right)$, there exists a martingale weak solution $\tilde{u}$ of the equation (\ref{projected Ito Salt Euler}). Moreover for any sequence of viscosities $(\nu^k)$ monotonically decreasing to zero, with corresponding strong solutions of the equation (\ref{projected Ito Salt}) for $\alpha:= 2\kappa$ given by $(\tilde{u}^k)$, then as $k \rightarrow \infty$, there is a subsequence indexed by $n_k$ such that  $\tilde{u}^{n_k} \longrightarrow \tilde{u}$  $\mathbbm{P}-a.s.$ in $C\left([0,T];L^2_{\sigma}\right)$. 
\end{theorem}

\begin{remark}
    For each viscosity $\nu^k$ and for any given probability space and cylindrical Brownian Motion, there is a strong solution of the equation (\ref{projected Ito Salt}) (for $\alpha:= 2\kappa)$ given by Theorem \ref{theorem 2D strong}. In the statement of Theorem \ref{theorem for limit of navier stokes is euler}, $\tilde{u}^k$ is taken to be the strong solution relative to the probability space and cylindrical Brownian Motion specified in the definition of the weak martingale solution $\tilde{u}$.
\end{remark}

Theorem \ref{theorem for limit of navier stokes is euler} will be proved in Section \ref{section zero viscy limit}.

\section{Solutions of the Stochastic Navier-Stokes Equation} \label{section solutions sns}

We note that $\nu$ remains fixed in this section, so the dependency on $\nu$ in our constants will not be made explicit.

\subsection{Weak Solutions} \label{sub weak sols}

We begin to sketch the proof of Theorem \ref{theorem 2D} by following the same process as in [\cite{goodair2023zero}] Section 3, making reference to the appropriate changes. We fix an $\mathcal{F}_0-$measurable $u_0 \in L^\infty\left(\Omega; L^2_{\sigma}\right)$ and work with a Galerkin Approximation, doing so by considering (\ref{projected Ito}) in its spatially strong form, projected by $\bar{\mathcal{P}}_n$ to give
\begin{equation} \label{projected Ito galerkin}
    u^n_t = u^n_0 - \int_0^t\bar{\mathcal{P}}_n\mathcal{P}\mathcal{L}_{u^n_s}u^n_s\ ds + \nu\int_0^t \bar{\mathcal{P}}_n \mathcal{P} \Delta u^n_s\, ds + \frac{1}{2}\int_0^t\sum_{i=1}^\infty \bar{\mathcal{P}}_n \mathcal{P}\mathcal{Q}_i^2u^n_s ds - \int_0^t \bar{\mathcal{P}}_n\mathcal{P}\mathcal{G}u^n_s d\mathcal{W}_s 
\end{equation}
where $u^n_0:=\bar{\mathcal{P}}_nu_0$. There is no change to obtaining a maximal strong solution $(u^n, \Theta^n)$ and by defining the stopping times
\begin{equation}\label{the stopping times}\tau^{M,S}_n := S \wedge \inf \left\{s \geq 0: \sup_{r \in [0,s]}\norm{u^n_r}^2 + \int_0^s \norm{u^n_r}_1^2dr \geq M + \norm{u^n_0}^2 \right\}\end{equation} we claim that there exists a constant $C$ independent of $M,n,\nu$ such that
    \begin{equation}\label{first result}\mathbbm{E}\left[\sup_{r \in [0,T \wedge  \tau^{M,T+1}_n]}\norm{u^n_r}^2 + \nu\int_0^{T \wedge  \tau^{M,T+1}_n}\norm{u^n_r}_1^2dr  \right] \leq C\left[\mathbbm{E}\left(\norm{u^n_0}^2\right) + 1\right].\end{equation}
Recalling Lemma \ref{navier boundary nonlinear} then the only difference is in control of the Laplacian term, which after using the self-adjoint property of the projections appears simply as $$\inner{\Delta u^n_s }{u^n_s}.$$ In [\cite{goodair2023zero}] the approximation occurs in a space of zero trace functions so this is simply the desired $-\norm{u^n_s}_1^2$ term. In this context, as justified in Lemma \ref{greens for navier}, we have to control the additional term $$\inner{(\kappa - \alpha)u_s}{u_s}_{L^2(\partial \mathscr{O};\R^2)}$$ which we do by observing that $$\inner{(\kappa - \alpha)u_s}{u_s}_{L^2(\partial \mathscr{O};\R^2)} \leq \norm{(\kappa - \alpha)u_s}_{L^2(\partial \mathscr{O};\R^2)}\norm{u_s}_{L^2(\partial \mathscr{O};\R^2)} \leq \norm{\kappa - \alpha}_{L^{\infty}(\partial\mathscr{O};\R)}\norm{u_s}_{L^2(\partial \mathscr{O};\R^2)}^2.$$ By allowing a constant $c$ to depend on $\kappa, \alpha$, which is not meaningful to our analysis at this stage as these quantities remain fixed, we can employ (\ref{inequality from Lions}) and Young's Inequality to control this further with  $$c\norm{u_s}^2 + \varepsilon \norm{u_s}_1^2$$ for any $\varepsilon > 0$. Incorporating these additional terms into the running inequality of [\cite{goodair2023zero}] Proposition 3.1 justifies (\ref{first result}). The result implies that the maximal strong solution $(u^n, \Theta^n)$ is in fact global and satisfies the energy inequality  \begin{equation}\label{better first result}\mathbbm{E}\left[\sup_{r \in [0,T]}\norm{u^n_r}^2 + \nu\int_0^{T }\norm{u^n_r}_1^2dr  \right] \leq C\left[\mathbbm{E}\left(\norm{u^n_0}^2\right) + 1\right]\end{equation} as obtained in [\cite{goodair2023zero}] equation (57), so we move on to the tightness considerations. We verify that the sequence of the laws of $(u^n)$ is tight in the space of probability measures over $L^2\left([0,T];L^2_{\sigma}\right)$ and $\mathcal{D}\left([0,T];\left(\bar{W}^{1,2}_{\sigma}\right)^*\right)$. The tightness in $L^2\left([0,T];L^2_{\sigma}\right)$ follows [\cite{goodair2023zero}] Proposition 3.2, whilst the result for $\mathcal{D}\left([0,T];\left(\bar{W}^{1,2}_{\sigma}\right)^*\right)$ follows Proposition 3.3. In the former case, the only difference is in the term $$\inner{\Delta \hat{u}^n_r}{\hat{u}^n_r- \hat{u}^n_s} = - \inner{\hat{u}^n_r}{\hat{u}^n_r - \hat{u}^n_s}_1 + \inner{(\kappa - \alpha)\hat{u}^n_r}{\hat{u}^n_r - \hat{u}^n_s}_{L^2(\partial \mathscr{O};\R^2)}$$ where the boundary integral needs to be controlled. The contribution of $\inner{(\kappa - \alpha)\hat{u}^n_r}{\hat{u}^n_r}_{L^2(\partial \mathscr{O};\R^2)}$ is handled as above, and for $-\inner{(\kappa - \alpha)\hat{u}^n_r}{\hat{u}^n_s}_{L^2(\partial \mathscr{O};\R^2)}$ only a bound by $c\norm{\hat{u}^n_r}_1\norm{\hat{u}^n_s}_1$ is required as in (64), so the result follows. The same bound is all that is required in Proposition 3.3, so the tightness in $\mathcal{D}\left([0,T];\left(\bar{W}^{1,2}_{\sigma}\right)^*\right)$ follows similarly. Comparing to Theorem 3.5, Lemma 3.6 and Proposition 3.7, we have the following. We note the need to pass to a subsequence, which we again index by $n$ for simplicity.

\begin{theorem} \label{theorem for new prob space}
There exists a filtered probability space $\left(\tilde{\Omega},\tilde{\mathcal{F}},(\tilde{\mathcal{F}}_t), \tilde{\mathbb{P}}\right)$, a cylindrical Brownian Motion $\tilde{\mathcal{W}}$ over $\mathfrak{U}$ with respect to $\left(\tilde{\Omega},\tilde{\mathcal{F}},(\tilde{\mathcal{F}}_t), \tilde{\mathbb{P}}\right)$, a sequence of random variables $(\tilde{u}^n_0)$, $u^n_0: \tilde{\Omega} \rightarrow L^2_{\sigma}$ and a $\tilde{u}_0:\tilde{\Omega} \rightarrow L^2_{\sigma}\left(\mathscr{O};\R^N\right)$, a sequence of processes $(\tilde{u}^n)$, $\tilde{u}^n:\tilde{\Omega} \times [0,T] \rightarrow \bar{W}^{1,2}_{\sigma}$ is progressively measurable and a progressively measurable process $\tilde{u}:\tilde{\Omega} \times [0,T] \rightarrow \bar{W}^{1,2}_{\sigma}$ such that:
\begin{enumerate}
    \item For each $n \in \N$, $\tilde{u}^n_0$ has the same law as $u^n_0$;
    \item \label{new item 2} For $\tilde{\mathbb{P}}-a.e.$ $\omega$, $\tilde{u}^n_0(\omega) \rightarrow \tilde{u}_0(\omega)$  in $L^2_{\sigma}$, and thus $\tilde{u}_0$ has the same law as $u_0$;
    \item \label{new item 3} For each $n \in \N$ and $t\in[0,T]$, $\tilde{u}^n$ satisfies the identity
    \begin{equation} \nonumber
    \tilde{u}^n_t = \tilde{u}^n_0 - \int_0^t\bar{\mathcal{P}}_n\mathcal{P}\mathcal{L}_{\tilde{u}^n_s}\tilde{u}^n_s\ ds - \nu\int_0^t \bar{\mathcal{P}}_n A \tilde{u}^n_s\, ds + \frac{1}{2}\int_0^t\sum_{i=1}^\infty \bar{\mathcal{P}}_n \mathcal{P}\mathcal{Q}_i^2\tilde{u}^n_s ds - \int_0^t \bar{\mathcal{P}}_n\mathcal{P}\mathcal{G}\tilde{u}^n_s d\tilde{\mathcal{W}}_s 
\end{equation}
$\tilde{\mathbb{P}}-a.s.$ in $\bar{V}_n$;
\item For $\tilde{\mathbb{P}}-a.e$ $\omega$, $\tilde{u}^n(\omega) \rightarrow \tilde{u}(\omega)$ in $L^2\left([0,T]; L^2_{\sigma} \right)$ and $\mathcal{D}\left([0,T];\left(W^{1,2}_{\sigma}\right)^* \right)$; \label{new item 4}
\item $\tilde{u}^n \rightarrow \tilde{u}$ in $L^2\left(\tilde{\Omega};L^2\left([0,T]; L^2_{\sigma} \right)\right)$;
\item For $\tilde{\mathbb{P}}-a.e.$ $\omega$, $\tilde{u}_{\cdot}(\omega) \in L^{\infty}\left([0,T];L^2_{\sigma}\right)\cap L^2\left([0,T];\bar{W}^{1,2}_{\sigma}\right)$.
\end{enumerate}

\end{theorem}

In order to show that $\tilde{u}$ is a \textit{martingale} weak solution, a concept not defined here but used in [\cite{goodair2023zero}], it only remains to verify the following proposition.

\begin{proposition} \label{prop for first energy bar}
    The identity
\begin{align} \nonumber
     \inner{\tilde{u}_t}{\phi} = \inner{u_0}{\phi} - \int_0^{t}\inner{\mathcal{L}_{\tilde{u}_s}\tilde{u}_s}{\phi}ds &- \nu\int_0^{t} \inner{\tilde{u}_s}{\phi}_1 ds + \nu\int_0^t \inner{(\kappa - \alpha)\tilde{u}_s}{\phi}_{L^2(\partial \mathscr{O}; \R^2)}ds\\ &+ \frac{1}{2}\int_0^{t}\sum_{i=1}^\infty \inner{\mathcal{Q}_i\tilde{u}_s}{\mathcal{Q}_i^*\phi} ds - \int_0^{t} \inner{\mathcal{G}\tilde{u}_s}{\phi} d\mathcal{W}_s\nonumber
\end{align}
holds for every $\phi \in \bar{W}^{1,2}_{\sigma}$, $\mathbb{P}-a.s.$ in $\R$ for all $t\in[0,T]$.
\end{proposition}

\begin{proof}
    The proof again parallels [\cite{goodair2023zero}], Proposition 3.8, but with some distinct differences outlined here. As such we fix a $t \in [0,T]$ and $\phi \in \bar{W}^{1,2}_{\sigma}$, but now consider arbitrary $\psi \in \bigcup_k \bar{V}_k$ and look to show the identity with $\psi$ before using the density of $\bigcup_k \bar{V}_k$ in $\bar{W}^{1,2}_{\sigma}$ to achieve the result. The significance here is that any fixed $\psi$ belongs to $\bar{V}_k$ for some $k$, so observing $\tilde{u}$ as a limit of $(\tilde{u}^n)$ we can consider only $n > k$. As such, all of the $(I - \bar{\mathcal{P}}_n)\psi$ terms that would appear here in analogy of [\cite{goodair2023zero}] Proposition 3.8 are simply zero. With this in place there is nothing further to prove for any term other than for the Stokes operator. The difference of the associated terms is given by \begin{equation}\label{associated diff}\inner{\Delta \tilde{u}^n_s}{\psi} - \left( -\inner{\tilde{u}_s}{\psi}_1 + \inner{(\kappa - \alpha)\tilde{u}_s}{\phi}_{L^2(\partial \mathscr{O}; \R^2)} \right)\end{equation} and using Lemma \ref{greens for navier}, $$\inner{\Delta \tilde{u}^n_s}{\psi} = - \inner{\tilde{u}^n_s}{\psi}_1 + \inner{(\kappa - \alpha)\tilde{u}^n_s}{\psi}_{L^2(\partial \mathscr{O}; \R^2)}.$$ We then rewrite (\ref{associated diff}) as $$\inner{\tilde{u}_s - \tilde{u}^n_s}{\psi}_1 +  \inner{(\kappa - \alpha)(\tilde{u}^n_s - \tilde{u}_s)}{\psi}_{L^2(\partial \mathscr{O}; \R^2)}.$$ Through another application of Lemma \ref{greens for navier} observe that \begin{align*}\inner{\tilde{u}_s - \tilde{u}^n_s}{\psi}_1 &= -\inner{\tilde{u}_s - \tilde{u}^n_s}{\Delta \psi} + \inner{\tilde{u}_s - \tilde{u}^n_s}{(\kappa - \alpha)\psi}_{L^2(\partial \mathscr{O}; \R^2)}\\  &= -\inner{\tilde{u}_s - \tilde{u}^n_s}{\Delta \psi} - \inner{(\kappa - \alpha)(\tilde{u}^n_s - \tilde{u}_s)}{\psi}_{L^2(\partial \mathscr{O}; \R^2)}.\end{align*} Therefore (\ref{associated diff}) ultimately reduces to the simple expression $$-\inner{\tilde{u}_s - \tilde{u}^n_s}{\Delta \psi} $$ which is exactly as we had in [\cite{goodair2023zero}], from which point the result follows. 

\end{proof}

By the Yamada-Watanabe type result [\cite{rockner2008yamada}], the existence of a weak solution of the equation (\ref{projected Ito}) would now follow from uniqueness.

\begin{proposition} \label{uniqueness prop}
    For any given $\mathcal{F}_0-$measurable $\hat{u}_0:\Omega \rightarrow L^2_{\sigma}$, weak solutions of the equation (\ref{projected Ito})  are unique.
\end{proposition}

\begin{proof}
    The corresponding proof is [\cite{goodair2023zero}] Proposition 3.10, and for this we need to first address how any weak solution satisfies an identity in $\left( \bar{W}^{1,2}_{\sigma} \right)^*$. To this end we need to consider a continuous extension of the Stokes operator $A$. Using the density of $\bar{W}^{2,2}_{\alpha}$ in $\bar{W}^{1,2}_{\sigma}$, Lemma \ref{greens for navier} and (\ref{inequality from Lions}) we justify that $A:\bar{W}^{1,2}_{\sigma} \rightarrow \left( \bar{W}^{1,2}_{\sigma} \right)^*$ defined for $f \in \bar{W}^{1,2}_{\sigma}$ by the duality pairing $$\inner{Af}{\phi}_{\left(\bar{W}^{1,2}_{\sigma} \right)^* \times \bar{W}^{1,2}_{\sigma}} = \inner{f}{g}_1 - \inner{(\kappa - \alpha)f}{\phi}_{L^2(\partial \mathscr{O}; \R^2)}$$ for all $\phi \in \bar{W}^{1,2}_{\sigma}$ is a continuous extension of the Stokes operator such that any weak solution $u$ satisfies, for every $t \in [0,T]$, the identity \begin{equation} \nonumber
    u_t = u_0 - \int_0^t\mathcal{P}\mathcal{L}_{u_s}u_s\ ds - \nu\int_0^t A u_s\, ds + \frac{1}{2}\int_0^t\sum_{i=1}^\infty \mathcal{P}\mathcal{Q}_i^2u_s ds - \int_0^t \mathcal{P}\mathcal{G}u_s d\mathcal{W}_s 
\end{equation} 
$\mathbb{P}-a.s.$ in $\left(\bar{W}^{1,2}_{\sigma} \right)^*$. In showing the uniqueness we now again have to address this additional boundary integral though there are no further complications from using (\ref{inequality from Lions}) in the manner seen already. The proof is concluded without additional difficulties.   
\end{proof}

In summary, we have proved that there exists a unique weak solution of the equation (\ref{projected Ito}) for the bounded initial condition. An extension to general $\mathcal{F}_0-$measurable $u_0:\Omega \rightarrow L^2_{\sigma}$ is precisely as in [\cite{goodair2023zero}] pp.44, by considering solutions for the initial condition $u_0\mathbbm{1}_{k \leq \norm{u_0} < k+1}$ and taking their sum. Hence we conclude the proof of Theorem \ref{theorem 2D} here.

\subsection{Strong Solutions} \label{sub strong sols}

We assume that $\alpha \geq \kappa$ as in the assumption of Theorem \ref{theorem 2D strong}. We fix an $\mathcal{F}_0-$measurable $u_0 \in L^\infty(\Omega;\bar{W}^{1,2}_{\sigma})$, for which we already have a unique weak solution $u$ of the equation (\ref{projected Ito}) due to Theorem \ref{theorem 2D}. Courtesy of Lemma \ref{lemma for strong is weak}, the task of proving existence of strong solutions is reduced to demonstrating that $u$ is progressively measurable in $\bar{W}^{2,2}_{\alpha}$ and such that for $\mathbb{P}-a.e.$ $\omega$, $u_{\cdot}(\omega) \in C\left([0,T];\bar{W}^{1,2}_{\sigma}\right) \cap L^2\left([0,T];\bar{W}^{2,2}_{\alpha}\right)$. Of course again from Lemma \ref{lemma for strong is weak} we have that any two strong solutions are weak solutions, establishing the uniqueness of strong solutions from that of weak solutions. Therefore the existence and uniqueness of a strong solution of the equation (\ref{projected Ito}) will follow from showing the additional regularity of $u$. Our approach is to use some of the abstract arguments of [\cite{goodair2022existence1}] Section 3, so we discuss how the framework applies in our case. This framework is laid out in Subsection \ref{Appendix III}, for a triplet of densely embedded Hilbert Spaces $$V \xhookrightarrow{} H \xhookrightarrow{} U$$ which we unsurprisingly set as $$V:= \bar{W}^{2,2}_{\alpha}, \qquad H:=\bar{W}^{1,2}_{\sigma}, \qquad U:= L^2_{\sigma}.$$
We also required a continuous bilinear form $\inner{\cdot}{\cdot}_{U \times V}: L^2_{\sigma} \times \bar{W}^{2,2}_{\alpha} \rightarrow \R$ such that for $\phi \in \bar{W}^{1,2}_{\sigma}$ and $\psi \in \bar{W}^{2,2}_{\alpha}$, \begin{equation} \label{UVH for navier}
    \inner{\phi}{\psi}_{U \times V} =  \inner{\phi}{\psi}_{H}.
\end{equation}
To this end we will actually endow $\bar{W}^{1,2}_{\sigma}$ with a new inner product, which we denote as $$\inner{\phi}{\psi}_{H}:= \inner{\phi}{\psi}_{1} - \inner{(\kappa - \alpha)\phi}{\psi}_{L^2(\partial \mathscr{O};\R^2)}.$$
\begin{lemma} \label{proof of bilinear form H}
    The bilinear form $\inner{\cdot}{\cdot}_H$ defines an inner product on $\bar{W}^{1,2}_{\sigma}$ equivalent to $\inner{\cdot}{\cdot}_1$.
\end{lemma}

\begin{proof}
    The result follows from the observation that $$\norm{\phi}_1^2 \leq \norm{\phi}_1^2 - \inner{(\kappa - \alpha)\phi}{\phi}_{L^2(\partial \mathscr{O};\R^2)} \leq \norm{\phi}_1^2 + c\norm{\phi}\norm{\phi}_1 \leq c\norm{\phi}_1^2 $$
    having used that $ -(\kappa - \alpha)$ is non-negative and (\ref{inequality from Lions}). Once more the constant $c$ is dependent on $\kappa$, $\alpha$ through their supremums. 
\end{proof}

\begin{remark}
    This result is exactly why we require $\alpha \geq \kappa$. The necessity of this requirement is considered in the conclusion.
\end{remark}
Accordingly for $\phi \in L^2_{\sigma}$ and $\psi \in \bar{W}^{2,2}_{\alpha}$ we define $$\inner{\phi}{\psi}_{U \times V}:= \inner{\phi}{A\psi}$$ from which the property (\ref{UVH for navier}) follows as $$\inner{\phi}{A\psi} = -\inner{\phi}{\Delta\psi} = \inner{\phi}{\psi}_{1} - \inner{(\kappa - \alpha)\phi}{\psi}_{L^2(\partial \mathscr{O};\R^2)} = \inner{\phi}{\psi}_H$$
through Lemma \ref{greens for navier}. The framework in [\cite{goodair2022existence1}] is established for functions $\mathcal{A}$ and $\tilde{\mathcal{G}}$, as in (\ref{thespde}), which we implement as
\begin{align*}
    \mathcal{A}:= -\left(\mathcal{P}\mathcal{L} + \nu A \right) + \frac{1}{2}\sum_{i=1}^\infty \mathcal{P}\mathcal{Q}_i^2, \qquad
    \tilde{\mathcal{G}}:= -\mathcal{P}\mathcal{G}.
\end{align*}
For Assumption \ref{assumption fin dim spaces}, we use the system $(\bar{a}_k)$ specified in Lemma \ref{eigenfunctions for navier} and associated projections $(\bar{\mathcal{P}}_k)$. For this we consider the relation between $(\bar{a}_k)$ and the newly defined $H$ inner product:
$$\inner{\bar{a}_j}{\bar{a}_k}_H = \inner{\bar{a}_j}{A\bar{a}_k} = \bar{\lambda}_k\inner{\bar{a}_j}{\bar{a}_k}$$
which is equal to zero if $j \neq k$ and $\bar{\lambda}_k$ if $j=k$. Therefore the system $(\bar{a}_k)$ remains orthogonal in the $H$ inner product, so each $\bar{\mathcal{P}}_n$ is self-adjoint for this inner product and the property (\ref{projectionsboundedonH}) holds for $c=1$.
Moreover for $\phi \in \bar{W}^{1,2}_{\sigma}$, \begin{align*}\norm{\left(I - \bar{\mathcal{P}}_n\right)\phi}^2 &= \sum_{k=n+1}^{\infty}\inner{\phi}{\bar{a}_k}^2\\ &= \frac{1}{\bar{\lambda}_n}\sum_{k=n+1}^{\infty}\inner{\phi}{\bar{a}_k}^2\bar{\lambda}_n\\ &\leq \frac{1}{\bar{\lambda}_n}\sum_{k=n+1}^{\infty}\inner{\phi}{\bar{a}_k}^2\bar{\lambda}_k\\ &\leq \frac{1}{\bar{\lambda}_n}\sum_{k=1}^{\infty}\inner{\phi}{\bar{a}_k}^2\bar{\lambda}_k\\ &= \frac{1}{\bar{\lambda}_n}\norm{\phi}_H^2.
\end{align*}
It is, however, apparent that we cannot verify Assumption \ref{new assumption 1}. Pertaining to (\ref{111}), the best estimate that we can do for the nonlinear term is via H\"{o}lder's Inequality and use of Theorem \ref{gagliardonirenberginequality} with $p=4, q=2, \alpha = 1/2$ and $m=1$:\footnote{The assumption was verified in [\cite{goodair2022navier}] for the space $U:=W^{1,2}$, not $L^2$.}
\begin{equation}\label{subbing in}\norm{\mathcal{L}_{\phi}\phi} \leq c\sum_{j=1}^2\norm{\phi}_{L^4}\norm{\partial_j \phi}_{L^4} \leq c\sum_{j=1}^2\norm{\phi}^{\frac{1}{2}}\norm{\phi}_1^{\frac{1}{2}}\norm{\partial_j \phi}^{\frac{1}{2}}\norm{\partial_j \phi}^{\frac{1}{2}}_{W^{1,2}} \leq c\norm{\phi}^{\frac{1}{2}}\norm{\phi}_1\norm{\phi}^{\frac{1}{2}}_2.\end{equation}
We note that (\ref{111}) is satisfied for the remaining terms. In the direction of \ref{assumptions for uniform bounds2} first note that
\begin{equation} \label{first parting}
    -2\nu\inner{\bar{\mathcal{P}}_nA\phi^n}{\phi^n}_H = -2\nu\inner{A\phi^n}{A\phi^n} = -2\nu\norm{\phi^n}_2^2
\end{equation}
using the representation (\ref{UVH for navier}) and that $A\phi^n \in \bar{V}_n$. For the nonlinear term we take the same approach, reducing $-2\inner{\bar{\mathcal{P}}_n\mathcal{P}\mathcal{L}_{\phi^n}\phi^n}{\phi^n}_H$ to $-2\inner{\mathcal{P}\mathcal{L}_{\phi^n}\phi^n}{A\phi^n}$, which using (\ref{subbing in}) is controlled by
\begin{equation} \label{nonlinear parting}
    2\left\vert \inner{\mathcal{P}\mathcal{L}_{\phi^n}\phi^n}{A\phi^n} \right\vert \leq 2\norm{\mathcal{L}_{\phi^n}\phi^n}\norm{\phi^n}_2 \leq c\norm{\phi^n}^{\frac{1}{2}}\norm{\phi^n}_1\norm{\phi^n}^{\frac{3}{2}}_2 \leq c\norm{\phi^n}^{2}\norm{\phi^n}_1^4 + \frac{\nu}{4}\norm{\phi^n}^{2}_2
\end{equation}
where we note again the dependency on $\nu$ in $c$ is not meaningful at this stage as $\nu$ remains fixed. For the noise contribution in (\ref{uniformboundsassumpt1}) we consider
\begin{equation} \label{we will label it for consideration}\sum_{i=1}^\infty \left(\inner{\bar{\mathcal{P}}_n\mathcal{P}\mathcal{Q}_i^2\phi^n}{\phi^n}_H + \norm{\bar{\mathcal{P}}_n\mathcal{P}\mathcal{G}_i\phi^n}_H^2\right) \end{equation}
where we can use that $\bar{\mathcal{P}}_n$ is an orthogonal projection in $H$, the assumption (\ref{assumpty 12}) and (\ref{inequality from Lions}) to obtain
\begin{align*}\inner{\bar{\mathcal{P}}_n\mathcal{P}\mathcal{Q}_i^2\phi^n}{\phi^n}_H + \norm{\bar{\mathcal{P}}_n\mathcal{P}\mathcal{G}_i\phi^n}_H^2 &\leq \inner{\mathcal{P}\mathcal{Q}_i^2\phi^n}{\phi^n}_H + \norm{\mathcal{P}\mathcal{G}_i\phi^n}_H^2\\ &= \inner{\mathcal{Q}_i^2\phi^n}{A\phi^n} + \norm{\mathcal{P}\mathcal{G}_i\phi^n}_1^2 - \inner{(\kappa - \alpha)\mathcal{P}\mathcal{G}_i\phi^n}{\mathcal{P}\mathcal{G}_i\phi^n}_{L^2(\partial \mathscr{O};\R^2)}\\
&\leq c_i\left(1 + \norm{\phi^n}_1^2\right) + k_i\norm{\phi}_2^2 + c\norm{\mathcal{P}\mathcal{G}_i\phi^n}\norm{\mathcal{P}\mathcal{G}_i\phi^n}_1\\
&\leq c_i\left(1 + \norm{\phi^n}_1^2\right) + k_i\norm{\phi}_2^2 + cc_i(1 + \norm{\phi^n}_1)(1 + \norm{\phi^n}_2)\end{align*}
where we also note use of (\ref{assumpty 2}), (\ref{assumpty 11}). An application of Young's Inequality on the last term reveals that
\begin{align*}
    (1 + \norm{\phi^n}_1)(1 + \norm{\phi^n}_2) \leq c(1 + \norm{\phi^n}_1)^2 + \frac{\nu}{8}(1 + \norm{\phi^n}_2)^2 \leq c(1 + \norm{\phi^n}_1)^2 + \frac{\nu}{4}(1 + \norm{\phi^n}_2^2). 
\end{align*}
Summing over all $i$, we see that
\begin{equation} \label{summing over i}
    \sum_{i=1}^\infty \left(\inner{\bar{\mathcal{P}}_n\mathcal{P}\mathcal{Q}_i^2\phi^n}{\phi^n}_H + \norm{\bar{\mathcal{P}}_n\mathcal{P}\mathcal{G}_i\phi^n}_H^2\right) \leq c(1 + \norm{\phi^n}_1^2) + \frac{5\nu}{8} \norm{\phi^n}_2^2.
\end{equation}
Summation of the terms (\ref{first parting}), (\ref{nonlinear parting}) and (\ref{summing over i}) yields (\ref{uniformboundsassumpt1}). As for (\ref{uniformboundsassumpt2}), the assumption is not wholly satisfied however we shall show that the boundedness result (uniform across $n$) of [\cite{goodair2022existence1}] Proposition 3.21 still holds. We have that $$\inner{\bar{\mathcal{P}}_n\mathcal{P}\mathcal{G}_i\phi^n}{\phi^n}_H^2 = \inner{\bar{\mathcal{P}}_n\mathcal{P}\mathcal{G}_i\phi^n}{A\phi^n}^2 = \inner{\mathcal{G}_i\phi^n}{A\phi^n}^2 \leq C_{\varepsilon}c_i(1 + \norm{\phi^n}_1^6) + c_i\varepsilon\norm{\phi}_2^2
$$
through (\ref{assumpty13}). The additional term in this assumption is of course $c_i\varepsilon\norm{\phi}_2^2$, which manifests itself as $$c  \mathbbm{E}\left(\sum_{i=1}^\infty\int_{\theta_j}^{\theta_k} c_i\varepsilon\norm{\tilde{\sy}^n_s}_V^2ds\right)^{\frac{1}{2}}$$ in the proof of [\cite{goodair2022existence1}] Proposition 3.21. We control this with the simple observation that 
$$\mathbbm{E}\left(\sum_{i=1}^\infty\int_{\theta_j}^{\theta_k} c_i\varepsilon\norm{\tilde{\sy}^n_s}_V^2ds\right)^{\frac{1}{2}} \leq \left(\mathbbm{E}\sum_{i=1}^\infty\int_{\theta_j}^{\theta_k} c_i\varepsilon\norm{\tilde{\sy}^n_s}_V^2ds\right)^{\frac{1}{2}} \leq \mathbbm{E}\sum_{i=1}^\infty\int_{\theta_j}^{\theta_k} c_i\varepsilon\norm{\tilde{\sy}^n_s}_V^2ds + 1$$ for which we can choose $\varepsilon = \frac{1}{2c\sum_{i=1}^\infty c_i}$ (note that the equation has already been scaled in terms of the viscosity in this abstract argument of [\cite{goodair2022existence1}]). Through this argument we prove the following, recalling the stopping times $(\tau^{M,S}_{n})$ defined in (\ref{the stopping times}). 

\begin{proposition} \label{uniform boundedness proposition}
    There exists a constant $C_{M}$ independent of $n$ (but dependent on $M$) such that $$ \mathbbm{E}\left[\sup_{r \in [0,\tau^{M,T}_n]}\norm{u^n_r}_1^2 + \int_0^{ \tau^{M,T}_n}\norm{u^n_r}_2^2dr  \right] \leq C_{M}\left[\mathbbm{E}\left(\norm{u^n_0}_1^2\right) + 1\right].\footnote{The result formally follows for $\norm{\cdot}_H$ as opposed to $\norm{\cdot}_1$, though we freely use the equivalence of the norms.}$$
\end{proposition}

Such a uniform boundedness result on the Galerkin Approximation is in the right direction for showing that $u$ has the required additional regularity, though there is a clear hurdle in the dependency on the stopping times. We look to find some stopping time which is $\mathbb{P}-a.s.$ less than an appropriate subsequence of $(\tau^{M,S}_n)$, to be characterised in terms of known properties of $u$ and take beyond $T$. Such an idea also appears in [\cite{goodair2022existence1}], originally from [\cite{glatt2009strong}], through showing a Cauchy property for the Galerkin system. This is facilitated by Assumption \ref{therealcauchy assumptions}, which we note was used implicitly for the uniqueness result Proposition \ref{uniqueness prop}. More precisely the Cauchy result we wish to show is [\cite{goodair2022existence1}] Proposition 3.24, which follows from Lemma 3.23. This lemma hinges on Assumptions \ref{therealcauchy assumptions} and \ref{new assumption 1}, the latter of which we know fails due to (\ref{subbing in}) being our best estimate for the nonlinear term. This translates to an additional term required to bound $`A$' in the lemma: in the proof, the term that needs a new control is $\beta$, specifically $$\frac{2}{\mu_m}\norm{\psi}^{\frac{1}{2}}\norm{\psi}_1\norm{\psi}^{\frac{1}{2}}_2\norm{\phi - \psi}_1$$ where we recall $\mu_m:=\bar{\lambda}_m$. By two applications of H\"{o}lder's Inequality, for any $\varepsilon > 0$, $$\frac{2}{\mu_m}\norm{\psi}^{\frac{1}{2}}\norm{\psi}_1\norm{\psi}^{\frac{1}{2}}_2\norm{\phi - \psi}_1 \leq \frac{C_{\varepsilon}}{\mu_m^2}\norm{\psi}\norm{\psi}_1^2\norm{\psi}_2 + \varepsilon\norm{\phi-\psi}_1^2 \leq \frac{C_{\varepsilon}}{\mu_m^2}\norm{\psi}_1^4 + \frac{C_{\varepsilon}}{\mu_m^2}\norm{\psi}^2\norm{\psi}_2^2 + \varepsilon\norm{\phi-\psi}_1^2 $$
 where the additional term is of course $$\frac{C_{\varepsilon}}{\mu_m^2}\norm{\psi}_1^4.$$ The proof of Proposition 3.24 suffers only a minor impact, realised as an additional term given by (in the notation of [\cite{goodair2022existence1}]) $$\frac{c}{\lambda_m}\mathbbm{E}\int_{\theta_j}^{\theta_k}\norm{\tilde{\sy}_s^m}_H^4ds$$ which is treated through (3.25) and Proposition 3.21:
 $$\frac{c}{\lambda_m}\mathbbm{E}\int_{\theta_j}^{\theta_k}\norm{\tilde{\sy}_s^m}_H^4ds \leq \frac{c}{\lambda_m}\mathbbm{E}\left(\sup_{r\in[0,t]}\norm{\tilde{\sy}^m_r}_H^2\int_{\theta_j}^{\theta_k}\norm{\tilde{\sy}_s^m}_H^2ds\right) \leq \frac{C_M}{\lambda_m}\mathbbm{E}\left(\sup_{r\in[0,t]}\norm{\tilde{\sy}^m_r}_H^2\right) \leq \frac{C_M}{\lambda_m}.$$
As such the proof is still valid in our setting, which verifies the following result.  In correspondence with [\cite{goodair2022existence1}] we introduce the time-dependent norm
    \begin{equation} \label{time dependent norm}
        \norm{\cdot}_{UH,t}^2 := \sup_{r \in [0,t]}\norm{\cdot}^2 + \int_0^t\norm{\cdot}_1^2.
    \end{equation}

\begin{proposition}\label{the cauchy prop}
    We have that $$\lim_{m \rightarrow \infty}\sup_{n \geq m}\mathbbm{E}\left[\norm{u^n-u^m}^2_{UH,\tau
_{m}^{M,t}\wedge \tau_{n}^{M,t}} \right] =0.$$
\end{proposition}
The purpose of this Cauchy property materialises in Proposition \ref{the one that gives limiting stopping time bound}, though we need one additional result first. This is somewhat analogous to [\cite{goodair2022existence1}] Proposition 3.25.
\begin{lemma} \label{probably unnecessary lemma}
    Let $\gamma$ be a stopping time and $(\delta_j)$ a sequence of stopping times which converge to $0$ $\mathbb{P}-a.s.$. Then $$\lim_{j \rightarrow \infty}\sup_{n\in\N}\mathbbm{E}\left(\norm{u^n}_{UH,(\gamma + \delta_j) \wedge \tau^{M,T}_n}^2 - \norm{u^n}_{UH,\gamma \wedge \tau^{M,T}_n}^2 \right) =0.$$
\end{lemma}

\begin{proof}
    We look at the energy identity satisfied by $u^n$ up until the stopping time $\gamma \wedge \tau^{M,T}_n$ and then $(\gamma + r)  \wedge \tau^{M,T}_n$ for some $r \geq 0$. We have that
        \begin{align*}
        \norm{u^n_{\gamma \wedge \tau^{M,T}_n}}^2 &= \norm{u^n_0}^2 - 2\int_0^{\gamma \wedge \tau^{M,T}_n}\inner{\mathcal{P}_n\mathcal{P}\mathcal{L}_{u^n_s}u^n_s}{u^n_s}\ ds - 2\nu\int_0^{\gamma \wedge \tau^{M,T}_n} \inner{\mathcal{P}_n A u^n_s}{u^n_s}\, ds\\ &+ \int_0^{\gamma \wedge \tau^{M,T}_n}\left\langle\sum_{i=1}^\infty \mathcal{P}_n \mathcal{P}\mathcal{Q}_i^2u^n_s, u^n_s\right\rangle ds +\int_0^{\gamma \wedge \tau^{M,T}_n}\sum_{i=1}^\infty \norm{\mathcal{P}_n\mathcal{P}\mathcal{G}_iu^n_s}^2ds\\ &- \int_0^{\gamma \wedge \tau^{M,T}_n}\inner{\mathcal{P}_n\mathcal{P}\mathcal{G}u^n_s}{u^n_s} d\mathcal{W}_s 
    \end{align*}
    and similarly for $(\gamma + r)\wedge \tau^{M,T}_n$, from which the difference of the equalities gives
     \begin{align*}
        \norm{u^n_{(\gamma + r)\wedge \tau^{M,T}_n}}^2 &= \norm{u^n_{\gamma \wedge \tau^{M,T}_n}}^2 - 2\int_{\gamma \wedge \tau^{M,T}_n}^{(\gamma + r)\wedge \tau^{M,T}_n}\inner{\mathcal{P}_n\mathcal{P}\mathcal{L}_{u^n_s}u^n_s}{u^n_s}\ ds - 2\nu\int_{\gamma}^{(\gamma + r)\wedge \tau^{M,T}_n} \inner{\mathcal{P}_n A u^n_s}{u^n_s}\, ds\\ &+ \int_{\gamma \wedge \tau^{M,T}_n}^{(\gamma + r)\wedge \tau^{M,T}_n}\left\langle\sum_{i=1}^\infty \mathcal{P}_n \mathcal{P}\mathcal{Q}_i^2u^n_s, u^n_s\right\rangle ds +\int_{\gamma \wedge \tau^{M,T}_n}^{(\gamma + r)\wedge \tau^{M,T}_n}\sum_{i=1}^\infty \norm{\mathcal{P}_n\mathcal{P}\mathcal{G}_iu^n_s}^2ds\\ &- \int_{\gamma \wedge \tau^{M,T}_n}^{(\gamma + r)\wedge \tau^{M,T}_n}\inner{\mathcal{P}_n\mathcal{P}\mathcal{G}u^n_s}{u^n_s} d\mathcal{W}_s. 
    \end{align*}
This has a straightforwards simplification to
     \begin{align*}
        \norm{u^n_{(\gamma + r)\wedge \tau^{M,T}_n}}^2 &- \norm{u^n_{\gamma \wedge \tau^{M,T}_n}}^2 + 2\nu\int_{\gamma \wedge \tau^{M,T}_n}^{(\gamma + r)\wedge \tau^{M,T}_n} \norm{u^n_s}_H^2 ds\\ &=  \int_{\gamma \wedge \tau^{M,T}_n}^{(\gamma + r)\wedge \tau^{M,T}_n}\sum_{i=1}^\infty\left(\left\langle \mathcal{Q}_i^2u^n_s, u^n_s\right\rangle+ \norm{\mathcal{P}_n\mathcal{P}\mathcal{G}_iu^n_s}^2\right)ds - \int_{\gamma \wedge \tau^{M,T}_n}^{(\gamma + r)\wedge \tau^{M,T}_n}\inner{\mathcal{G}u^n_s}{u^n_s} d\mathcal{W}_s. 
    \end{align*}
Passing to an inequality, applying (\ref{assumpty 7}), taking the supremum over $r\in[0,\delta_j]$ and then using the Burkholder-Davis-Gundy Inequality,
\begin{align*}
        &\mathbbm{E}\left[\sup_{r \in [0,\delta_j]}\norm{u^n_{(\gamma + r)\wedge \tau^{M,T}_n}}^2 - \norm{u^n_{\gamma \wedge \tau^{M,T}_n}}^2 + 2\nu\int_{\gamma \wedge \tau^{M,T}_n}^{(\gamma + \delta_j)  \wedge \tau^{M,T}_n} \norm{u^n_s}_H^2 ds\right]\\ & \qquad \qquad \leq \mathbbm{E}\int_{\gamma \wedge \tau^{M,T}_n}^{(\gamma + \delta_j)\wedge \tau^{M,T}_n}c + c\norm{u^n_s}^2 + \nu \norm{u^n_s}^2_1 ds + c\mathbbm{E}\left(\int_{\gamma \wedge \tau^{M,T}_n}^{(\gamma + \delta_j)\wedge \tau^{M,T}_n}\sum_{i=1}^\infty\inner{\mathcal{G}_iu^n_s}{u^n_s}^2 ds\right)^{\frac{1}{2}}. 
    \end{align*}
Using (\ref{assumpty8}) we pass to the inequality
\begin{align*}
        &\mathbbm{E}\left[\sup_{r \in [0,\delta_j]}\norm{u^n_{(\gamma + r)\wedge \tau^{M,T}_n}}^2 - \norm{u^n_{\gamma \wedge \tau^{M,T}_n}}^2 + 2\nu\int_{\gamma \wedge \tau^{M,T}_n}^{(\gamma + \delta_j)  \wedge \tau^{M,T}_n} \norm{u^n_s}_H^2 ds\right]\\ & \qquad \qquad \qquad \qquad \leq \mathbbm{E}\int_{\gamma \wedge \tau^{M,T}_n}^{(\gamma + \delta_j)\wedge \tau^{M,T}_n}c + c\norm{u^n_s}^2  + \nu \norm{u^n_s}^2_1 ds + \mathbbm{E}\left(\int_{\gamma \wedge \tau^{M,T}_n}^{(\gamma + \delta_j)\wedge \tau^{M,T}_n}c + c\norm{u^n_s}^4 ds\right)^{\frac{1}{2}}. 
    \end{align*}
Here we remind ourselves that, as shown in the proof of Lemma \ref{proof of bilinear form H}, $\norm{u^n_s}_1^2 \leq \norm{u^n_s}_H^2$. In addition we use the control on $u^n$ in $L^2_{\sigma}$ up until the stopping time $\tau^{M,T}_n$, deducing that
\begin{align*}
        &\mathbbm{E}\left[\sup_{r \in [0,\delta_j]}\norm{u^n_{(\gamma + r)\wedge \tau^{M,T}_n}}^2 - \norm{u^n_{\gamma \wedge \tau^{M,T}_n}}^2 + \nu\int_{\gamma \wedge \tau^{M,T}_n}^{(\gamma + \delta_j)  \wedge \tau^{M,T}_n} \norm{u^n_s}_H^2 ds\right]\\ & \qquad \qquad \qquad \qquad \qquad \qquad \qquad \qquad \leq \mathbbm{E}\int_{\gamma \wedge \tau^{M,T}_n}^{(\gamma + \delta_j)\wedge \tau^{M,T}_n}C_M ds +\mathbbm{E} \left(\int_{\gamma \wedge \tau^{M,T}_n}^{(\gamma + \delta_j)\wedge \tau^{M,T}_n}C_M ds\right)^{\frac{1}{2}}\\
        & \qquad \qquad \qquad \qquad \qquad \qquad \qquad \qquad \leq \mathbbm{E}\left(C_M\delta_j + C_M \delta_j^{\frac{1}{2}} \right)    
    \end{align*}
and further
\begin{align} \label{further align}
        \mathbbm{E}\left[\sup_{r \in [0,\delta_j]}\norm{u^n_{(\gamma + r)\wedge \tau^{M,T}_n}}^2 - \norm{u^n_{\gamma \wedge \tau^{M,T}_n}}^2 + \int_{\gamma \wedge \tau^{M,T}_n}^{(\gamma + \delta_j)  \wedge \tau^{M,T}_n} \norm{u^n_s}_1^2 ds\right] \leq \mathbbm{E}\left(C_{M}\delta_j + C_{M} \delta_j^{\frac{1}{2}} \right)
    \end{align}
via the equivalence of the norms and a simple scaling with $\nu$. We now need to relate the expression on the left hand side with what we are interested in, that is $\norm{u^n}_{UH,(\gamma + \delta_j) \wedge \tau^{M,T}_n}^2 - \norm{u^n}_{UH,\gamma \wedge \tau^{M,T}_n}^2$. We have that \begin{align*}\norm{u^n}_{UH,(\gamma + \delta_j) \wedge \tau^{M,T}_n}^2 &- \norm{u^n}_{UH,\gamma \wedge \tau^{M,T}_n}^2\\ &= \sup_{s \in [0,(\gamma + \delta_j)\wedge \tau^{M,T}_n]}\norm{u^n_s}^2 -\sup_{s \in [0,\gamma \wedge \tau^{M,T}_n]}\norm{u^n_s}^2 + \int_{\gamma \wedge \tau^{M,T}_n}^{(\gamma + \delta_j)  \wedge \tau^{M,T}_n} \norm{u^n_s}_1^2 ds.\end{align*} and claim that \begin{equation}\label{theclaim}\sup_{s \in [0,(\gamma + \delta_j)\wedge \tau^{M,T}_n]}\norm{u^n_s}^2 -\sup_{s \in [0,\gamma \wedge \tau^{M,T}_n]}\norm{u^n_s}^2 \leq \sup_{r \in [0,\delta_j]}\norm{u^n_{(\gamma + r)\wedge \tau^{M,T}_n}}^2 - \norm{u^n_{\gamma \wedge \tau^{M,T}_n}}^2.\end{equation}
Indeed, we have that $$ \sup_{s \in [0,(\gamma + \delta_j)\wedge \tau^{M,T}_n]}\norm{u^n_s}^2 \leq \sup_{s \in [0,\gamma \wedge \tau^{M,T}_n]}\norm{u^n_s}^2 + \sup_{s \in [\gamma \wedge \tau^{M,T}_n,(\gamma + \delta_j)\wedge \tau^{M,T}_n]}\norm{u^n_s}^2 - \norm{u^n_{\gamma \wedge \tau^{M,T}_n}}^2$$ as the left hand side must equal either $\sup_{s \in [0,\gamma \wedge \tau^{M,T}_n]}\norm{u^n_s}^2$ or $\sup_{s \in [\gamma \wedge \tau^{M,T}_n,(\gamma + \delta_j)\wedge \tau^{M,T}_n]}\norm{u^n_s}^2$, both of which are greater than the subtracted term $\norm{u^n_{\gamma \wedge \tau^{M,T}_n}}^2$. Appreciating that $$\sup_{s \in [\gamma \wedge \tau^{M,T}_n,(\gamma + \delta_j)\wedge \tau^{M,T}_n]}\norm{u^n_s}^2 = \sup_{r \in [0,\delta_j]}\norm{u^n_{(\gamma + r)\wedge \tau^{M,T}_n}}^2 $$ then yields the claim (\ref{theclaim}), which in combination with (\ref{further align}) grants that
\begin{align} \nonumber
        \mathbbm{E}\left[\norm{u^n}_{UH,(\gamma + \delta_j) \wedge \tau^{M,T}_n}^2 - \norm{u^n}_{UH,\gamma \wedge \tau^{M,T}_n}^2\right] \leq \mathbbm{E}\left(C_{M}\delta_j + C_{M} \delta_j^{\frac{1}{2}} \right).
    \end{align}
To conclude the proof we simply note that $C_{M}\delta_j+ C_{M} \delta_j^{\frac{1}{2}}$ is $\mathbbm{P}-a.s.$ monotone decreasing (as $j \rightarrow \infty$) and  convergent to $0$. The Monotone Convergence Theorem thus justifies that the limit as $j \rightarrow \infty$ of the right hand side is zero, so the result is shown.
   
\end{proof}


\begin{proposition} \label{the one that gives limiting stopping time bound} For any $R> 0$, there exists a subsequence indexed by $(m_j)$ and a constant $M>1$ such that the stopping time $$\tau^{R,T} := T \wedge \inf \left\{s \geq 0: \sup_{r \in [0,s]}\norm{u_r}^2 + \int_0^s \norm{u_r}_1^2dr \geq R \right\}$$ satisfies the property $\tau^{R,T} \leq \tau^{M,T}_{m_j}$ for all $m_j$ $\mathbb{P}-a.s.$.
\end{proposition}

\begin{proof}
    This comes from a direct application of Lemma \ref{amazing cauchy lemma}, for the Banach Space $X_t:=C\left([0,t];L^2_{\sigma}\right) \cap L^2\left([0,t];\bar{W}^{1,2}_{\sigma} \right)$ equipped with norm $\norm{\cdot}_{UH,t}$. The conditions (\ref{supposition 1}) and (\ref{supposition 2}) are given in Proposition \ref{the cauchy prop} and Lemma \ref{probably unnecessary lemma}. The only component in need of proof is that the limiting process $\sy$ we achieve from Lemma \ref{amazing cauchy lemma} is in fact $u$. We are given that $(u^{m_j})$ converges to $\sy$ in $L^2\left([0,\tau^{M,T}_{\infty}];\bar{W}^{1,2}_{\sigma}\right)$ which is a stronger convergence (up to $\tau^{M,T}_{\infty}$) than what we had in Theorem \ref{theorem for new prob space}, which was used to show that the limit process was a weak solution of the equation (\ref{projected Ito}). Hence in the same manner as Proposition \ref{prop for first energy bar}, we see that $\sy$ must be a $\textit{local}$ weak solution as well, up until the stopping time $\tau^{M,T}_{\infty}$, so from the uniqueness of weak solutions we must have that $\sy$ is in fact equal to $u$ on $[0,\tau^{M,T}_{\infty}]$, which is sufficient for the result.
\end{proof}

This allows us to deduce the additional regularity required for $u$ to be a strong solution.

\begin{proposition} \label{final regularity of solutions}
    The process $u$ is progressively measurable in $\bar{W}^{2,2}_{\alpha}$ and such that for $\mathbb{P}-a.e.$ $\omega$, $u_{\cdot}(\omega) \in C\left([0,T];\bar{W}^{1,2}_{\sigma}\right) \cap L^2\left([0,T];\bar{W}^{2,2}_{\alpha}\right)$.
\end{proposition}
\begin{proof}
Propositions \ref{uniform boundedness proposition} and \ref{the one that gives limiting stopping time bound} together imply that $$ \mathbbm{E}\left[\sup_{r \in [0,\tau^{R,T}]}\norm{u^{m_j}_r}_1^2 + \int_0^{ \tau^{R,T}}\norm{u^{m_j}_r}_2^2dr  \right] \leq C_{R, \nu}$$
where $C_{R,\nu}$ now incorporates $\norm{u_0}^2_{L^\infty(\Omega;\bar{W}^{1,2}_{\sigma})}$ and shows the dependency on $M$ in terms of $R$. In consequence the sequence of processes $(u^{m_j}_{\cdot}\mathbbm{1}_{\cdot \leq \tau^{R,T}})$ is uniformly bounded in both $L^2\left(\Omega; L^\infty\left([0,T];\bar{W}^{1,2}_{\sigma}\right) \right)$ and $L^2\left(\Omega; L^2\left([0,T];\bar{W}^{2,2}_{\alpha}\right) \right)$. We can deduce the existence of a subsequence which is weakly convergent in the Hilbert Space $L^2\left(\Omega;L^2\left([0,T];\bar{W}^{2,2}_{\alpha}\right)\right)$ to some $\tilde{\sy}$, but we may also identify $L^2\left(\Omega;L^\infty\left([0,T];\bar{W}^{1,2}_{\sigma}\right)\right)$ with the dual space of $L^2\left(\Omega;L^1\left([0,T];\bar{W}^{1,2}_{\sigma}\right)\right)$ and as such from the Banach-Alaoglu Theorem we can extract a further subsequence which is convergent to some $\tilde{\py}$ in the weak* topology. These limits imply convergence to both $\tilde{\sy}$ and $\tilde{\py}$ in the weak topology of $L^2\left(\Omega;L^2\left([0,T];\bar{W}^{1,2}_{\sigma}\right)\right)$. However, we know that $(u^{m_j}_{\cdot}\mathbbm{1}_{\cdot \leq \tau^{R,T}})$ converges to $u_{\cdot}\mathbbm{1}_{\cdot \leq \tau^{R,T}}$ $\mathbb{P}-a.s.$ in $L^2\left([0,T];\bar{W}^{1,2}_{\sigma}\right)$ from Lemma \ref{amazing cauchy lemma}, and from the Dominated Convergence Theorem (domination comes easily from   $\mathbbm{1}_{\cdot \leq \tau^{R,T}}$, $\tau^{R,T} \leq \tau^{M,T}_{m_j}$) then in fact the convergence holds in $L^2\left(\tilde{\Omega};L^2\left([0,T];\bar{W}^{1,2}_{\sigma}\right)\right)$. Of course the same convergence then holds in the weak topology, and by uniqueness of limits in the weak topology then $u_{\cdot}\mathbbm{1}_{\cdot \leq \tau^{R,T}} = \tilde{\sy} = \tilde{\py}$ as elements of $L^2\left(\Omega;L^2\left([0,T];\bar{W}^{1,2}_{\sigma}\right)\right)$, so they agree $\mathbbm{P} \times \lambda-a.s.$. Thus for $\mathbbm{P}-a.e.$ $\omega$, $u_{\cdot}(\omega)\mathbbm{1}_{\cdot \leq \tau^{R,T}(\omega)} \in L^{\infty}\left([0,T];\bar{W}^{1,2}_{\sigma}\right)\cap L^2\left([0,T];\bar{W}^{2,2}_{\alpha}\right)$. At each such $\omega$ the regularity of $u$ as a weak solution ensures that for sufficiently large $R$ (dependent on $\omega$), $\tau^{R,T}(\omega) = T$. Therefore for $\mathbb{P}-a.e.$ $\omega$
, $u_{\cdot}(\omega) \in L^{\infty}\left([0,T];\bar{W}^{1,2}_{\sigma}\right)\cap L^2\left([0,T];\bar{W}^{2,2}_{\alpha}\right)$.\\

The progressive measurability is justified similarly; for any $t\in [0,T]$, we can use the progressive measurability of $(u^{m_j})$ and hence $(u^{m_j}_{\cdot}\mathbbm{1}_{\cdot \leq \tau^{R,T}})$ to instead deduce $u_{\cdot}\mathbbm{1}_{\cdot \leq \tau^{R,T}}$ as the weak limit in $L^2\left(\Omega \times [0,t];\bar{W}^{2,2}_{\alpha}\right)$ where $\Omega \times [0,t]$ is equipped with the $\mathcal{F}_t \times \mathcal{B}\left([0,t]\right)$ sigma-algebra. Therefore $u_{\cdot}\mathbbm{1}_{\cdot \leq \tau^{R,T}}: \Omega \times [0,t] \rightarrow \bar{W}^{2,2}_{\alpha}$ is measurable with respect to this product sigma-algebra which justifies the progressive measurability of the truncated process. To carry this property over to $u$, we recognise $u$ as the $\mathbb{P} \times \lambda$ almost everywhere limit of the sequence $(u_{\cdot}\mathbbm{1}_{\cdot \leq \tau^{R,T}})$ as $R \rightarrow \infty$ over the product space $\Omega \times [0,t]$ equipped with product sigma-algebra $\mathcal{F}_t \times \mathcal{B}\left([0,t]\right)$. Such a limit preserves the measurability in this product sigma-algebra, justifying the progressive measurability of $u$.\\

It thus only remains to show that $u \in C\left([0,T]; \bar{W}^{1,2}_{\sigma}\right)$ $\mathbb{P}-a.s.$. From Lemma \ref{lemma for strong is weak} we have that $u$ satisfies the identity (\ref{a new id}) $\mathbb{P}-a.s.$, from which the continuity follows as an immediate application of Proposition \ref{rockner prop}.
\end{proof}

This concludes the proof that $u$ is a strong solution of the equation (\ref{projected Ito}). Of course this was done only for an initial condition $u_0 \in L^{\infty}\left(\Omega;\bar{W}^{1,2}_{\sigma}\right)$, but the extension of a solution to a general $\mathcal{F}_0-$measurable $u_0: \Omega \rightarrow \bar{W}^{1,2}_{\sigma}$ is again mechanically identical to [\cite{goodair2023zero}] pp.44. as also used in Theorem \ref{theorem 2D}. We have proven Theorem \ref{theorem 2D strong}.

\subsection{Verifying the Assumptions for SALT Noise} \label{sub lie transport}

As mentioned in Subsection \ref{sub salt}, the majority of the assumptions were verified in [\cite{goodair2023zero}] and it only remains to justify (\ref{assumpty 12}) and (\ref{assumpty13}).
\begin{itemize}
    \item[ (\ref{assumpty 12}):] We want to use the same approach as in [\cite{goodair2022navier}] Lemma 2.7, where this bound was demonstrated on the torus. This requires a Green's Identity type result in the $\norm{\mathcal{P}B_i\phi}_1^2$ term, for example Lemma \ref{greens for navier}, though it is not clear how we can do this as $\mathcal{P}B_i\phi$ does not necessarily satisfy the Navier boundary conditions (\ref{navier boundary conditions}). We can however utilise Corollary \ref{vorticity corollary}, given that $\mathcal{P}B_i\phi \cdot \mathbf{n}$ is zero on $\partial \mathscr{O}$, and Lemma \ref{lemma for curl and P} establishes that $\textnormal{curl}\left(\mathcal{P}B_i\phi \right) = \textnormal{curl}\left(B_i\phi \right)$. Inspecting this further,
\begin{align*}
    \textnormal{curl}\left(B_i\phi \right) &= \partial_1\left(\sum_{j=1}^2\left[\xi_i^j\partial_j\phi^2 + \phi^j\partial_2\xi_i^j\right]\right) - \partial_2\left(\sum_{j=1}^2\left[\xi_i^j\partial_j\phi^1 + \phi^j\partial_1\xi_i^j\right] \right)\\
    &= \sum_{j=1}^2\left(\partial_1\xi_i^j\partial_j\phi^2 +   \xi_i^j\partial_1\partial_j\phi^2 + \partial_1\phi^j\partial_2\xi_i^j +  \phi^j\partial_1\partial_2\xi_i^j \right)\\ & \qquad \qquad - \sum_{j=1}^2\left( \partial_2\xi_i^j\partial_j\phi^1 + \xi_i^j\partial_2\partial_j\phi^1 + \partial_2\phi^j\partial_1\xi_i^j + \phi^j\partial_2\partial_1\xi_i^j \right)\\
    &= \sum_{j=1}^2\left(\partial_1\xi_i^j\partial_j\phi^2 +   \xi_i^j\partial_1\partial_j\phi^2 + \partial_1\phi^j\partial_2\xi_i^j - \partial_2\xi_i^j\partial_j\phi^1 - \xi_i^j\partial_2\partial_j\phi^1 - \partial_2\phi^j\partial_1\xi_i^j\right)\\
    &= \sum_{j=1}^2\xi_i^j\partial_j\left(\partial_1\phi^2 - \partial_2\phi^1\right) + \sum_{j=1}^2\left(\partial_1\xi_i^j\partial_j\phi^2 + \partial_1\phi^j\partial_2\xi_i^j - \partial_2\xi_i^j\partial_j\phi^1- \partial_2\phi^j\partial_1\xi_i^j\right).
\end{align*}
Observe that the first sum is simply $\mathcal{L}_{\xi_i}(\textnormal{curl}\phi)$, and to consider the second we expand out the summation: \begin{align*}
 \left(\partial_1\xi_i^1\partial_1\phi^2 + \partial_1\phi^1\partial_2\xi_i^1 - \partial_2\xi_i^1\partial_1\phi^1- \partial_2\phi^1\partial_1\xi_i^1\right) + \left(\partial_1\xi_i^2\partial_2\phi^2 + \partial_1\phi^2\partial_2\xi_i^2 - \partial_2\xi_i^2\partial_2\phi^1- \partial_2\phi^2\partial_1\xi_i^2\right)
\end{align*}
for which the only remaining terms are $$\partial_1\xi_i^1\partial_1\phi^2 -  \partial_2\phi^1\partial_1\xi_i^1 + \partial_1\phi^2\partial_2\xi_i^2 - \partial_2\xi_i^2\partial_2\phi^1$$
or simply $$\left(\textnormal{curl}\phi\right)\sum_{j=1}^2\partial_j\xi_i^j$$ which is zero given that $\xi_i$ is divergence-free. In summary then we have that $\textnormal{curl}\left(\mathcal{P}B_i\phi \right) = \mathcal{L}_{\xi_i}(\textnormal{curl}\phi)$. The fact that this term belongs to $W^{1,2}_0(\mathscr{O};\R^2)$ for $\xi_i \in W^{2,2}_0(\mathscr{O};\R^2)$ was shown in [\cite{goodair2022navier}] pp.32. In particular the tangential component at the boundary is zero, so by Corollary \ref{vorticity corollary} then $\mathcal{P}B_i\phi$ satisfies the boundary condition $$2(D[\mathcal{P}B_i\phi])\mathbf{n} \cdot \mathbf{\iota} + 2\kappa f\cdot \mathbf{\iota} = 0.$$ In particular $\mathcal{P}B_i\phi$ belongs to $\bar{W}^{2,2}_{2\kappa}$, so from Lemma \ref{greens for navier} we have that $$\norm{\mathcal{P}B_i\phi}_1^2 = -\inner{\mathcal{P}B_i\phi}{\Delta \mathcal{P}B_i\phi} - \inner{\kappa \mathcal{P}B_i\phi}{\mathcal{P}B_i\phi}.$$
Returning to what we have to show, (\ref{assumpty 12}), it suffices to control \begin{equation}\nonumber \inner{B_i^2\phi}{A\phi}-\inner{\mathcal{P}B_i\phi}{\Delta \mathcal{P}B_i\phi} - \inner{\kappa \mathcal{P}B_i\phi}{\mathcal{P}B_i\phi}_{L^2(\partial \mathscr{O};\R^2)}\end{equation} or equivalently $$\inner{\mathcal{P}B_i^2\phi}{A\phi} +\inner{\mathcal{P}B_i\phi}{A\mathcal{P}B_i\phi} - \inner{\kappa \mathcal{P}B_i\phi}{\mathcal{P}B_i\phi}_{L^2(\partial \mathscr{O};\R^2)}.$$ In [\cite{goodair2022navier}] Lemma 2.7 it is demonstrated that \begin{equation}\label{suffices to control}\inner{\mathcal{P}B_i^2\phi}{A\phi}+\inner{\mathcal{P}B_i\phi}{A \mathcal{P}B_i\phi}\leq C_{\varepsilon}\norm{\xi_i}_{W^{3,\infty}}^2\norm{\phi}_1^2 + \varepsilon\norm{\xi_i}_{W^{3,\infty}}^2\norm{\phi}_2^2 \end{equation}
for any $\varepsilon > 0$. It is formally proven on the torus, but there is no change in the case of a bounded domain. The boundary integral is by now familiar, applying (\ref{inequality from Lions}):
\begin{align*}\left\vert\inner{\kappa \mathcal{P}B_i\phi}{\mathcal{P}B_i\phi}_{L^2(\partial \mathscr{O};\R^2)}\right\vert &\leq c\norm{\mathcal{P}B_i\phi}\norm{\mathcal{P}B_i\phi}_1\\ &\leq c\norm{\xi_i}^2_{W^{2,\infty}}\norm{\phi}_1\norm{\phi}_2\\ &\leq C_{\varepsilon}\norm{\xi_i}^2_{W^{2,\infty}}\norm{\phi}_1^2 + \varepsilon\norm{\xi_i}^2_{W^{2,\infty}}\norm{\phi}_2^2\end{align*}
where the bound on $B_i$ comes from [\cite{goodair2022navier}] Corollary 1.26.1. A choice of $$\varepsilon = \frac{\nu}{2\sum_{i=1}^\infty \norm{\xi_i}^2_{W^{3,\infty}}}$$ concludes the justification.
\item[(\ref{assumpty13}):] Lemma \ref{greens for navier} is again central, as we observe that $$\inner{B_i\phi}{A\phi} = -\inner{\mathcal{P}B_i\phi}{\Delta\phi} = \inner{\mathcal{P}B_i\phi}{\phi}_1 - \inner{(\kappa - \alpha)\mathcal{P}B_i\phi}{\phi}_{L^2(\partial \mathscr{O};\R^2)}$$
and by the same processes as the verification of (\ref{assumpty 12}),
$$\inner{\mathcal{P}B_i\phi}{\phi}_1 = -\inner{\Delta\mathcal{P}B_i\phi}{\phi} - \inner{\kappa\mathcal{P}B_i\phi}{\phi}_{L^2(\partial \mathscr{O};\R^2)}.$$ Combining the two and incorporating $\mathcal{P}$, then in particular $$\inner{B_i\phi}{A\phi}^2 \leq 2\inner{A\mathcal{P}B_i\phi}{\phi}^2 + 2\inner{(\alpha -2\kappa)\mathcal{P}B_i\phi}{\phi}_{L^2(\partial \mathscr{O};\R^2)}^2.$$
The first term is now handled exactly as in [\cite{goodair2022navier}] (37), where one has to use the $W^{3,2}_{0}(\mathscr{O};\R^2)$ regularity on the $\xi_i$ to conduct the integration by parts with null boundary integral, achieving that \begin{equation} \label{to combine with} \inner{A\mathcal{P}B_i\phi}{\phi}^2 \leq c\norm{\xi_i}_{W^{3,\infty}}^2\norm{\phi}_1^4.   
\end{equation}For the remaining term simply apply (\ref{inequality from Lions}):
$$\inner{(\alpha -2\kappa)\mathcal{P}B_i\phi}{\phi}_{L^2(\partial \mathscr{O};\R^2)}^2 \leq c\norm{\mathcal{P}B_i\phi}\norm{\mathcal{P}B_i\phi}_1\norm{\phi}\norm{\phi}_1 \leq c \norm{\xi_i}_{W^{2,\infty}}^2\norm{\phi}\norm{\phi}_1^2\norm{\phi}_2$$
and see further that $$ c \norm{\xi_i}_{W^{2,\infty}}^2\norm{\phi}\norm{\phi}_1^2\norm{\phi}_2 \leq C_{\varepsilon} \norm{\xi_i}_{W^{2,\infty}}^2\norm{\phi}^2\norm{\phi}_1^4 + \varepsilon\norm{\xi_i}_{W^{2,\infty}}^2 \norm{\phi}_2^2.$$
Combining with (\ref{to combine with}) achieves the result.

\end{itemize}

\newpage
\section{The Zero Viscosity Limit} \label{section zero viscy limit}

In this section we assume the conditions required for Theorem \ref{theorem for limit of navier stokes is euler}, namely $u_0 \in L^2\left(\Omega; \bar{W}^{1,2}_{\sigma} \right)$ and $\kappa \geq 0$ on $\partial \mathscr{O}$. This ensures that $2\kappa \geq \kappa$ so Theorem \ref{theorem 2D strong} applies for $\alpha:= 2\kappa$ and $\mathcal{G}_i = \mathcal{Q}_i = B_i$. 

\subsection{Estimates Independent of the Viscosity} \label{sub estimates indep viscos}

For any given $\nu$ we know that there exists  a unique strong solution $u$ of the equation (\ref{projected Ito Salt}) given by Theorem \ref{theorem 2D strong} for $\alpha:=2\kappa$. In this subsection we look to obtain an estimate on $u$ independent of $\nu$, to facilitate the required regularity in the limit as $\nu \rightarrow 0$. To achieve this we shall consider the vorticity of $u$, and denote again $w:= \textnormal{curl}u$. Recall that in Subsection \ref{sub lie transport} we showed that $\textnormal{curl}(B_if) = \mathcal{L}_{\xi_i}(\textnormal{curl}f)$. We denote $\mathcal{L}_{\xi_i}$ by the simple $\mathcal{L}_{i}.$

\begin{lemma} \label{lemmaofspatiallyweak}
The process $w$ is progressively measurable in $W^{1,2}_{0}(\mathscr{O};\R)$ and such that for $\mathbb{P}-a.e.$ $\omega$, $w_{\cdot}(\omega) \in C\left([0,T];L^2(\mathscr{O};\R)\right) \cap L^2\left([0,T];W^{1,2}_{0}(\mathscr{O};\R)\right)$. Moreover $w$ satisfies the identity
\begin{align} \nonumber
     \inner{w_t}{\phi}_{L^2(\mathscr{O};\R)} &= \inner{w_0}{\phi}_{L^2(\mathscr{O};\R)} - \int_0^{t}\inner{\mathcal{L}_{u_s}w_s}{\phi}_{L^2(\mathscr{O};\R)}ds - \nu\int_0^{t} \inner{w_s}{\phi}_{W^{1,2}_{0}(\mathscr{O};\R)} ds  \\ &- \frac{1}{2}\int_0^{t}\sum_{i=1}^\infty \inner{\mathcal{L}_iw_s}{\mathcal{L}_i\phi}_{L^2(\mathscr{O};\R)} ds - \sum_{i=1}^\infty\int_0^{t} \inner{\mathcal{L}_iw_s}{\phi}_{L^2(\mathscr{O};\R)} dW^i_s \label{another random new id}
\end{align}
for every $\phi \in W^{1,2}_{0}(\mathscr{O};\R)$, $\mathbb{P}-a.s.$ in $\R$ for all $t\in[0,T]$.
\end{lemma}

\begin{proof}
    We first comment on the regularity of $w$, noting that the curl is a continuous linear operator from $W^{2,2}(\mathscr{O};\R^2) \rightarrow W^{1,2}(\mathscr{O};\R)$ and $W^{1,2}(\mathscr{O};\R^2) \rightarrow L^2(\mathscr{O};\R)$. This justifies the progressive measurability and pathwise regularity for the general $W^{1,2}(\mathscr{O};\R)$ space, whilst the fact that $w$ belongs to $W^{1,2}_0(\mathscr{O};\R)$ is a direct consequence of Corollary \ref{vorticity corollary} in the choice of $\alpha = 2\kappa$. It only remains to show the identity (\ref{another random new id}), which we look to do by first considering $\phi \in C^{
    \infty}_{0}(\mathscr{O};\R)$ and defining $\psi \in C^{
    \infty}_{0}(\mathscr{O};\R^2)$ by $\psi^1 := \partial_2 \phi$, $\psi^2:= -\partial_1 \phi$. The similarity with the curl is swiftly noted, and the idea is to take the inner product of (\ref{a new id}) with $\psi$ and pass over the `curl' through integration by parts. Proceeding with this, $u$ satisfies the identity
    \begin{align} \nonumber
     \inner{u_t}{\psi} = \inner{u_0}{\psi} - \int_0^{t}\inner{\mathcal{P}\mathcal{L}_{u_s}u_s}{\psi}ds &- \nu\int_0^{t} \inner{Au_s}{\psi} ds \\ &+ \frac{1}{2}\int_0^{t}\sum_{i=1}^\infty \inner{\mathcal{P}B_i^2u_s}{\psi} ds - \sum_{i=1}^\infty\int_0^{t} \inner{\mathcal{P}B_iu_s}{\psi} dW^i_s\nonumber
\end{align}
and we look to simplify each term down to its respective one in (\ref{another random new id}). Expanding the first inner product into its components, $$ \inner{u_t}{\psi} = \inner{u^1_t}{\partial_2\phi}_{L^2(\mathscr{O};\R)} - \inner{u^2_t}{\partial_1\phi}_{L^2(\mathscr{O};\R)} = -\inner{\partial_2u^1_t}{\phi}_{L^2(\mathscr{O};\R)} + \inner{\partial_1u^2_t}{\phi}_{L^2(\mathscr{O};\R)} = \inner{w_t}{\phi}_{L^2(\mathscr{O};\R)}$$
having used that $\phi$ is compactly supported. The same integration by parts process will be used for the remaining terms. There is nothing to note for the initial condition, so moving on to the nonlinear term, we recall the continuity of $\mathcal{L}: W^{2,2}(\mathscr{O};\R^2) \rightarrow W^{1,2}(\mathscr{O};\R^2)$, so that $\mathcal{P}\mathcal{L}_{u_s}u_s \in  W^{1,2}(\mathscr{O};\R^2)$ hence $\textnormal{curl}\left(\mathcal{P}\mathcal{L}_{u_s}u_s\right) \in  L^{2}(\mathscr{O};\R)$. We use Lemma \ref{lemma for curl and P} to remove $\mathcal{P}$ from consideration, from which the well established result that $\textnormal{curl}\left( \mathcal{L}_{u_s}u_s\right) = \mathcal{L}_{u_s}w_s$ (see e.g. [\cite{robinson2016three}] Subsection 12.4) provides the required expression. We have thus far achieved that 
    \begin{align} \nonumber
      \inner{w_t}{\phi}_{L^2(\mathscr{O};\R)} = \inner{w_0}{\phi}_{L^2(\mathscr{O};\R)} &- \int_0^{t}\inner{\mathcal{L}_{u_s}w_s}{\phi}_{L^2(\mathscr{O};\R)}ds - \nu\int_0^{t} \inner{Au_s}{\psi} ds \\ &+ \frac{1}{2}\int_0^{t}\sum_{i=1}^\infty \inner{\mathcal{P}B_i^2u_s}{\psi} ds - \sum_{i=1}^\infty\int_0^{t} \inner{\mathcal{P}B_iu_s}{\psi} dW^i_s\nonumber.
\end{align}
For the stochastic integral there is nothing to prove given that $\textnormal{curl}(\mathcal{P}B_iu_s) = \mathcal{L}_iw_s$ was established in Subsection \ref{sub lie transport}. For the remaining terms we do not have sufficient regularity to conduct the integration by parts. To account for this we take an approximation of $u_s$ by functions $(f^n)$, $f^n \in W^{3,2}(\mathscr{O};\R^2)$ and $f^n \longrightarrow u_s$ in $W^{2,2}(\mathscr{O};\R^2)$. Firstly for the Stokes Operator, we have that \begin{align*}\inner{Af^n}{\psi} =  \inner{\textnormal{curl}\left(Af^n\right)}{\phi}_{L^2(\mathscr{O};\R)} = -\inner{\textnormal{curl}\left(\Delta f^n\right)}{\phi}_{L^2(\mathscr{O};\R)} &= -\inner{\Delta\textnormal{curl}\left( f^n\right)}{\phi}_{L^2(\mathscr{O};\R)}\\ &= \inner{\textnormal{curl}\left( f^n\right)}{\phi}_{W^{1,2}_0(\mathscr{O};\R)}\end{align*}
hence $$\inner{Au_s}{\psi} = \lim_{n \rightarrow \infty}\inner{Af^n}{\psi} = \lim_{n \rightarrow \infty}\inner{\textnormal{curl}\left( f^n\right)}{\phi}_{W^{1,2}_0(\mathscr{O};\R)} = \inner{w_s}{\phi}_{W^{1,2}_0(\mathscr{O};\R)}.$$ Similarly we have that $\textnormal{curl}\left(\mathcal{P}B_i^2f^n\right) =\mathcal{L}_i\left[\textnormal{curl}\left(B_if^n\right)\right] = \mathcal{L}_i^2\left[ \textnormal{curl}(f^n)\right]$. Thus $$\inner{\mathcal{P}B_i^2f^n}{\psi} = \inner{\mathcal{L}_i^2[\textnormal{curl}(f^n)]}{\phi}_{L^2(\mathscr{O};\R)} = -\inner{\mathcal{L}_i[\textnormal{curl}(f^n)]}{\mathcal{L}_i\phi}_{L^2(\mathscr{O};\R)}$$ having used that the $L^2(\mathscr{O};\R)$ adjoint of the densely defined $\mathcal{L}_i$ is $-\mathcal{L}_i$ (which follows from Lemma \ref{navier boundary nonlinear} in the one dimensional case too). Via the same approximation just used, we conclude that the identity (\ref{another random new id}) holds for $\phi \in C^{\infty}_0(\mathscr{O};\R)$. We use the density of this space in $W^{1,2}_0(\mathscr{O};\R)$ to finish the proof. 
\end{proof}

We shall use this identity to consider estimates on the vorticity, which then translate to estimates on the velocity.

\begin{proposition} \label{how not label}
    For any $t\in[0,T]$, $w$ satisfies the identity $$\norm{w_t}_{L^2(\mathscr{O};\R)}^2 + 2\nu\int_0^t\norm{w_s}_{W^{1,2}_0(\mathscr{O};\R)}^2ds = \norm{w_0}_{L^2(\mathscr{O};\R)}^2$$
    $\mathbb{P}-a.s.$.
\end{proposition}

\begin{proof}
    Just as we required in Proposition \ref{uniqueness prop}, energy methods demand that the weak solution satisfies an identity in $\left(W^{1,2}_0(\mathscr{O};\R) \right)^*$ to then apply Proposition \ref{rockner prop}. The nonlinear term enjoys better regularity here as opposed to the weak solution considered in Proposition \ref{uniqueness prop}, given that $u$ has additional regularity over $w$. In the same manner we deduce that $w$ satisfies, for every $t \in [0,T]$, the identity $$   w_t = w_0 - \int_0^t\mathcal{L}_{u_s}w_s\ ds + \nu\int_0^t \Delta w_s\, ds + \frac{1}{2}\int_0^t\sum_{i=1}^\infty \mathcal{L}_i^2w_s ds - \sum_{i=1}^\infty\int_0^t \mathcal{L}_iw_s dW^i_s $$
$\mathbb{P}-a.s.$ in $\left(W^{1,2}_0(\mathscr{O};\R )\right)^*$, with sufficient regularity to apply Proposition \ref{rockner prop} for the spaces $\mathcal{H}_1:=W^{1,2}_0(\mathscr{O};\R)$, $\mathcal{H}_2:=L^2(\mathscr{O};\R)$ and $\mathcal{H}_3:=\left(W^{1,2}_0(\mathscr{O};\R)\right)^*$. Therefore $w$ satisfies 
\begin{align}\nonumber
    \norm{w_t}_{L^2(\mathscr{O};\R)}^2 = \norm{w_0}_{L^2(\mathscr{O};\R)}^2 &- 2\int_0^t\inner{\mathcal{L}_{u_s}w_s}{w_s}_{\left(W^{1,2}_0(\mathscr{O};\R)\right)^* \times W^{1,2}_0(\mathscr{O};\R)}ds\\ \nonumber &+ 2\nu\int_0^t\inner{\Delta w_s}{w_s}_{\left(W^{1,2}_0(\mathscr{O};\R)\right)^* \times W^{1,2}_0(\mathscr{O};\R)}ds\\ \nonumber
    &+ \int_0^t\sum_{i=1}^\infty \left( \inner{\mathcal{L}_i^2 w_s}{w_s}_{\left(W^{1,2}_0(\mathscr{O};\R)\right)^* \times W^{1,2}_0(\mathscr{O};\R)} + \norm{\mathcal{L}_iw_s}^2_{L^2(\mathscr{O};\R)}\right)ds\\
    &- 2\sum_{i=1}^\infty \int_0^t\inner{\mathcal{L}_i w_s}{w_s}_{L^2(\mathscr{O};\R)}dW^i_s. \label{put it back!}
\end{align}
By definition $$\inner{\mathcal{L}_{u_s}w_s}{w_s}_{\left(W^{1,2}_0(\mathscr{O};\R)\right)^* \times W^{1,2}_0(\mathscr{O};\R)} = \inner{\mathcal{L}_{u_s}w_s}{w_s}_{L^{4/3}(\mathscr{O};\R) \times L^{4}(\mathscr{O};\R)}$$ which is simply zero again from Lemma \ref{navier boundary nonlinear}, whilst $$\inner{\Delta w_s}{w_s}_{\left(W^{1,2}_0(\mathscr{O};\R)\right)^* \times W^{1,2}_0(\mathscr{O};\R)} = -\norm{w_s}_{W^{1,2}_0(\mathscr{O};\R)}^2$$ is also directly from the definition. Similarly $$ \inner{\mathcal{L}_i^2 w_s}{w_s}_{\left(W^{1,2}_0(\mathscr{O};\R)\right)^* \times W^{1,2}_0(\mathscr{O};\R)} = -\inner{\mathcal{L}_i w_s}{\mathcal{L}_i w_s}_{L^2(\mathscr{O};\R)} = -\norm{\mathcal{L}_iw_s}^2_{L^2(\mathscr{O};\R)}$$ hence \begin{align*}
    \inner{\mathcal{L}_i^2 w_s}{w_s}_{\left(W^{1,2}_0(\mathscr{O};\R)\right)^* \times W^{1,2}_0(\mathscr{O};\R)} + \norm{\mathcal{L}_iw_s}^2_{L^2(\mathscr{O};\R)}  = 0
\end{align*}
whilst we also have that 
$$\inner{\mathcal{L}_i w_s}{w_s}_{L^2(\mathscr{O};\R)} = 0$$ again from Lemma \ref{navier boundary nonlinear}. Putting this all back into (\ref{put it back!}) gives the result.
\end{proof}

\begin{corollary} \label{boundycorollary}
    There exists a constant $C$ independent of $\nu$, $\omega$, such that $$\sup_{r \in [0,T]}\norm{u_r}_1^2 \leq C\norm{w_0}_{L^2(\mathscr{O};R)}^2$$
    $\mathbb{P}-a.s$. 
\end{corollary}

\begin{proof}
    This now follows immediately from the classical control on the gradient of functions in $\bar{W}^{1,2}_{\sigma}$ by their curl (see for example [\cite{kelliher2006navier}] Lemma 3.1, a result attributed to Yudovich). 
\end{proof}

\newpage

\subsection{Tightness} \label{sub tighty}
We now consider a sequence of viscosities $(\nu^k)$ monotonically decreasing to zero as in Theorem \ref{theorem for limit of navier stokes is euler}, and corresponding strong solutions of (\ref{projected Ito Salt}) for $\alpha:= 2\kappa$ given by $u^k$. Without loss of generality we assume $\nu^k < 1$ for all $k$. 

\begin{proposition} \label{tightness in D prop}
    For any sequence of stopping times $(\gamma_n)$ with $\gamma_n: \Omega \rightarrow [0,T]$, and any $\phi \in \bar{W}^{1,2}_{\sigma}$, $$\lim_{\delta \rightarrow 0^+}\sup_{k \in \N}\mathbbm{E}\left( \left\vert \inner{u^k_{(\gamma_k + \delta)\wedge T} - u^k_{\gamma_k}}{\phi} \right\vert\right) = 0.$$
\end{proposition}

\begin{proof}
    We use the identity satisfied by the inner product, which is
     \begin{align} \nonumber
     \inner{u^k_{(\gamma_k + \delta)\wedge T} - u^k_{\gamma_k}}{\phi} &= - \int_{\gamma_k}^{(\gamma_k + \delta)\wedge T}\inner{\mathcal{L}_{u^k_s}u^k_s}{\phi}ds - \nu^k\int_{\gamma_k}^{(\gamma_k + \delta)\wedge T} \inner{Au^k_s}{\phi} ds \\ &+ \frac{1}{2}\int_{\gamma_k}^{(\gamma_k + \delta)\wedge T}\sum_{i=1}^\infty \inner{B_i^2u^k_s}{\phi} ds - \sum_{i=1}^\infty\int_{\gamma_k}^{(\gamma_k + \delta)\wedge T} \inner{B_iu^k_s}{\phi} dW^i_s\nonumber.
\end{align}
having immediately passed $\mathcal{P}$ over to $\phi$. Inspecting each term, from Lemma \ref{the 2D bound lemma} we have that
$$ \left\vert\inner{\mathcal{L}_{u^k_s}u^k_s}{\phi}\right\vert \leq c\norm{u^k_s}^{\frac{1}{2}}\norm{u^k_s}_1^{\frac{3}{2}}\norm{\phi}_1 \leq c\norm{w_0}_{L^2(\mathscr{O};\R)}^2$$ where we have used Corollary \ref{boundycorollary} and will continue to absorb norms of $\phi$ into our constant. Similarly we have that $\inner{Au^k_s}{\phi} = \inner{u^k_s}{\phi}_H$ and as $\nu^k < 1$, $$\nu^k\left \vert \inner{Au^k_s}{\phi}\right\vert \leq c\norm{u^k_s}_1\norm{\phi}_1 \leq c\norm{w_0}_{L^2(\mathscr{O};\R)}$$ and using that $\inner{B_i^2u^k_s}{\phi} = \inner{B_iu^k_s}{B_i^*\phi}$ we achieve the same bound there. Using that $\norm{w_0}_{L^2(\mathscr{O};\R)} \leq 1 + \norm{w_0}_{L^2(\mathscr{O};\R)}^2$, we have so far shown that 
\begin{align} \nonumber
     \left\vert\inner{u^k_{(\gamma_k + \delta)\wedge T} - u^k_{\gamma_k}}{\phi}\right\vert \leq  \int_{\gamma_k}^{(\gamma_k + \delta)\wedge T}c(1 + \norm{w_0}_{L^2(\mathscr{O};\R)}^2) ds + \left\vert\sum_{i=1}^\infty\int_{\gamma_k}^{(\gamma_k + \delta)\wedge T} \inner{B_iu^k_s}{\phi} dW^i_s \right\vert.
\end{align}
Taking expectation and applying the Burkholder-Davis-Gundy Inequality, observe firstly that \begin{align*}\mathbbm{E}\left( \left\vert\sum_{i=1}^\infty\int_{\gamma_k}^{(\gamma_k + \delta)\wedge T} \inner{B_iu^k_s}{\phi} dW^i_s \right\vert\right) &\leq c\mathbbm{E}\left( \int_{\gamma_k}^{(\gamma_k + \delta)\wedge T} \sum_{i=1}^\infty \inner{B_iu^k_s}{\phi}^2 ds \right)^{\frac{1}{2}}\\ &\leq c\mathbbm{E}\left( \int_{\gamma_k}^{(\gamma_k + \delta)\wedge T} \norm{w_0}_{L^2(\mathscr{O};\R)}^2 ds \right)^{\frac{1}{2}}\\
&\leq c\left[\mathbbm{E} \int_{\gamma_k}^{(\gamma_k + \delta)\wedge T} \norm{w_0}_{L^2(\mathscr{O};\R)}^2 ds \right]^{\frac{1}{2}}\end{align*}
Using $\delta$ as an upper bound for the length of the integrals, we have that $$\mathbbm{E}\left(\left\vert\inner{u^k_{(\gamma_k + \delta)\wedge T} - u^k_{\gamma_k}}{\phi}\right\vert\right) \leq c\delta\mathbbm{E}\left( \norm{w_0}_{L^2(\mathscr{O};\R)}^2\right) + c\left[\delta\mathbbm{E}\left( \norm{w_0}_{L^2(\mathscr{O};\R)}^2\right)\right]^{\frac{1}{2}}$$ where the result now follows from the fact that $u_0 \in L^2\left(\Omega;\bar{W}^{1,2}_{\sigma}\right)$ implies $w_0 \in L^2\left(\Omega;L^2(\mathscr{O};\R)\right)$.
\end{proof}

\begin{corollary} \label{tightness in D corollary}
    The sequence of the laws of $(u^k)$ is tight in the space of probability measures over $\mathcal{D}\left([0,T];L^2_\sigma \right)$.
\end{corollary}

\begin{proof}
    We apply Lemma \ref{lemma for D tight} in the case of $V:=\mathcal{H}_1:=\bar{W}^{1,2}_{\sigma}$ and $\mathcal{H}_2:= L^2_{\sigma}$. Condition (\ref{first condition primed}) is satisfied from Corollary \ref{boundycorollary}. As for (\ref{second condition primed}), we use Chebyshev's Inequality to see that $$\mathbbm{P}\left(\left\{
    \omega \in \Omega: \left\vert \left\langle u^k_{(\gamma_k + \delta) \wedge T} -u^k_{\gamma_k }   , \phi    \right\rangle \right\vert > \varepsilon \right\}\right) \leq \frac{1}{\varepsilon}\mathbbm{E}\left(\left\vert\inner{u^k_{(\gamma_k + \delta)\wedge T} - u^k_{\gamma_k}}{\phi}\right\vert\right)$$
    from which the condition follows courtesy of Proposition \ref{tightness in D prop}.
\end{proof}

\subsection{Passage to the Stochastic Euler Equation} \label{passage to stochy euler}

Using the uniform boundedness from Corollary \ref{boundycorollary} and the tightness due to Corollary \ref{tightness in D corollary}, as well as the progressive measurability of the processes $(u^k)$ in $\bar{W}^{2,2}_{\alpha}$, then just as we obtained Theorem \ref{theorem for new prob space} we acquire the following. One requires taking a subsequence, as in the statement of Theorem \ref{theorem for limit of navier stokes is euler}, though for simplicity we index it again by $k$ here. 

\begin{theorem} \label{theorem for new prob space 2}
There exists a filtered probability space $\left(\tilde{\Omega},\tilde{\mathcal{F}},(\tilde{\mathcal{F}}_t), \tilde{\mathbb{P}}\right)$, a cylindrical Brownian Motion $\tilde{\mathcal{W}}$ over $\mathfrak{U}$ with respect to $\left(\tilde{\Omega},\tilde{\mathcal{F}},(\tilde{\mathcal{F}}_t), \tilde{\mathbb{P}}\right)$, a random variable $\tilde{u}_0:\tilde{\Omega} \rightarrow \bar{W}^{1,2}_{\sigma}$, a sequence of processes $(\tilde{u}^k)$, $\tilde{u}^k:\tilde{\Omega} \times [0,T] \rightarrow \bar{W}^{2,2}_{\alpha}$ is progressively measurable and a progressively measurable process $\tilde{u}:\tilde{\Omega} \times [0,T] \rightarrow \bar{W}^{1,2}_{\sigma}$ such that:
\begin{enumerate}
  \item \label{le 1}$\tilde{u}_0$ has the same law as $u_0$;
    \item \label{le 2} For each $k \in \N$ and $t\in[0,T]$, $\tilde{u}^k$ satisfies the identity
    \begin{equation} \nonumber
    \tilde{u}^k_t = \tilde{u}_0 - \int_0^t\mathcal{P}\mathcal{L}_{\tilde{u}^k_s}\tilde{u}^k_s\ ds - \nu^k\int_0^t  A \tilde{u}^k_s\, ds + \frac{1}{2}\int_0^t\sum_{i=1}^\infty \mathcal{P}B_i^2\tilde{u}^k_s ds - \int_0^t \mathcal{P}B\tilde{u}^k_s d\tilde{\mathcal{W}}_s 
\end{equation}
$\tilde{\mathbb{P}}-a.s.$ in $L^2_{\sigma}$;
\item \label{le 3} For $\tilde{\mathbb{P}}-a.e$ $\omega$, $\tilde{u}^k(\omega) \rightarrow \tilde{u}(\omega)$ in $\mathcal{D}\left([0,T]; L^2_{\sigma} \right)$;
\item \label{le 4} For $\tilde{\mathbb{P}}-a.e$ $\omega$, $\tilde{u}(\omega) \in L^\infty\left([0,T];\bar{W}^{1,2}_{\sigma} \right)$.
\end{enumerate}

\end{theorem}

In fact we can immediately upgrade the convergence in item \ref{le 3} to that in $C\left([0,T];L^2_{\sigma}\right)$. Recall that the restriction of the $\mathcal{D}\left([0,T];L^2_{\sigma}\right)$ topology to $C\left([0,T];L^2_{\sigma}\right)$ is equivalent to that induced by the supremum norm (e.g. [\cite{jakubowski1986skorokhod}] Proposition 1.6), which is a Banach space. Hence for $\tilde{\mathbb{P}}-a.e$ $\omega$, the convergent sequence $(\tilde{u}^k(\omega))$ in $\mathcal{D}\left([0,T]; L^2_{\sigma} \right)$ is Cauchy in this space, hence too in $C\left([0,T]; L^2_{\sigma} \right)$, thus exhibits a limit in $C\left([0,T]; L^2_{\sigma} \right)$ which is of course equal to that in $\mathcal{D}\left([0,T]; L^2_{\sigma} \right)$ being $\tilde{u}(\omega)$. Thus $\tilde{u}\in C\left([0,T]; L^2_{\sigma} \right) $ $\mathbbm{P}-a.s.$ and has the necessary regularity to be a martingale weak solution of the equation (\ref{projected Ito Salt Euler}). Indeed each $\tilde{u}^k$ is a strong solution of the equation (\ref{projected Ito Salt}) relative to this new probability space and cylindrical Brownian Motion (thus is \textit{the} strong solution), so in order to prove Theorem \ref{theorem for limit of navier stokes is euler} it only remains to show that $\tilde{u}$ satisfies the identity (\ref{newid1}).

\begin{proposition} \label{not another one!!!!}
    The process $\tilde{u}$ satisfies the identity (\ref{newid1}): that is \begin{align} 
     \inner{\tilde{u}_t}{\phi} = \inner{\tilde{u}_0}{\phi} - \int_0^{t}\inner{\mathcal{L}_{\tilde{u}_s}\tilde{u}_s}{\phi}ds + \frac{1}{2}\int_0^{t}\sum_{i=1}^\infty \inner{B_i\tilde{u}_s}{B_i^*\phi} ds - \int_0^{t} \inner{B\tilde{u}_s}{\phi} d\tilde{\mathcal{W}}_s\nonumber
\end{align}
holds for every $\phi \in \bar{W}^{1,2}_{\sigma}$ $\tilde{\mathbb{P}}-a.s.$ in $\R$ for all $t \in [0,T]$.
\end{proposition}

\begin{proof}
    We fix arbitrary $\psi \in \bar{W}^{2,2}_{\alpha}$ and $t\in[0,T]$, with the aim to show the result for $\psi$ and pass to the limit for any $\phi \in \bar{W}^{1,2}_{\sigma}$. Taking the inner product with $\psi$ in item \ref{le 2} of Theorem \ref{theorem for new prob space 2}, we have that
     \begin{align*}
    \inner{\tilde{u}^k_t}{\psi} = \inner{\tilde{u}_0}{\psi} - \int_0^t\inner{\mathcal{L}_{\tilde{u}^k_s}\tilde{u}^k_s}{\psi}\ ds &- \nu^k\int_0^t \inner{ A \tilde{u}^k_s}{\psi}\, ds\\ &+ \frac{1}{2}\int_0^t\sum_{i=1}^\infty \inner{B_i^2\tilde{u}^k_s}{\psi} ds - \int_0^t \inner{B\tilde{u}^k_s}{\psi} d\tilde{\mathcal{W}}_s 
\end{align*}
holds $\tilde{\mathbbm{P}}-a.s.$ in $\R$, so we look to show the convergence of each term in turn $\tilde{\mathbb{P}}-a.s.$. It is immediate that $\inner{\tilde{u}^k_t}{\psi}$ converges to $\inner{\tilde{u}_t}{\psi}$ from the convergence in $C\left([0,T];L^2_{\sigma}\right)$. For the nonlinear term, we have that \begin{align*}
    \left\vert \inner{\mathcal{L}_{\tilde{u}^k_s}\tilde{u}^k_s}{\psi} - \inner{\mathcal{L}_{\tilde{u}_s}\tilde{u}_s}{\psi}\right\vert &= \left\vert \inner{\mathcal{L}_{\tilde{u}^k_s- \tilde{u}_s}\tilde{u}^k_s + \mathcal{L}_{\tilde{u}_s}(\tilde{u}^k_s - \tilde{u}_s)}{\psi}\right\vert\\ &\leq \left\vert \inner{\mathcal{L}_{\tilde{u}^k_s- \tilde{u}_s}\tilde{u}^k_s }{\psi}\right\vert + \left\vert \inner{ \mathcal{L}_{\tilde{u}_s}(\tilde{u}^k_s - \tilde{u}_s)}{\psi}\right\vert.
\end{align*}
We treat both cases with Lemma \ref{navier boundary nonlinear} and H\"{o}lder's Inequality:
\begin{align*}
    \left\vert \inner{\mathcal{L}_{\tilde{u}^k_s- \tilde{u}_s}\tilde{u}^k_s }{\psi}\right\vert = \left\vert \inner{\tilde{u}^k_s }{\mathcal{L}_{\tilde{u}^k_s- \tilde{u}_s}\psi}\right\vert \leq \norm{\tilde{u}^k_s}_{L^6}\norm{\mathcal{L}_{\tilde{u}^k_s- \tilde{u}_s}\psi}_{L^{6/5}} &\leq c\sum_{l=1}^2\norm{\tilde{u}^k_s}_{L^6}\norm{\tilde{u}^k_s- \tilde{u}_s}\norm{\partial_l \psi}_{L^3}\\ &\leq c\norm{\tilde{u}^k_s}_{1}\norm{\tilde{u}^k_s- \tilde{u}_s}\norm{\psi}_{W^{2,2}}
\end{align*}
whilst
\begin{align*}
    \left\vert \inner{ \mathcal{L}_{\tilde{u}_s}(\tilde{u}^k_s - \tilde{u}_s)}{\psi}\right\vert = \left\vert \inner{ \tilde{u}^k_s - \tilde{u}_s}{\mathcal{L}_{\tilde{u}_s}\psi}\right\vert &\leq c\sum_{l=1}^2\norm{\tilde{u}^k_s - \tilde{u}_s}\norm{\tilde{u}_s}_{L^4}\norm{\partial_l\psi}_{L^4} \\&\leq c\norm{\tilde{u}^k_s - \tilde{u}_s}\norm{\tilde{u}_s}_{1}\norm{\psi}_{W^{2,2}}
\end{align*}

We now comment on the regularity of $\tilde{u}^k$; the estimates shown in Subsection \ref{sub estimates indep viscos} were done so for the unique strong solution of the equation (\ref{projected Ito Salt}) on the original probability space, though are identically obtained for the unique strong solution on this new probability space. Furthermore from Corollary \ref{boundycorollary} we have that there exists a constant $C$ such that for every $k$ and $\tilde{\mathbb{P}}-a.e.$ $\omega \in \tilde{\Omega}$, $$\sup_{r \in [0,T]}\norm{\tilde{u}^k_r(\omega)}_1^2 \leq C\norm{\tilde{w}_0(\omega)}_{L^2(\mathscr{O};\R)}^2$$
where $\tilde{w}_0$ is $\textnormal{curl}\tilde{u}_0$. With the same type of argument that we saw in Proposition \ref{final regularity of solutions}, taking a weak* convergent subsequence in $L^\infty\left([0,T];\bar{W}^{1,2}_{\sigma}\right)$ and using Banach-Alaoglu, one also has that $$\sup_{r \in [0,T]}\norm{\tilde{u}_r(\omega)}_1^2 \leq C\norm{\tilde{w}_0(\omega)}_{L^2(\mathscr{O};\R)}^2.$$
As $\tilde{u}_0$ has the same law as $u_0$ in $\bar{W}^{1,2}_{\sigma}$ then $\tilde{u}_0 \in L^2\left(\tilde{\Omega};\bar{W}^{1,2}_{\sigma}\right)$ so $\tilde{w}_0 \in L^2\left(\tilde{\Omega};L^2(\mathscr{O};\R)\right)$. One therefore obtains that
\begin{align*}
    \left\vert \int_0^t\inner{\mathcal{L}_{\tilde{u}^k_s}\tilde{u}^k_s}{\psi}ds - \int_0^t\inner{\mathcal{L}_{\tilde{u}_s}\tilde{u}_s}{\psi}ds\right\vert &\leq c\int_0^t\left(\norm{\tilde{u}^k_s}_{1} + \norm{\tilde{u}_s}_{1}\right)\norm{\tilde{u}^k_s- \tilde{u}_s}\norm{\psi}_{W^{2,2}} \\&\leq c\norm{\tilde{w}_0}_{L^2(\mathscr{O};\R)}\norm{\psi}_{W^{2,2}}\int_0^t\norm{\tilde{u}^k_s- \tilde{u}_s} ds
\end{align*}
which we know to converges to zero $\tilde{\mathbb{P}}-a.s.$, again using the convergence in $C\left([0,T];L^2_{\sigma}\right)$. For the Stokes Operator term we have that $$\left\vert \inner{ A \tilde{u}^k_s}{\psi}\right\vert = \left\vert \inner{ \tilde{u}^k_s}{\psi}_H\right\vert \leq c\norm{\tilde{u}^k_s}_1\norm{\psi}_1 \leq c\norm{\tilde{w}_0}_{L^2(\mathscr{O};\R)}\norm{\psi}_1$$ thus $$\left\vert \nu^k\int_0^t \inner{ A \tilde{u}^k_s}{\psi}\, ds\right\vert \leq ct\norm{\tilde{w}_0}_{L^2(\mathscr{O};\R)}\norm{\psi}_1\nu^k$$ which is convergent to zero $\tilde{\mathbb{P}}-a.s.$. Next observe that
\begin{align*}
    \left\vert \inner{B_i^2\tilde{u}^k_s}{\psi} -  \inner{B_i^2\tilde{u}_s}{\psi} \right\vert =  \left\vert \inner{B_i^2(\tilde{u}^k_s - \tilde{u}_s)}{\psi}\right\vert =  \left\vert \inner{\tilde{u}^k_s - \tilde{u}_s}{(B_i^*)^2\psi}\right\vert \leq c_i\norm{\tilde{u}^k_s - \tilde{u}_s}\norm{\psi}_{W^{2,2}}
\end{align*}
from which we again see the desired convergence in this term. In the stochastic integral we have to pass to a subsequence. Indeed we can only show the convergence in expectation; to this end we have that \begin{align*} \tilde{\mathbbm{E}}\left(\left\vert \int_0^t \inner{B\tilde{u}^k_s}{\psi} d\tilde{\mathcal{W}}_s 
 - \int_0^t \inner{B\tilde{u}_s}{\psi} d\tilde{\mathcal{W}}_s \right\vert\right) &\leq c\tilde{\mathbbm{E}}\left(\int_0^t\sum_{i=1}^\infty \inner{B_i(\tilde{u}^k_s - \tilde{u}_s)}{\psi}^2 ds \right)^{\frac{1}{2}}\\
 &\leq c\tilde{\mathbbm{E}}\left(\int_0^t \norm{\tilde{u}^k_s - \tilde{u}_s}^2\norm{\psi}_1^2 ds \right)^{\frac{1}{2}}\\
 &\leq c\norm{\psi}_1\left[\tilde{\mathbbm{E}}\left(\int_0^t \norm{\tilde{u}^k_s - \tilde{u}_s}^2 ds \right)\right]^{\frac{1}{2}}.
 \end{align*}
The fact that this quantity approaches zero as $k \rightarrow \infty$ now follows from the Dominated Convergence Theorem. Thus we can take a subsequence, indexed by $(k_j)$, such that
$$\int_0^t \inner{B\tilde{u}^{k_j}_s}{\psi} d\tilde{\mathcal{W}}_s \longrightarrow \int_0^t \inner{B\tilde{u}_s}{\psi} d\tilde{\mathcal{W}}_s $$ $\tilde{\mathbb{P}}-a.s.$. Moreover, taking the $\tilde{\mathbb{P}}-a.s.$ convergence along this subsequence for all terms, we realise that $\tilde{u}$ satisfies \begin{align} 
     \inner{\tilde{u}_t}{\psi} = \inner{\tilde{u}_0}{\psi} - \int_0^{t}\inner{\mathcal{L}_{\tilde{u}_s}\tilde{u}_s}{\psi}ds + \frac{1}{2}\int_0^{t}\sum_{i=1}^\infty \inner{B_i\tilde{u}_s}{B_i^*\psi} ds - \int_0^{t} \inner{B\tilde{u}_s}{\psi} d\tilde{\mathcal{W}}_s\nonumber
\end{align}
$\tilde{\mathbb{P}}-a.s.$ in $\R$. It only remains to show this identity for $\phi \in \bar{W}^{1,2}_{\sigma}$, for which we choose a sequence of such $\psi$ approximating $\phi$ in $\bar{W}^{1,2}_{\sigma}$. The convergence in all terms other than the stochastic integral is clear, whilst for the stochastic integral we know it only for a subsequence using a Stochastic Dominated Convergence Theorem (e.g. [\cite{goodair2022stochastic}] Lemma 1.6.15), much like what we have just demonstrated. This completes the proof.
\end{proof}

\section{Conclusion} \label{section conclusion}

We now address some key questions surrounding this paper and its context in the literature, as mentioned in the introduction.

\begin{itemize}
    \item How do the Navier boundary conditions solve the problems faced by the no-slip condition for the existence of strong solutions?
\end{itemize}

The details are of course implicit in the proof of Theorem \ref{theorem 2D strong}, which we highlight now. The author's paper [\cite{goodair2022navier}] brought attention to the difficulty in controlling the term 
\begin{equation}\label{ironic new label} \sum_{i=1}^\infty \left(\inner{\mathcal{P}_n\mathcal{P}B_i^2u^n_s}{u^n_s}_1 + \norm{\mathcal{P}_n\mathcal{P}B_iu^n_s}_1^2 \right)\end{equation}
in the no-slip case, where $(\mathcal{P}_n)$ here are projections onto eigenfunctions of the Stokes Operator belonging to $W^{1,2}_{\sigma}$, which is the intersection of $\bar{W}^{1,2}_{\sigma}$ with $W^{1,2}_{0}(\mathscr{O};\R^2)$. This is the counterpart of (\ref{we will label it for consideration}) needed in showing uniform estimates of the Galerkin Approximation in the energy norm of the strong solution. The difficulty is the fact that $\mathcal{P}B_i^2u^n_s$ and $\mathcal{P}B_iu^n_s$ do not themselves belong to $W^{1,2}_{\sigma}$, and $\mathcal{P}_n$ is \textit{only self-adjoint on $W^{1,2}_{\sigma}$ and not $\bar{W}^{1,2}_{\sigma}$}. This owes to the fact that the system of eigenfunctions in $W^{1,2}_{\sigma}$ is not a basis for $\bar{W}^{1,2}_{\sigma}$, whilst the system used in the present paper in any $\bar{W}^{2,2}_{\alpha}$ is a basis for $\bar{W}^{1,2}_{\sigma}$. Therefore we cannot bound (\ref{ironic new label}) by an expression \begin{equation} \label{shouldve labelled} \sum_{i=1}^\infty \left(\inner{\mathcal{P}B_i^2u^n_s}{u^n_s}_1 + \norm{\mathcal{P}B_iu^n_s}_1^2 \right)\end{equation}
as we do for (\ref{we will label it for consideration}), formally in the $H$ inner product but the approach works in either. The direct presence of the $\mathcal{P}_n$ disables our commutator arguments used in the verification of (\ref{assumpty 12}), as good commutator properties between $\mathcal{P}_n$ and $B_i$ are not apparent. We note the key property that $\mathcal{P}B_i = \mathcal{P}B_i\mathcal{P}$ used in [\cite{goodair2022navier}], so that we can cope with the presence of $\mathcal{P}$ and still achieve some cancellation of the top order derivative. A sufficient estimate on (\ref{ironic new label}) thus remains open.

\begin{itemize}
    \item Why is the condition $\alpha \geq \kappa$ necessary in Theorem \ref{theorem 2D strong}?
\end{itemize}

This property facilitated the use of the abstract framework of [\cite{goodair2022existence1}] which allowed us to lighten up the technical details given in this paper, though it may be unclear as to whether or not this condition was actually necessary. The alternative route to achieve Proposition \ref{uniform boundedness proposition} would be to work directly with the $\inner{\cdot}{\cdot}_1$ inner product. Using SALT noise as motivation, the issue is the same term discussed in the previous point, (\ref{shouldve labelled}). Given the presence of $\mathcal{P}$ our only option is to introduce the Stokes operator into these terms with Lemma \ref{greens for navier} and then use the commutator estimates of [\cite{goodair2022navier}] (see Proposition 1.29). Doing this we obtain that
$$ \inner{\mathcal{P}B_i^2u^n_s}{u^n_s}_1 = \inner{\mathcal{P}B_i^2u^n_s}{Au^n_s} + \inner{(\kappa - \alpha)\mathcal{P}B_i^2u^n_s}{u^n_s}_{L^2(\partial \mathscr{O} ;\R)}$$ where we see that the issue is in the control of the $\mathcal{P}B_i^2u^n_s$ term on the boundary, as with the usual employment of (\ref{inequality from Lions}) we would require the $W^{3,2}$ norm of $u^n_s$ which is too much. The assumption that $\alpha \geq \kappa$ allows us use of the $H$ inner product in the It\^{o} Formula (Proposition \ref{rockner prop}), removing the need to manage this problematic boundary condition but of course this isn't without a cost. Given that $\mathcal{P}B_iu^n_s \notin \bar{W}^{2,2}_{\alpha}$ then we cannot say that $$\norm{\mathcal{P}B_iu^n_s}_H^2 = \inner{\mathcal{P}B_iu^n_s}{A\mathcal{P}B_iu^n_s}$$ 
where the right hand side is the term that we need to reach to apply our commutator estimates. Thus we have to treat the boundary integral $- \inner{(\kappa - \alpha)\mathcal{P}B_iu^n_s}{\mathcal{P}B_iu^n_s}_{L^2(\partial \mathscr{O} ;\R)}$ directly, as we saw in the estimate of (\ref{we will label it for consideration}), which would not appear if we applied the It\^{o} Formula for $\inner{\cdot}{\cdot}_1$. The significance is that this boundary integral is controllable, as the derivatives are shared so the use of (\ref{inequality from Lions}) now only demands the $W^{2,2}$ norm of $u^n_s$ in an acceptable way. There is the more hidden additional cost of the boundary integrals in $\norm{u^n_t}_H^2$, $\norm{u^n_0}_H^2$ which are immediately dealt with once we have established the equivalence of $\norm{\cdot}_1$ and $\norm{\cdot}_H$ (and \textit{only} because of this equivalence). Thus the assumption that $\alpha \geq \kappa$ is necessary.

\begin{itemize}
    \item What types of noise satisfy the assumptions in Subsection \ref{subsection assumptions}? 
\end{itemize}

It is clear that the assumptions hold for any additive noise in $W^{1,2}(\mathscr{O};\R^2)$ or any linear $\mathcal{G}_i$, $\mathcal{Q}_i = \mathcal{G}_i$ or $0$ which is bounded as an operator $L^2(\mathscr{O};\R^2) \rightarrow L^2(\mathscr{O};\R^2)$ and $W^{1,2}(\mathscr{O};\R^2) \rightarrow W^{1,2}(\mathscr{O};\R^2)$. Such a Nemystkii type noise has proven popular in the study of fluid SPDEs, see for instance ([\cite{glatt2009strong}, \cite{glatt2012local}]. In addition an It\^{o} type transport noise with sufficiently small gradient dependency relative to viscosity is also valid. A very interesting question is whether or not the assumptions hold for a purely transport Stratonovich noise. In [\cite{goodair2023zero}] the majority of the assumptions are verified in the case of $\mathcal{G}_i = \mathcal{Q}_i := \mathcal{P}\mathcal{L}_{\xi_i}$ for $\xi_i$ satisfying the same assumptions as in the SALT Noise, though (\ref{assumpty 12}) and (\ref{assumpty13}) remain to be verified. Most critically is (\ref{assumpty 12}) pertaining to the key term that we have been considering, (\ref{we will label it for consideration}). For the SALT noise the property that $\mathcal{P}B_i = \mathcal{P}B_i\mathcal{P}$ is vital; our estimate, as in (\ref{suffices to control}), boils down to controlling $$\inner{\mathcal{P}B_i^2\phi}{A \phi} + \inner{\mathcal{P}B_i \phi}{A\mathcal{P}B_i \phi} $$ which was done in [\cite{goodair2022navier}] Lemma 2.7, relying heavily on this property that $\mathcal{P}B_i = \mathcal{P}B_i\mathcal{P}$. The question is whether or not we could achieve the necessary bound for \begin{equation}\nonumber \inner{(\mathcal{P}\mathcal{L}_{\xi_i})^2\phi}{A \phi} + \inner{\mathcal{P}\mathcal{L}_{\xi_i} \phi}{A\mathcal{P}\mathcal{L}_{\xi_i} \phi}.\end{equation} By taking the $\mathcal{P}$ over, using that $A = A\mathcal{P}$ and $\mathcal{L}_{\xi_i}^* = -\mathcal{L}_{\xi_i}$ there is no issue in reducing this to \begin{equation} \label{now there is a number} - \inner{\mathcal{P}\mathcal{L}_{\xi_i}\phi}{\mathcal{L}_{\xi_i}A \phi} + \inner{\mathcal{P}\mathcal{L}_{\xi_i} \phi}{A\mathcal{L}_{\xi_i} \phi} \end{equation}
or simply $$ \inner{\mathcal{P}\mathcal{L}_{\xi_i}\phi}{[A,\mathcal{L}_{\xi_i}]\phi}$$ where $[A,\mathcal{L}_{\xi_i}]$ represents the commutator $A \mathcal{L}_{\xi_i} - \mathcal{L}_{\xi_i}A$. It is sufficient to show that $[A,\mathcal{L}_{\xi_i}]$ is of second order, but although that would be possible for $[\Delta,\mathcal{L}_{\xi_i}]$ the presence of $\mathcal{P}$ means this is not clear. We hope to make a compelling point regarding the importance of SALT noise here. To simply illustrate why the property that $\mathcal{P}B_i = \mathcal{P}B_i\mathcal{P}$ is so important, \textit{assume that $\mathcal{P}\mathcal{L}_{\xi_i} = \mathcal{P}\mathcal{L}_{\xi_i}\mathcal{P}$ (which is untrue!)}. Then we could write (\ref{now there is a number}) as \begin{align*} -\inner{\mathcal{P}\mathcal{L}_{\xi_i}\phi}{\mathcal{P}\mathcal{L}_{\xi_i}A \phi} + \inner{\mathcal{P}\mathcal{L}_{\xi_i}\phi}{A\mathcal{L}_{\xi_i} \phi} &= 
\inner{\mathcal{P}\mathcal{L}_{\xi_i}\phi}{\mathcal{P}\mathcal{L}_{\xi_i}\Delta \phi} - \inner{\mathcal{P}\mathcal{L}_{\xi_i}\phi}{\Delta\mathcal{L}_{\xi_i} \phi}\\ &= \inner{\mathcal{P}\mathcal{L}_{\xi_i}\phi}{\mathcal{L}_{\xi_i}\Delta \phi} - \inner{\mathcal{P}\mathcal{L}_{\xi_i}\phi}{\Delta\mathcal{L}_{\xi_i} \phi} \end{align*}
or simply $$ -\inner{\mathcal{P}\mathcal{L}_{\xi_i}\phi}{[\Delta,\mathcal{L}_{\xi_i}]\phi}$$
which as suggested is controllable. The proof for $B_i$ is more complicated, but this is the essential idea. It seems, therefore, that the best method to show the necessary bound for $\mathcal{L}_{\xi_i}$ is to actually write $$\mathcal{L}_{\xi_i} = B_i - \mathcal{T}_{\xi_i}$$ then use the fact that $\mathcal{T}_{\xi_i}$ is of zeroth order (so unproblematic) and then treat the $B_i$ as described. More precisely, for the first term in (\ref{now there is a number}), we would write
\begin{align*}
  -\inner{\mathcal{P}\mathcal{L}_{\xi_i}\phi}{\mathcal{P}\mathcal{L}_{\xi_i}A \phi} &= -\inner{\mathcal{P}\mathcal{L}_{\xi_i}\phi}{\mathcal{P}B_iA \phi} + \inner{\mathcal{P}\mathcal{L}_{\xi_i}\phi}{\mathcal{P}\mathcal{T}_{\xi_i}A \phi}\\
  &= \inner{\mathcal{P}\mathcal{L}_{\xi_i}\phi}{\mathcal{P}B_i \Delta \phi} + \inner{\mathcal{P}\mathcal{L}_{\xi_i}\phi}{\mathcal{P}\mathcal{T}_{\xi_i}A \phi}\\
  &= \inner{\mathcal{P}\mathcal{L}_{\xi_i}\phi}{\mathcal{L}_{\xi_i} \Delta \phi} + \inner{\mathcal{P}\mathcal{L}_{\xi_i}\phi}{\mathcal{T}_{\xi_i} \Delta \phi} + \inner{\mathcal{P}\mathcal{L}_{\xi_i}\phi}{\mathcal{T}_{\xi_i}A \phi}.
\end{align*}
Therefore the entirety of (\ref{now there is a number}) is $$-\inner{\mathcal{P}\mathcal{L}_{\xi_i}\phi}{[\Delta,\mathcal{L}_{\xi_i}]\phi} +  \inner{\mathcal{P}\mathcal{L}_{\xi_i}\phi}{\mathcal{T}_{\xi_i} \Delta \phi} + \inner{\mathcal{P}\mathcal{L}_{\xi_i}\phi}{\mathcal{T}_{\xi_i}A \phi}$$
or more compactly $$ \inner{\mathcal{P}\mathcal{L}_{\xi_i}\phi}{\left(\mathcal{T}_{\xi_i} \Delta + \mathcal{T}_{\xi_i}A - [\Delta,\mathcal{L}_{\xi_i}] \right)\phi}.$$
This is now easy to control as the operator in the brackets is of second order, so the assumption (\ref{assumpty 12}) is indeed verified (we note that one can apply Lemma \ref{greens for navier} in the same way using the $W^{3,2}_0(\mathscr{O};\R^2)$ regularity of the $\xi_i$ directly). The proof for (\ref{assumpty13}) sees no change to the one for $B_i$. In summary, a purely transport Stratonovich noise does satisfy the assumptions, but to prove it one has to revert to the SALT noise hence we emphasise how significant the structure of the SALT noise is for the analysis. 

\begin{itemize}
    \item For which noise does the result of Theorem \ref{theorem for limit of navier stokes is euler} still apply?
\end{itemize}

The choice to state and prove Theorem \ref{theorem for limit of navier stokes is euler} for SALT noise specifically, instead of a general noise as in Theorem \ref{theorem 2D strong}, comes down to a few reasons. Firstly it is apparent that the case of an It\^{o} transport noise where the gradient dependency is small relative to viscosity will not work (as viscosity is approaching zero), so if we wanted to keep this interesting case in our solution theory for Theorem \ref{theorem 2D strong} then a distinct set of assumptions would be required. On this note, it is also clear from our approach with the vorticity form that precise assumptions would be needed on $\textnormal{curl}\mathcal{Q}_i$ which would be somewhat cumbersome. Moreover, as discussed in the previous point, we want to highlight the power of the SALT noise; we see this in Proposition \ref{how not label}, as we can achieve pathwise bounds given the complete cancellation in the stochastic integral. This does not exclude the result for additive and Nemystkii type noise though: we instead can only demonstrate Proposition \ref{how not label} as an inequality in expectation, however the tightness and hence passage to the Stochastic Euler Equation is still valid. Additional technicalities are needed to introduce a pathwise bound into the system, which is done through the familiar first hitting times. Proposition \ref{tightness in D prop} is then shown up until these first hitting times, which still gives the tightness result exactly as seen in [\cite{goodair2023zero}] Proposition 3.3. One would also have to take expectations in the proof of Proposition \ref{not another one!!!!}, but the result follows with only these  minor technical adjustments. The scenario of a transport noise is much more delicate, as the correct It\^{o}-Stratonovich corrector in the weak formulation of the Euler Equation is $-\inner{\mathcal{P}\mathcal{L}_{\xi_i}u_s}{\mathcal{L}_{\xi_i}u_s}$ where this different structure retaining the Leray Projector (as opposed to for $B_i$) requires some different considerations which we do not detail here. The idea, however, is to again introduce $\mathcal{T}_{\xi_i}$ and use the good commutative properties of $\mathcal{P}$ and $B_i$. 

\begin{itemize}
    \item Could we have uniqueness for the SALT Euler Equation?
\end{itemize}

The question of uniqueness of weak solutions of the Euler Equation is certainly of interest, where both uniqueness and non-uniqueness have been established dependent on the regularity imposed. The classical (positive) uniqueness result is for initial vorticities in $L^\infty\left(\mathscr{O};\R\right)$ proven by Yudovich in [\cite{yudovich1963non}], which he then extends to a wider class in [\cite{yudovich1995uniqueness}]. Simply talking about the $L^\infty\left(\mathscr{O};\R\right)$ case here, we note that on the torus this result was proven in [\cite{lang2023well}]. Success hinges on showing an energy estimate of the vorticity in $L^\infty\left([0,T];L^\infty(\mathscr{O};\R)\right)$, and in fact it was proven in [\cite{lang2023well}] (22) that the $L^{\infty}(\mathscr{O};\R)$ norm of the vorticity is constant in time via use of the It\^{o}-Wentzel formula on the flow map. This was done for a sequence of strong solutions with more regular initial conditions approximating the desired $w_0 \in L^\infty(\mathscr{O};\R)$, which is an approach we couldn't take here given that only weak solutions have been established. Any ideas for us would rely on showing additional regularity at the level of the Galerkin Approximations.

\begin{itemize}
    \item Do the strong solutions of (\ref{projected Ito}) converge to the corresponding no-slip weak solutions shown to exist in [\cite{goodair2023zero}]?
\end{itemize}

As discussed in the introduction, the idea is that as $\alpha$ approaches infinity in (\ref{navier boundary conditions}) then the second term dominates the identity giving that $u \cdot \iota = 0$ on $\partial \mathscr{O}$, hence $u = 0$ on $\partial \mathscr{O}$. Existence and uniqueness of weak solutions of (\ref{projected Ito}) under this boundary condition was proven in [\cite{goodair2023zero}]. In the deterministic case this result is proven in [\cite{kelliher2006navier}] Theorem 9.2, though we appreciate that it is in fact done for the strong solution with no-slip condition and this a priori additional regularity is necessary in the approach. In lieu of this regularity, tightness arguments on the sequence of strong solutions of (\ref{projected Ito}) with $\alpha$ tending to infinity have proven fruitless.\\

In fact this method would be successful if we could show that the solutions have $L^{\infty}\left([0,T];\bar{W}^{1,2}_{\sigma}\right) \cap L^2\left([0,T];W^{2,2}(\mathscr{O};\R^2)\right)$ regularity uniformly in $\alpha$. Such a result is exceptionally appealing as with the limit inheriting this regularity, we would prove that \textit{strong solutions} of the equation exist for the no-slip condition. The difficulty in this problem has already been explicated, so naturally this approach is of great interest. This regularity would enter at the level of the Galerkin Approximation, in particular Proposition \ref{uniform boundedness proposition}. Scanning through where the dependency on $\alpha$ lies in this result, the first instance is in the $\norm{u^n_r}_H^2$ term, though as $\norm{u^n_r}_1^2 \leq \norm{u^n_r}_H^2$ then this is irrelevant. Importantly is the initial condition $\norm{u^n_0}_H^2$, where the boundary integral $\inner{(\kappa - \alpha)u^n_0}{u^n_0}_{L^2(\partial \mathscr{O};\R^2)}$ is problematic. This is surely to be expected as strong solutions for the no-slip condition will only exist for $u_0 \in W^{1,2}_{0}(\mathscr{O};\R^2)$. Indeed if we impose this condition, then we can use that $\norm{u^n_0}_H^2 \leq \norm{u_0}_H^2$ and then the boundary integral $\inner{(\kappa - \alpha)u_0}{u_0}_{L^2(\partial \mathscr{O};\R^2)}$ is null. We see that there is no $\alpha$ dependency in treating the nonlinear term, and in the Stokes Operator it is only in the equivalence of the $\norm{\cdot}_2$ and $\norm{\cdot}_{W^{2,2}}$ norms which does require some nuance but can be dealt with. Inevitably it is the transport type noise where this unravels, and we see the dependency on $\alpha$ in the term $$-\inner{(\kappa - \alpha)\mathcal{P}\mathcal{G}_i\phi^n}{\mathcal{P}\mathcal{G}_i \phi^n}_{L^2(\partial \mathscr{O};\R^2)} $$ which appears to be impossible to handle. We bring attention to this as it is again the issue of the Leray Projector not preserving the zero trace property which prevents progress.

\begin{itemize}
    \item What are the overarching conclusions of the paper? 
\end{itemize}

We hope, therefore, to have made two important points regarding the analysis of transport type stochastic fluid PDEs on a bounded domain. The first is that the choice of SALT noise satisfies some natural compatibility with the boundary which purely transport Stratonovich noise does not. The involved geometric derivation of the SALT noise does suggest this, and our results give analytical validation to the introduction of $\mathcal{T}_{\xi_i}$ into the transport noise. The second point is that the less typical Navier boundary conditions appear to be the right choice for transport noise, given the unavoidable difficulty in the solution theory for the no-slip case. This is highlighted in both the direct Galerkin approach, and the $\alpha \rightarrow \infty$ limit for solutions with Navier boundary conditions. The no-slip condition inherently demands that a transport type noise be both divergence-free and zero on the boundary, which is out of reach. It appears that these fine properties of transport noise in a bounded domain have not been appreciated, and we hope that this work can promote future studies in the area.\\

\textbf{Thanks:} I would like to give my sincerest thanks to Dan Crisan for the regular and extended discussions around the paper, his feedback on it, and overall guidance during this process.\\

\textbf{Acknowledgements:} The author was supported by the Engineering and Physical Sciences Research Council (EPSCR) Project 2478902.


\section{Appendix} \label{section appendix}

\subsection{Cauchy Lemma} \label{sub cauchy}

We state and prove the critical result used in Proposition \ref{the one that gives limiting stopping time bound}.

\begin{lemma} \label{amazing cauchy lemma}
    Fix $T>0$. For $t\in[0,T]$ let $X_t$ denote a Banach Space with norm $\norm{\cdot}_{X,t}$ such that for all $s > t$, $X_s \xhookrightarrow{}X_t$ and $\norm{\cdot}_{X,t} \leq \norm{\cdot}_{X,s}$. Suppose that $(\sy^n)$ is a sequence of processes $\sy^n:\Omega \mapsto X_T$, $\norm{\sy^n}_{X,\cdot}$ is adapted and $\mathbb{P}-a.s.$ continuous, $\sy^n \in L^2\left(\Omega;X_T\right)$, and such that $\sup_{n}\norm{\sy^n}_{X,0} \in L^\infty\left(\Omega;\R\right)$. For any given $M >1$ define the stopping times
    \begin{equation} \label{another taumt}
        \tau^{M,T}_n := T \wedge \inf\left\{s \geq 0: \norm{\sy^n}_{X,s}^2 \geq M + \norm{\sy^n}_{X,0}^2 \right\}.
    \end{equation}
Furthermore suppose \begin{equation} \label{supposition 1}
    \lim_{m \rightarrow \infty}\sup_{n \geq m}\mathbbm{E}\left[\norm{\sy^n-\sy^m}^2_{X,\tau
_{m}^{M,t}\wedge \tau_{n}^{M,t}} \right] =0
\end{equation}
and that for any stopping time $\gamma$ and sequence of stopping times $(\delta_j)$ which converge to $0$ $\mathbb{P}-a.s.$, \begin{equation} \label{supposition 2} \lim_{j \rightarrow \infty}\sup_{n\in\N}\mathbbm{E}\left(\norm{\sy^n}_{X,(\gamma + \delta_j) \wedge \tau^{M,T}_n}^2 - \norm{\sy^n}_{X,\gamma \wedge \tau^{M,T}_n}^2 \right) =0.
\end{equation}
Then there exists a stopping time $\tau^{M,T}_{\infty}$, a process $\sy:\Omega \mapsto X_{\tau^{M,T}_{\infty}}$ whereby $\norm{\sy}_{X,\cdot \wedge \tau^{M,T}_{\infty}}$ is adapted and $\mathbb{P}-a.s.$ continuous, and a subsequence indexed by $(m_j)$ such that 
\begin{itemize}
    \item $\tau^{M,T}_{\infty} \leq \tau^{M,T}_{m_j}$ $\mathbb{P}-a.s.$,
    \item $\lim_{j \rightarrow \infty}\norm{\sy - \sy^{m_j}}_{X,\tau^{M,T}_{\infty}} = 0$ $\mathbb{P}-a.s.$.
\end{itemize}
Moreover for any $R>0$ we can choose $M$ to be such that the stopping time \begin{equation} \label{another tauR}
        \tau^{R,T} := T \wedge \inf\left\{s \geq 0: \norm{\sy}_{X,s\wedge\tau^{M,T}_{\infty}}^2 \geq R \right\}
    \end{equation}
satisfies $\tau^{R,T} \leq \tau^{M,T}_{\infty}$ $\mathbb{P}-a.s.$. Thus $\tau^{R,T}$ is simply $T \wedge \inf\left\{s \geq 0: \norm{\sy}_{X,s}^2 \geq R \right\}$.

\end{lemma}

    \begin{proof}
   Property (\ref{supposition 1}) implies that for any given $j \in \N$ we can choose an $n_j\in \N$ such that for all $k \geq n_j$, \begin{equation}\label{a good property}\mathbbm{E}\left(\norm{\sy^k - \sy^{n_j}}^2_{X,\tau
_{n_{j}}^{M,t}\wedge \tau_{k}^{M,t}} \right) \leq 2^{-4j}.\end{equation}
We shall make use of highly sensitive manipulations of the subsequence indexed by $(n_j)$, and for this we introduce a new sequence of stopping times. We now impose that $$M > 2 + \left\Vert\sup_{n \in \N}\norm{\sy^n}_{X,0}^2\right\Vert_{L^\infty(\Omega;\R)}$$ and define $$\tilde{M}^2:= \frac{M -\left\Vert\sup_{n }\norm{\sy^n}_{X,0}^2\right\Vert_{L^\infty(\Omega;\R)}}{2}.$$
The purpose of this is to define $$\sigma^M_j:= T \wedge \inf\left\{s > 0: \norm{\sy^{n_{j}}}_{X,s}\geq (\tilde{M} - 1 +2^{-j}) + \norm{\sy^{n_{j}}}_{X,0} \right\}$$
and ensure that $\sigma^M_j \leq \tau^{M,T}_{n_j}$ at every $\omega$. Note the key difference in not squaring the norm, and also that $\tilde{M}>1$ so each $\sigma^M_j$ is necessarily positive. To demonstrate the inequality, it is sufficient to show that, $\mathbb{P}-a.s.$, \begin{equation} \label{or more easily}
    \left((\tilde{M} - 1 +2^{-j}) + \norm{\sy^{n_{j}}}_{X,0}\right)^2 \leq M +  \norm{\sy^{n_{j}}}_{X,0}^2
\end{equation}
or more easily 
$$  \left(\tilde{M} + \norm{\sy^{n_{j}}}_{X,0}\right)^2 \leq M +  \norm{\sy^{n_{j}}}_{X,0}^2.$$
This is possible as 
\begin{align*}
    \left(\tilde{M} + \norm{\sy^{n_{j}}}_{X,0}\right)^2 \leq 2\tilde{M}^2 + 2\norm{\sy^{n_{j}}}_{X,0}^2 &= M -\left\Vert\sup_{n }\norm{\sy^n}_{X,0}^2\right\Vert_{L^\infty(\Omega;\R)} + 2\norm{\sy^{n_{j}}}_{X,0}^2\\ &\leq M + \norm{\sy^{n_{j}}}_{X,0}^2.
\end{align*}
The property (\ref{or more easily}) is thus verified, so $\sigma^M_j \leq \tau^{M,T}_{n_j}$ and hence the subsequence $(\sy^{n_j})$ enjoys the same properties up until the corresponding $\sigma^M_j$. In particular from (\ref{a good property}), \begin{equation}\label{a good property 2}
    \mathbbm{E}\left(\norm{\sy^{n_{j+1}} - \sy^{n_j}}_{X,\sigma^M
_j\wedge \sigma_{j+1}^{M}} \right) \leq \left[\mathbbm{E}\left(\norm{\sy^{n_{j+1}} - \sy^{n_j}}_{X,\sigma^M
_j\wedge \sigma_{j+1}^{M}}^2 \right)\right]^{\frac{1}{2}}\leq 2^{-2j}
\end{equation}
hence in defining the sets
\begin{equation}\label{defined omega j}\Omega_j:=\left\{\omega \in \Omega: \norm{\sy^{n_{j+1}}(\omega)-\sy^{n_j}(\omega)}_{X,\sigma^{M}_{j}(\omega)\wedge \sigma^{M}_{j+1}(\omega)} < 2^{-(j+2)} \right\}\end{equation}
we have, by Chebyshev's Inequality and (\ref{a good property 2}),
\begin{align*}
    \mathbb{P}\left(\Omega_j^C \right) \leq 2^{j+2}\mathbbm{E}\left(\norm{\sy^{n_{j+1}}-\sy^{n_j}}_{X,\sigma^{M}_{j}\wedge \sigma^{M}_{j+1}} \right) \leq 2^{-j + 2}.
\end{align*}
We have, therefore, that $$\sum_{j=1}^\infty \mathbb{P}\left( \Omega_j^C   \right) < \infty$$ from which we see $$\mathbb{P} \left( \bigcap_{K=1}^\infty \bigcup_{j=K}^\infty \Omega_j^C \right) = 0$$ courtesy of Borel-Cantelli. It then follows that the set $$\hat{\Omega} := \bigcup_{K=1}^\infty \bigcap_{j=K}^\infty \Omega_j$$ is such that $\mathbb{P}( \hat{\Omega}) = 1$ so that in verifying $\mathbb{P}-a.e.$ properties, we can in fact simply show that they hold everywhere on $\hat{\Omega}$. More precisely, we also take $\hat{\Omega}$ to be such that every $\norm{\sy^{n_j}}_{X,\cdot}$ is continuous on $\hat{\Omega}$, which is only a further countable intersection of full measure sets. We proceed by considering the sets $$\hat{\Omega}_K:= \bigcap_{j=K}^\infty \Omega_j$$ with the idea to just show such properties on $\hat{\Omega}_K$ for all $K$ (as their union makes up $\hat{\Omega}$). We look to construct a new stopping time $\sigma^M_{\infty}$ (which will prove to be the desired $\tau^{M,T}_{\infty}$) given as the $\mathbb{P}-a.e.$ limit of $(\sigma^M_j)$, built from demonstrating that $(\sigma^M_j)$ is monotone decreasing everywhere on $\hat{\Omega}_K$ for all $K$. In other words we show that for sufficiently large $j$ (in fact, just $j\geq K$) that the set \begin{equation} \label{def of set} \{ \sigma^M_j < \sigma^M_{j+1} \} \cap \hat{\Omega}_K\end{equation} is empty. Firstly we observe from the strict inequality $\sigma^M_j < \sigma^M_{j+1}$ on this set that $\sigma^M_j < T$, implying that 
$$\sigma^M_j= \inf\left\{s > 0: \norm{\sy^{n_{j}}}_{X,s}\geq (\tilde{M} - 1 +2^{-j}) + \norm{\sy^{n_{j}}}_{X,0} \right\}$$
so by the continuity of $\norm{\sy^{n_j}}_{X,\cdot}$, \begin{equation}\label{by the continuity}\norm{\sy^{n_{j}}}_{X,\sigma^M_j} =(\tilde{M} - 1 +2^{-j}) + \norm{\sy^{n_{j}}}_{X,0}. \end{equation} Using the definition of $\Omega_j$, (\ref{defined omega j}), for $j \geq K$, we have that \begin{equation}\label{it is useful}\norm{\sy^{n_j}}_{X,\sigma^M_j \wedge \sigma^M_{j+1}} - \norm{\sy^{n_{j+1}}}_{X,\sigma^M_j \wedge \sigma^M_{j+1}} \leq \norm{\sy^{n_{j+1}} - \sy^{n_j} }_{X,\sigma^M_j \wedge \sigma^M_{j+1}} < 2^{-\left(j+2\right)}\end{equation}
and also
\begin{equation} \label{useful ic}
    \norm{\sy^{n_{j+1}}}_{X,0} - \norm{\sy^{n_j}}_{X,0}  \leq \norm{\sy^{n_{j+1}} - \sy^{n_j} }_{X,0} < 2^{-\left(j+2\right)}.
\end{equation}
Combining (\ref{by the continuity}), (\ref{it is useful}) and (\ref{useful ic}), whilst using that $\sigma^M_j < \sigma^M_{j+1}$, we see that \begin{align}
    \nonumber \norm{\sy^{n_{j+1}}}_{X,\sigma^M_j \wedge \sigma^M_{j+1}} &> \norm{\sy^{n_j}}_{X,\sigma^M_j \wedge \sigma^M_{j+1}} - 2^{-\left(j+2\right)}\\ \nonumber
    &= \norm{\sy^{n_j}}_{X,\sigma^M_j} - 2^{-\left(j+2\right)}\\ \nonumber
    &=  \norm{\sy^{n_{j}}}_{X,0} + (\tilde{M} - 1 +2^{-j}) - 2^{-\left(j+2\right)}\\  \nonumber
    &> \norm{\sy^{n_{j+1}}}_{X,0} - 2^{-(j+2)} + (\tilde{M} - 1 +2^{-j}) - 2^{-\left(j+2\right)}\\
    &= \norm{\sy^{n_{j+1}}}_{X,0} + (\tilde{M} - 1 +2^{-(j+1)}) \label{end of the align}
\end{align}
where in the last line we have used the manipulation
$$ - 2^{-(j+2)} - 2^{-(j+2)} + 2^{-j} = -2^{-j-1} + 2^{-j} = 2^{-j}\left( -2^{-1} + 1\right) = 2^{-j}\left( 2^{-1} \right) = 2^{-(j+1)}.$$ The hard work is done in showing that the set (\ref{def of set}) is empty, as on this set note that $$\norm{\sy^{n_{j+1}}}_{X,\sigma^M_j \wedge \sigma^M_{j+1}} \leq \norm{\sy^{n_{j+1}}}_{X, \sigma^M_{j+1}} \leq \norm{\sy^{n_{j+1}}}_{X,0} + (\tilde{M} - 1 +2^{-(j+1)})$$ which contradicts (\ref{end of the align}), hence (\ref{def of set}) must be empty. Thus on every $\hat{\Omega}_K$, and furthermore the whole of $\hat{\Omega}$, the sequence $(\sigma^M_j)$ is eventually monotone decreasing (and bounded below by $0$). Furthermore we define $\sigma^M_{\infty}$ as the pointwise limit $\lim_{j \rightarrow \infty}\sigma^M_j$ on $\hat{\Omega}$, which must itself be a stopping time as the $\mathbb{P}-a.s.$ limit of stopping times. As mentioned this shall prove to be our $\tau^{M,T}_\infty$, and for the existence of $\sy$ we show that on $\hat{\Omega}$ the subsequence $(\sy^{n_j})$ is Cauchy in $X_{\sigma^M_{\infty}}$. Every $\omega \in \hat{\Omega}$ belongs to $\hat{\Omega}_K$ for some $K$, and furthermore to $\hat{\Omega}_L$ for all $L > K$. We fix arbitrary $\omega$ and select an associated $K$. At this $\omega$, for any $j > k \geq K$, observe that
\begin{align*}
    \norm{\sy^{n_j} - \sy^{n_k}}_{X, \sigma^M_{\infty}} &= \norm{\sy^{n_{j}} -\sy^{n_{k+1}} + \sy^{n_{k+1}} - \sy^{n_k}}_{X, \sigma^M_{\infty}}\\ &\leq \norm{\sy^{n_{j}} -\sy^{n_{k+1}}}_{X, \sigma^M_{\infty}} + \norm{\sy^{n_{k+1}} - \sy^{n_k}}_{X, \sigma^M_{\infty}}\\ &\leq \norm{\sy^{n_{j}} -\sy^{n_{k+1}}}_{X, \sigma^M_{\infty}} + 2^{-(k+2)}\\& \leq \sum_{l=k}^{j}2^{-(l+2)}\\
    &\leq 2^{-(k+1)}
\end{align*}
having carried out an inductive argument in the penultimate step. We are thus free to take $K$ large enough so that this difference is arbitrarily small; therefore there exists a limit in the Banach Space $X_{\sigma^M_{\infty}}$, which we call $\sy$. The process $\norm{\sy}_{X,\cdot \wedge \sigma^M_{\infty}}$ is adapted and $\mathbb{P}-a.s.$ continuous, as $$\sup_{r \in [0,T]}\left\vert \norm{\sy}_{X,r \wedge \sigma^M_{\infty}} - \norm{\sy^{n_j}}_{X,r \wedge \sigma^M_{\infty}}\right\vert \leq \sup_{r \in [0,T]}\left\vert \norm{\sy - \sy^{n_j}}_{X,r \wedge \sigma^M_{\infty}}\right\vert = \norm{\sy - \sy^{n_j}}_{X,\sigma^M_{\infty}}$$ 
which has $\mathbb{P}-a.s.$ limit as $j \rightarrow \infty$ equal to zero. Thus $\norm{\sy}_{X,\cdot \wedge \sigma^M_{\infty}}$ is given, $\mathbb{P}-a.s.$, as the uniform in time limit of adapted and continuous processes, verifying the result. Moving on, it is now that we make use of (\ref{supposition 2}) much in the same way as we did for (\ref{supposition 1}). This will be done in the context of $\gamma:= \sigma^M_{\infty}$ and $\delta_j:= \sigma^M_j - \sigma^M_{\infty}$. Indeed for any $j \in \N$ we can choose an $m_j \in \N$ (where $m_j=n_l$ some $l$) such that for all $k \geq m_j$, $$\sup_{n\in\N}\mathbbm{E}\left(\norm{\sy^n}_{X,\sigma^M_k \wedge \tau^{M,T}_{n}}^2 - \norm{\sy^n}_{X,\sigma^M_{\infty} \wedge \tau^{M,T}_n}^2\right) \leq 2^{-2j}.$$ In particular, through a relabelling of $\sigma^M_{m_j}=\sigma^M_l$, $$\mathbbm{E}\left(\norm{\sy^{m_j}}_{X,\sigma^M_{m_j}}^2 - \norm{\sy^{m_j}}_{X,\sigma^M_{\infty}}^2\right) \leq 2^{-2j}$$
by choosing $n$ as $m_j$ and using that $\sigma^M_{\infty} \leq \sigma^M_k \leq \sigma^M_{m_j} \leq \tau^{M,T}_{m_j}$. In a familiar way we define $$\Omega'_j:= \left\{\norm{\sy^{m_j}}_{X,\sigma^M_{m_j}}^2 - \norm{\sy^{m_j}}_{X,\sigma^M_{\infty}}^2 < 2^{-(j+2)} \right\} $$
so that, just as we showed for (\ref{defined omega j}), $$\check{\Omega}_K:=\bigcap_{j=K}^\infty \Omega'_j, \qquad \check{\Omega}:=\bigcup_{K=1}^\infty \check{\Omega}_K, \qquad  \mathbb{P}\left(\check{\Omega}\right)=1.$$ For arbitrary given $R>0$, the plan now is to find a constant $M$ such that at every $\omega \in \hat{\Omega}\cap\check{\Omega}$, either $\sigma^M_{\infty} = T$ or $\norm{\sy}^2_{X,\sigma^M_{\infty}} \geq R$. In both instances it is clear that $\tau^{R,T} \leq \sigma^M_{\infty}$, thus proving the proposition. To this end we fix an $\omega \in \hat{\Omega}\cap\check{\Omega}$ such that $\sigma^M_{\infty} < T$. As $\sigma^M_{\infty}$ is the decreasing limit of $(\sigma^M_{m_j})$ then for sufficiently large $m_j$ we must also have that $\sigma^M_{m_j} < T$. Exactly as in (\ref{by the continuity}), \begin{equation}\label{newcontinuitything} \norm{\sy^{m_{j}}}_{X,\sigma^M_{m_j}} =(\tilde{M} - 1 +2^{-j}) + \norm{\sy^{m_{j}}}_{X,0}. \end{equation}
From the proven convergence we also have that for sufficiently large $m_j$, \begin{equation}\label{applynewestcauchy}
    \norm{\sy - \sy^{m_j}}_{X,\sigma^M_{\infty}} < 1
\end{equation}
which implies that $\norm{\sy}_{X,\sigma^M_{\infty}} > \norm{\sy^{m_j}}_{X,\sigma^M_{\infty}} -1$, and likewise as $\omega \in \check{\Omega}_K$ for some $K$, \begin{equation} \label{lkjh}
    \norm{\sy^{m_j}}_{X,\sigma^M_{m_j}}^2 - \norm{\sy^{m_j}}_{X,\sigma^M_{\infty}}^2 < 1.
\end{equation}
We fix an $m_j$ large enough so that (\ref{newcontinuitything}), (\ref{applynewestcauchy}) and (\ref{lkjh}) all hold. Substituting (\ref{newcontinuitything}) into (\ref{lkjh}) gives that $$\norm{\sy^{m_j}}_{X,\sigma^M_{\infty}}^2 > \left((\tilde{M} - 1 +2^{-j}) + \norm{\sy^{m_{j}}}_{X,0}\right)^2 -1 > (\tilde{M}-1)^2 -1.$$ If $\tilde{M} > 2$ then the expression on the right is positive and $$ \norm{\sy^{m_j}}_{X,\sigma^M_{\infty}} > \left((\tilde{M}-1)^2 -1 \right)^{\frac{1}{2}}.$$
Furthermore $$ \norm{\sy}_{X,\sigma^M_{\infty}} > \left((\tilde{M}-1)^2 -1 \right)^{\frac{1}{2}} -1$$ where the right hand side is of course monotone increasing and unbounded in $\tilde{M}$ and hence $M$. By choosing $M$ large enough such that $$ \left[\left((\tilde{M}-1)^2 -1 \right)^{\frac{1}{2}} -1\right]^2 > R$$ we complete the proof.
\end{proof}

\subsection{Abstract Framework for Solutions} \label{Appendix III}

We provide the framework of [\cite{goodair2022existence1}] Section 3 used in Subsection \ref{sub strong sols}, considering an It\^{o} SPDE 
\begin{equation} \label{thespde}
    \sy_t = \sy_0 + \int_0^t \mathcal{A}(s,\sy_s)ds + \int_0^t\tilde{\mathcal{G}} (s,\sy_s) d\mathcal{W}_s.
\end{equation}
We state the assumptions for a triplet of embedded Hilbert Spaces $$V \hookrightarrow H \hookrightarrow U$$ and ask that there is a continuous bilinear form $\inner{\cdot}{\cdot}_{U \times V}: U \times V \rightarrow \R$ such that for $\phi \in H$ and $\psi \in V$, \begin{equation} \label{bilinear formog}
    \inner{\phi}{\psi}_{U \times V} =  \inner{\phi}{\psi}_{H}.
\end{equation}
The mappings $\mathcal{A},\tilde{\mathcal{G}}$ are such that
    $\mathcal{A}:[0,T] \times V \rightarrow U,
    \tilde{\mathcal{G}}:[0,T] \times V \rightarrow \mathscr{L}^2(\mathfrak{U};H)$ are measurable. We assume that $V$ is dense in $H$ which is dense in $U$. 

\begin{assumption} \label{assumption fin dim spaces}
There exists a system $(a_n)$ of elements of $V$ such that, defining the spaces $V_n:= \textnormal{span}\left\{a_1, \dots, a_n \right\}$ and $\mathcal{P}_n$ as the orthogonal projection to $V_n$ in $U$, then:
\begin{enumerate}
    \item There exists some constant $c$ independent of $n$ such that for all $\phi\in H$,
\begin{equation} \label{projectionsboundedonH}
    \norm{\mathcal{P}_n \phi}_H^2 \leq c\norm{\phi}_H^2.
\end{equation}
\item There exists a real valued sequence $(\mu_n)$ with $\mu_n \rightarrow \infty$ such that for any $\phi \in H$, \begin{align}
     \label{mu2}
    \norm{(I - \mathcal{P}_n)\phi}_U \leq \frac{1}{\mu_n}\norm{\phi}_H
\end{align}
where $I$ represents the identity operator in $U$.
\end{enumerate}
\end{assumption}

 These conditions are supplemented by a series of assumptions on the mappings. We shall use general notation $c_t$ to represent a function $c_\cdot:[0,\infty) \rightarrow \R$ bounded on $[0,T]$, evaluated at the time $t$. Moreover we define functions $K$, $\tilde{K}$ relative to some non-negative constants $p,\tilde{p},q,\tilde{q}$. We use a generic notation to define the functions $K: U \rightarrow \R$, $K: U \times U \rightarrow \R$, $\tilde{K}: H \rightarrow \R$ and $\tilde{K}: H \times H \rightarrow \R$ by
\begin{align*}
    K(\phi)&:= 1 + \norm{\phi}_U^{p}, \qquad
    K(\phi,\psi):= 1+\norm{\phi}_U^{p} + \norm{\psi}_U^{q},\\
    \tilde{K}(\phi) &:= K(\phi) + \norm{\phi}_H^{\tilde{p}}, \qquad
    \tilde{K}(\phi,\psi) := K(\phi,\psi) + \norm{\phi}_H^{\tilde{p}} + \norm{\psi}_H^{\tilde{q}}
\end{align*}
 Distinct use of the function $K$ will depend on different constants but in no meaningful way in our applications, hence no explicit reference to them shall be made. In the case of $\tilde{K}$, when $\tilde{p}, \tilde{q} = 2$ then we shall denote the general $\tilde{K}$ by $\tilde{K}_2$. In this case no further assumptions are made on the $p,q$. That is, $\tilde{K}_2$ has the general representation \begin{equation}\label{Ktilde2}\tilde{K}_2(\phi,\psi) = K(\phi,\psi) + \norm{\phi}_H^2 + \norm{\psi}_H^2\end{equation} and similarly as a function of one variable.\\
 
 We state the subsequent assumptions for arbitrary elements $\phi,\psi \in V$, $\phi^n \in V_n$, $\eta \in H$ and $t \in [0,T]$, and a fixed $\gamma > 0$. Understanding $\tilde{\mathcal{G}}$ as a mapping $\tilde{\mathcal{G}}: [0,\infty) \times V \times \mathfrak{U} \rightarrow H$, we introduce the notation $\tilde{\mathcal{G}}_i(\cdot,\cdot):= \tilde{\mathcal{G}}(\cdot,\cdot,e_i)$.
 
 
  \begin{assumption} \label{new assumption 1} \begin{align}
     \label{111} \norm{\mathcal{A}(t,\boldsymbol{\phi})}^2_U +\sum_{i=1}^\infty \norm{\tilde{\mathcal{G}}_i(t,\boldsymbol{\phi})}^2_H &\leq c_t K(\boldsymbol{\phi})\left[1 + \norm{\boldsymbol{\phi}}_V^2\right],\\ \label{222}
     \norm{\mathcal{A}(t,\boldsymbol{\phi}) - \mathcal{A}(t,\boldsymbol{\psi})}_U^2 &\leq  c_t\left[K(\phi,\psi) + \norm{\phi}_V^2 + \norm{\psi}_V^2\right]\norm{\phi-\psi}_V^2,\\ \label{333}
    \sum_{i=1}^\infty \norm{\tilde{\mathcal{G}}_i(t,\boldsymbol{\phi}) - \tilde{\mathcal{G}}_i(t,\boldsymbol{\psi})}_U^2 &\leq c_tK(\phi,\psi)\norm{\phi-\psi}_H^2.
 \end{align}
 \end{assumption}

\begin{assumption} \label{assumptions for uniform bounds2}
 \begin{align}
   \label{uniformboundsassumpt1}  2\inner{\mathcal{P}_n\mathcal{A}(t,\boldsymbol{\phi}^n)}{\boldsymbol{\phi}^n}_H + \sum_{i=1}^\infty\norm{\mathcal{P}_n\tilde{\mathcal{G}}_i(t,\boldsymbol{\phi}^n)}_H^2 &\leq c_t\tilde{K}_2(\boldsymbol{\phi}^n)\left[1 + \norm{\boldsymbol{\phi}^n}_H^2\right] - \gamma\norm{\boldsymbol{\phi}^n}_V^2,\\  \label{uniformboundsassumpt2}
    \sum_{i=1}^\infty \inner{\mathcal{P}_n\tilde{\mathcal{G}}_i(t,\boldsymbol{\phi}^n)}{\boldsymbol{\phi}^n}^2_H &\leq c_t\tilde{K}_2(\boldsymbol{\phi}^n)\left[1 + \norm{\boldsymbol{\phi}^n}_H^4\right].
\end{align}
\end{assumption}

\begin{assumption} \label{therealcauchy assumptions}
\begin{align}
  \nonumber 2\inner{\mathcal{A}(t,\boldsymbol{\phi}) - \mathcal{A}(t,\boldsymbol{\psi})}{\boldsymbol{\phi} - \boldsymbol{\psi}}_U &+ \sum_{i=1}^\infty\norm{\tilde{\mathcal{G}}_i(t,\boldsymbol{\phi}) - \tilde{\mathcal{G}}_i(t,\boldsymbol{\psi})}_U^2\\ \label{therealcauchy1} &\leq  c_{t}\tilde{K}_2(\boldsymbol{\phi},\boldsymbol{\psi}) \norm{\boldsymbol{\phi}-\boldsymbol{\psi}}_U^2 - \gamma\norm{\boldsymbol{\phi}-\boldsymbol{\psi}}_H^2,\\ \label{therealcauchy2}
    \sum_{i=1}^\infty \inner{\tilde{\mathcal{G}}_i(t,\boldsymbol{\phi}) - \tilde{\mathcal{G}}_i(t,\boldsymbol{\psi})}{\boldsymbol{\phi}-\boldsymbol{\psi}}^2_U & \leq c_{t} \tilde{K}_2(\boldsymbol{\phi},\boldsymbol{\psi}) \norm{\boldsymbol{\phi}-\boldsymbol{\psi}}_U^4.
\end{align}
\end{assumption}

\begin{assumption} \label{assumption for prob in V}
\begin{align}
   \label{probability first} 2\inner{\mathcal{A}(t,\boldsymbol{\phi})}{\boldsymbol{\phi}}_U + \sum_{i=1}^\infty\norm{\tilde{\mathcal{G}}_i(t,\boldsymbol{\phi})}_U^2 &\leq c_tK(\boldsymbol{\phi})\left[1 +  \norm{\boldsymbol{\phi}}_H^2\right],\\\label{probability second}
    \sum_{i=1}^\infty \inner{\tilde{\mathcal{G}}_i(t,\boldsymbol{\phi})}{\boldsymbol{\phi}}^2_U &\leq c_tK(\boldsymbol{\phi})\left[1 + \norm{\boldsymbol{\phi}}_H^4\right].
\end{align}
\end{assumption}

\begin{assumption} \label{finally the last assumption}
 \begin{equation} \label{lastlast assumption}
    \inner{\mathcal{A}(t,\phi)-\mathcal{A}(t,\psi)}{\eta}_U \leq c_t(1+\norm{\eta}_H)\left[K(\phi,\psi) + \norm{\phi}_V + \norm{\psi}_V\right]\norm{\phi-\psi}_H.
    \end{equation}
\end{assumption}

\subsection{Useful Results} \label{sub useful results}

\begin{theorem}[Gagliardo-Nirenberg Inequality] \label{gagliardonirenberginequality}
    Let $p,q,\alpha \in \R$, $m \in \N$, $\mathscr{O} \subset \R^N$ be such that $p > q \geq 1$, $m > N(\frac{1}{2} - \frac{1}{p})$ and $\frac{1}{p} = \frac{\alpha}{q} + (1-\alpha)(\frac{1}{2} - \frac{m}{N})$. Then there exists a constant $c$ (dependent on the given parameters) such that for any $f \in L^p(\mathscr{O};\R) \cap W^{m,2}(\mathscr{O};\R)$, we have \begin{equation}\label{gag bounded domain}\norm{f}_{L^p(\mathscr{O};\R)} \leq c\norm{f}^{\alpha}_{L^q(\mathscr{O};\R)}\norm{f}^{1-\alpha}_{W^{m,2}(\mathscr{O};\R)}.\end{equation}
\end{theorem}

\begin{proof}
See [\cite{nirenberg2011elliptic}] pp.125-126. 
\end{proof}

\begin{remark}
    In the original paper [\cite{nirenberg2011elliptic}], the inequality is stated for only the $m^{\textnormal{th}}$ order derivative and with an additional $\norm{f}_{L^r}$ term on the bounded domain, for any $r > 0$. By considering the full $W^{m,2}(\mathscr{O};\R^N)$ norm, one can remove this additional term through interpolation. 
\end{remark}

\begin{proposition} \label{rockner prop}
Let $\mathcal{H}_1 \subset \mathcal{H}_2 \subset \mathcal{H}_3$ be a triplet of embedded Hilbert Spaces where $\mathcal{H}_1$ is dense in $\mathcal{H}_2$, with the property that there exists a continuous nondegenerate bilinear form $\inner{\cdot}{\cdot}_{\mathcal{H}_3 \times \mathcal{H}_1}: \mathcal{H}_3 \times \mathcal{H}_1 \rightarrow \R$ such that for $\phi \in \mathcal{H}_2$ and $\psi \in \mathcal{H}_1$, $$\inner{\phi}{\psi}_{\mathcal{H}_3 \times \mathcal{H}_1} = \inner{\phi}{\psi}_{\mathcal{H}_2}.$$ Suppose that for some $T > 0$ and stopping time $\tau$,
\begin{enumerate}
        \item $\sy_0:\Omega \rightarrow \mathcal{H}_2$ is $\mathcal{F}_0-$measurable;
        \item $\eta:\Omega \times [0,T] \rightarrow \mathcal{H}_3$ is such that for $\mathbbm{P}-a.e.$ $\omega$, $\eta(\omega) \in L^2([0,T];\mathcal{H}_3)$;
        \item $B:\Omega \times [0,T] \rightarrow \mathscr{L}^2(\mathfrak{U};\mathcal{H}_2)$ is progressively measurable and such that for $\mathbbm{P}-a.e.$ $\omega$, $B(\omega) \in L^2\left([0,T];\mathscr{L}^2(\mathfrak{U};\mathcal{H}_2)\right)$;
        \item  \label{4*} $\sy:\Omega \times [0,T] \rightarrow \mathcal{H}_1$ is such that for $\mathbbm{P}-a.e.$ $\omega$, $\sy_{\cdot}(\omega)\mathbbm{1}_{\cdot \leq \tau(\omega)} \in L^2([0,T];\mathcal{H}_1)$ and $\sy_{\cdot}\mathbbm{1}_{\cdot \leq \tau}$ is progressively measurable in $\mathcal{H}_1$;
        \item \label{item 5 again*} The identity
        \begin{equation} \label{newest identity*}
            \sy_t = \sy_0 + \int_0^{t \wedge \tau}\eta_sds + \int_0^{t \wedge \tau}B_s d\mathcal{W}_s
        \end{equation}
        holds $\mathbbm{P}-a.s.$ in $\mathcal{H}_3$ for all $t \in [0,T]$.
    \end{enumerate}
The the equality 
  \begin{align} \label{ito big dog*}\norm{\sy_t}^2_{\mathcal{H}_2} = \norm{\sy_0}^2_{\mathcal{H}_2} + \int_0^{t\wedge \tau} \bigg( 2\inner{\eta_s}{\sy_s}_{\mathcal{H}_3 \times \mathcal{H}_1} + \norm{B_s}^2_{\mathscr{L}^2(\mathfrak{U};\mathcal{H}_2)}\bigg)ds + 2\int_0^{t \wedge \tau}\inner{B_s}{\sy_s}_{\mathcal{H}_2}d\mathcal{W}_s\end{align}
  holds for any $t \in [0,T]$, $\mathbbm{P}-a.s.$ in $\R$. Moreover for $\mathbbm{P}-a.e.$ $\omega$, $\sy_{\cdot}(\omega) \in C([0,T];\mathcal{H}_2)$. 
\end{proposition}

\begin{proof}
This is a minor extension of [\cite{prevot2007concise}] Theorem 4.2.5, which is stated and justified as Proposition 2.5.5. in [\cite{goodair2022stochastic}]. This is necessary for us as it extends the Gelfand Triple setting. 
\end{proof}

\begin{lemma} \label{lemma for D tight}
    Let $\mathcal{H}_1, \mathcal{H}_2$ be separable Hilbert Spaces with $\mathcal{H}_1$ compactly embedded into $\mathcal{H}_2$, and $V$ any dense set in $\mathcal{H}_2$. For some fixed $T>0$ let $\sy^n: \Omega \rightarrow C\left([0,T];\mathcal{H}_1\right)$ be a sequence of measurable processes such that \begin{equation} \label{first condition primed}
        \sup_{n \in \N}\mathbbm{E}\left(\sup_{t\in[0,T]}\norm{\sy^n_t}_{\mathcal{H}_1}\right) < \infty
    \end{equation}
    and for any sequence of stopping times $(\gamma_n)$ with $\gamma_n: \Omega \rightarrow [0,T]$, and any $\varepsilon > 0$, $v \in V$,
    \begin{equation} \label{second condition primed}
        \lim_{\delta \rightarrow 0^+}\sup_{n \in \N}\mathbbm{P}\left(\left\{
    \omega \in \Omega: \left\vert \left\langle \sy^n_{(\gamma_n + \delta) \wedge T} -\sy^n_{\gamma_n }   , v     \right\rangle_{\mathcal{H}_2} \right\vert > \varepsilon \right\}\right)  = 0.
    \end{equation}
    Then the sequence of the laws of $(\sy^n)$ is tight in the space of probability measures over $\mathcal{D}\left([0,T];\mathcal{H}_2\right)$.
\end{lemma}

\begin{proof}
    We essentially combine the tightness criteria of [\cite{jakubowski1986skorokhod}] Theorem 3.1 and [\cite{aldous1978stopping}] Theorem 1, in the specific case outlined here. Note that this is only a tweak on the result [\cite{goodair2023zero}] Lemma 4.2. Firstly in reference to [\cite{jakubowski1986skorokhod}] Theorem 3.1 we may take $E$ to be $\mathcal{H}_2$ and $\mathbbm{F}$ to be the collection of functions defined for each $v\in V$ by $\inner{\cdot}{v}_{\mathcal{H}_2}$, which separates points in $\mathcal{H}_2$ from the density of $V$. We also note that condition $(3.3)$ in [\cite{jakubowski1986skorokhod}] is satisfied for $(\mu_n)$ taken to be the sequence of laws of $(\sy^n)$ over $\mathcal{D}\left([0,T];\mathcal{H}_2\right)$, owing to the property (\ref{first condition primed}). Indeed as $\mathcal{H}_1$ is compactly embedded into $\mathcal{H}_2$ one only needs to take a bounded subset of $\mathcal{H}_1$, hence considering the closed ball of radius $M$ in $\mathcal{H}_1$, $\tilde{B}_M$, we have that 
\begin{align*}
    \mathbbm{P}\left(\left\{\omega \in \Omega: \sy^n(\omega) \notin D\left([0,T] ;\tilde{B}_M\right) \right\} \right) &\leq \mathbbm{P}\left(\left\{\omega \in \Omega: \sy^n(\omega) \notin C\left([0,T] ;\tilde{B}_M\right) \right\} \right)\\ &\leq 
    \mathbbm{P}\left(\left\{\omega \in \Omega: \sup_{t\in[0,T]}\norm{\sy^n_t(\omega)}_{\mathcal{H}_1} > M \right\} \right)\\
    &\leq \frac{1}{M}\mathbbm{E}(\sup_{t\in[0,T]}\norm{\sy^n_t}_{\mathcal{H}_1})\\
    &\leq  \frac{1}{M}\sup_{n \in \N}\mathbbm{E}(\sup_{t\in[0,T]}\norm{\sy^n_t}_{\mathcal{H}_1})
\end{align*}
from which we see an arbitrarily large choice of $M$ will justify (3.3). Therefore by Theorem 3.1 it only remains to show that for every $v \in V$ the sequence of the laws of $\inner{\sy^n}{v}_{\mathcal{H}_2}$ is tight in the space of probability measures over $\mathcal{D}\left([0,T];\R\right)$.  By Theorem 1 of [\cite{aldous1978stopping}] this is satisfied if we can show that for for any sequence of stopping times $(\gamma_n)$, $\gamma_n: \Omega \rightarrow [0,T]$, and constants $(\delta_n)$, $\delta_n \geq 0$ and $\delta_n \rightarrow 0$ as $n \rightarrow \infty$:
    \begin{enumerate}
        \item For every $t \in [0,T]$, the sequence of the laws of $\inner{\sy^n_t}{v}_{\mathcal{H}_2}$ is tight in the space of probability measures over $\R$, \label{item 1}

        \item For every $\varepsilon > 0$, $\lim_{n \rightarrow \infty}\mathbbm{P}\left( \left\{ \omega \in \Omega: \left\vert \left\langle \sy^n_{(\gamma_n + \delta_n) \wedge T} -\sy^n_{\gamma_n }   , v     \right\rangle_{\mathcal{H}_2} \right\vert > \varepsilon \right\} \right) = 0.$ \label{item 2}
    \end{enumerate}
We address each item in turn: as for \ref{item 1}, we are required to show that for every $\varepsilon > 0$ and $t\in[0,T]$, there exists a compact $K_{\varepsilon} \subset \R$ such that for every $n \in \N$, $$\mathbbm{P}\left(\left\{\omega \in \Omega: \inner{\sy^n_t(\omega)}{v}_{\mathcal{H}_2} \notin K_{\varepsilon} \right\} \right) < \varepsilon.$$ To this end define $B_M$ as the closed ball of radius $M$ in $\R$, then
\begin{align*}
    \mathbbm{P}\left(\left\{\omega \in \Omega: \inner{\sy^n_t(\omega)}{v}_{\mathcal{H}_2} \notin B_M \right\} \right) &= \mathbbm{P}\left(\left\{\omega \in \Omega: \left\vert\inner{\sy^n_t(\omega)}{v}_{\mathcal{H}_2}\right\vert > M \right\} \right)\\
    &\leq \frac{1}{M}\mathbbm{E}(\left\vert\inner{\sy^n_t}{v}_{\mathcal{H}_2}\right\vert)\\
    &\leq \frac{\norm{v}_{\mathcal{H}_2}}{M}\sup_{n \in \N}\mathbbm{E}\left(\norm{\sy^n_t}_{\mathcal{H}_2}\right)
\end{align*}
so setting $$M:= \frac{\varepsilon}{2\norm{v}_{\mathcal{H}_2}\sup_{n \in \N}\mathbbm{E}\left(\norm{\sy^n_t}_{\mathcal{H}_2}\right)} $$ justifies item \ref{item 1}. As for \ref{item 2}, note that for each fixed $j \in \N$ we have that $$\left\vert \left\langle \sy^j_{(\gamma_j + \delta_j) \wedge T} -\sy^j_{\gamma_j }   , v     \right\rangle_{\mathcal{H}_2} \right\vert  \leq \sup_{n \in \N}\left\vert \left\langle \sy^n_{(\gamma_n + \delta_j) \wedge T} -\sy^n_{\gamma_n }   , v     \right\rangle_{\mathcal{H}_2} \right\vert $$ so in particular $$\lim_{j \rightarrow \infty}\mathbbm{P}\left(\left\{\left\vert \left\langle \sy^j_{(\gamma_j + \delta_j) \wedge T} -\sy^j_{\gamma_j }   , v     \right\rangle_{\mathcal{H}_2} \right\vert > \varepsilon\right\}\right)  \leq \lim_{j \rightarrow \infty}\sup_{n \in \N}\mathbbm{P}\left(\left\{\left\vert \left\langle \sy^n_{(\gamma_n + \delta_j) \wedge T} -\sy^n_{\gamma_n }   , v     \right\rangle_{\mathcal{H}_2} \right\vert > \varepsilon \right\}\right).$$ As $(\delta_j)$ was an arbitrary sequence of non-negative constants approaching zero, we can generically take $\delta \rightarrow 0^+$ and \ref{item 2} is implied by (\ref{second condition primed}). The proof is complete.

\end{proof}

\bibliographystyle{spmpsci}
\bibliography{myBibBib}

\end{document}